\documentclass[english,3p]{elsarticle}
\usepackage[T1]{fontenc}
\usepackage[utf8]{inputenc}
\usepackage{verbatim}
\usepackage{float}
\usepackage{units}
\usepackage{amsmath}
\usepackage{amsthm}
\usepackage{amssymb}
\usepackage{stmaryrd}
\usepackage{graphicx}
\usepackage{wasysym}

\makeatletter
\numberwithin{equation}{section}
\numberwithin{figure}{section}
\theoremstyle{plain}
\newtheorem{thm}{\protect\theoremname}
\theoremstyle{remark}
\newtheorem{rem}[thm]{\protect\remarkname}
\theoremstyle{plain}
\newtheorem{lem}[thm]{\protect\lemmaname}

\usepackage{algpseudocode,algorithm,algorithmicx}
\usepackage{tikz}
\usetikzlibrary{shapes,arrows,calc,decorations.pathreplacing}

\@ifundefined{showcaptionsetup}{}{%
 \PassOptionsToPackage{caption=false}{subfig}}
\usepackage{subfig}
\makeatother

\usepackage{babel}
\providecommand{\lemmaname}{Lemma}
\providecommand{\remarkname}{Remark}
\providecommand{\theoremname}{Theorem}

\begin{document}

\title{Positivity-preserving and entropy-bounded discontinuous Galerkin
method for the chemically reacting, compressible Euler equations.
Part I: The one-dimensional case}

\author{Eric J. Ching, Ryan F. Johnson, and Andrew D. Kercher}
\address{Laboratories for Computational Physics and Fluid Dynamics,  U.S. Naval Research Laboratory, 4555 Overlook Ave SW, Washington, DC 20375}

\begin{abstract}
In this paper, we develop a fully conservative, positivity-preserving,
and entropy-bounded discontinuous Galerkin scheme for simulating the
multicomponent, chemically reacting, compressible Euler equations
with complex thermodynamics. The proposed formulation is an extension
of the fully conservative, high-order numerical method previously
developed by Johnson and Kercher {[}\emph{J. Comput. Phys.}, 423 (2020),
109826{]} that maintains pressure equilibrium between adjacent elements.
In this first part of our two-part paper, we focus on the one-dimensional
case. Our methodology is rooted in the minimum entropy principle satisfied
by entropy solutions to the multicomponent, compressible Euler equations,
which was proved by Gouasmi et al. {[}\emph{ESAIM: Math. Model. Numer.
Anal.}, 54 (2020), 373-{}-389{]} for nonreacting flows. We first show
that the minimum entropy principle holds in the reacting case as well.
Next, we introduce the ingredients, including a simple linear-scaling
limiter, required for the discrete solution to have nonnegative species
concentrations, positive density, positive pressure, and bounded entropy.
We also discuss how to retain the aforementioned ability to preserve
pressure equilibrium between elements. Operator splitting is employed
to handle stiff chemical reactions. To guarantee discrete satisfaction
of the minimum entropy principle in the reaction step, we develop
an entropy-stable discontinuous Galerkin method based on diagonal-norm
summation-by-parts operators for solving ordinary differential equations.
The developed formulation is used to compute canonical one-dimensional
test cases, namely thermal-bubble advection, multicomponent shock-tube
flow, and a moving hydrogen-oxygen detonation wave with detailed chemistry.
We demonstrate that the developed formulation can achieve optimal
high-order convergence in smooth flows. Furthermore, we find that
the enforcement of an entropy bound can considerably reduce the large-scale
nonlinear instabilities that emerge when only the positivity property
is enforced, to an even greater extent than in the monocomponent,
calorically perfect case. Finally, mass, total energy, and atomic
elements are shown to be discretely conserved.
\end{abstract}
\begin{keyword}
Discontinuous Galerkin method; Combustion; Detonation; Minimum entropy
principle; Positivity-preserving; Entropy stability; Summation-by-parts
\end{keyword}
\maketitle
\global\long\def\middlebar{\,\middle|\,}%
\global\long\def\average#1{\left\{  \!\!\left\{  #1\right\}  \!\!\right\}  }%
\global\long\def\expnumber#1#2{{#1}\mathrm{e}{#2}}%
 \newcommand*{\horzbar}{\rule[.5ex]{2.5ex}{0.5pt}}

\global\long\def\revisionmath#1{\textcolor{red}{#1}}%

\makeatletter \def\ps@pprintTitle{  \let\@oddhead\@empty  \let\@evenhead\@empty  \def\@oddfoot{\centerline{\thepage}}  \let\@evenfoot\@oddfoot} \makeatother

\let\svthefootnote\thefootnote\let\thefootnote\relax\footnotetext{\\ \hspace*{65pt}DISTRIBUTION STATEMENT A. Approved for public release. Distribution is unlimited.}\addtocounter{footnote}{-1}\let\thefootnote\svthefootnote

\section{Introduction\label{sec:Introduction}}

The discontinuous Galerkin (DG) method~\citep{Ree73,Bas97_2,Bas97,Coc98,Coc00}
has recently gained considerable attention in the computational fluid
dynamics community~\citep{Wan13}. Several desirable properties,
such as local conservation, arbitrarily high order of accuracy on
unstructured grids, and suitability for heterogeneous computing systems,
highlight its great potential to accurately and efficiently simulate
complex fluid flows. However, it is well-known that nonlinear instabilities
are easily introduced in underresolved regions and near flow-field
discontinuities. This issue is exacerbated when realistic thermodynamics
(e.g., the thermally perfect gas model) and multispecies chemical
reactions are incorporated. For example, fully conservative schemes
(not just the DG method) are known to generate spurious pressure oscillations
in moving interface problems~\citep{Abg88,Kar94,Abg96}. To remedy
this issue, quasi-conservative methods are often employed, such as
the double-flux scheme~\citep{Abg01}, in which the equation of state
is recast based on a calorically perfect gas model using frozen, elementwise-constant
auxiliary variables. The double-flux method has been utilized in a
number of studies involving DG simulations of multicomponent flows~\citep{Bil11,Lv15,Ban20}.
While effective at eliminating the aforementioned pressure oscillations,
the double-flux approach violates energy conservation, which can be
crucial for the prediction of shock speeds and locations, as well
as heat release in combustion processes. As a compromise, Lv and Ihme~\citep{Lv15}
proposed a hybrid double-flux strategy wherein a fully conservative
method is employed at shocks.

On the other hand, Johnson and Kercher recently proposed an explicit,
fully conservative, high-order method that does not generate spurious
pressure oscillations in smooth flow regions or across material interfaces
when the temperature is continuous~\citep{Joh20_2}. This is done
via (a) exact evaluation of the thermodynamics and (b) calculation
of the inviscid and viscous fluxes in a consistent manner that maintains
pressure equilibrium, which is drastically simplified through a particular
choice of nodal basis. It is worth noting that their proposed strategy
is not limited to DG schemes, but can be applied to other numerical
methods as well. Stiff chemical reactions were handled via operator
splitting. Efficient and accurate integration of the chemical source
terms was achieved via an $hp$-adaptive DG method for solving ordinary
differential equations, termed \emph{DGODE}. Optimal high-order accuracy
was demonstrated for smooth flows, and a suite of complex multicomponent
reacting flows was computed. The high-order calculation of a three-dimensional
reacting shear flow in the presence of a splitter plate did not require
any additional stabilization. Also computed was a two-dimensional,
moving detonation wave. Artificial viscosity was used to stabilize
the shock fronts present in the solution. However, a very fine mesh
was required to maintain robustness while achieving correct prediction
of the cellular structures behind the shock, highlighting the difficulty
of robustly and accurately simulating multidimensional detonation
waves on coarse meshes, especially for high-order methods. Even if
spurious pressure oscillations are sufficiently minimized, the wide
range of complex flow features characterizing detonations is difficult
to capture~\citep{Dei03}. Such features include thin reaction zones,
traveling pressure waves, shock-shock interactions, Kelvin-Helmholtz
instabilities, vortical structures, and triple points. As previously
discussed, underresolution of flow features induces instabilities
that can cause solver divergence, and these instabilities may be amplified
by the added nonlinearity of the variable thermodynamics, multicomponent
flow, and stiff chemical reactions. Another difficulty associated
with multicomponent flow is the frequent occurrence of negative species
concentrations, especially since initial and boundary conditions often
specify the mole fractions of certain species to be zero. Many reacting-flow
solvers simply ``clip'' negative concentrations to zero, which violates
conservation and introduces low-order errors.

In light of the above, our primary objective in this study is to develop
a positivity-preserving and entropy-bounded DG scheme for simulating
the multicomponent, chemically reacting Euler equations with exact
thermodynamics for mixtures of thermally perfect gases. Specifically,
we build upon the aforementioned fully conservative high-order method
that can maintain pressure equilibrium~\citep{Joh20_2}. In this
first part of our two-part paper, we focus on the one-dimensional
case. The unique challenges posed by realistic thermodynamics and
stiff chemical source terms are discussed and addressed. \emph{Entropy-bounded}
in this context means that the specific thermodynamic entropy of the
discrete solution is bounded from below, an idea rooted in the minimum
entropy principle satisfied by entropy solutions to the compressible,
multicomponent Euler equations. This principle was recently proved
for the nonreacting case by Gouasmi et al.~\citep{Gou20}, which
is an important prerequisite of this work. The developed formulation
in general preserves order of accuracy for smooth solutions. Our main
contributions are as follows:
\begin{itemize}
\item A minimum entropy principle for the compressible, multicomponent,
chemically reacting Euler equations is demonstrated, which follows
naturally from the proof in~\citep{Gou20}. 
\item In a DG framework, we extend the fully conservative high-order discretization
in~\citep{Joh20_2} to be positivity-preserving and entropy-bounded.
We also discuss a different local entropy bound that is less restrictive
than that previously introduced in~\citep{Lv15_2} in the context
of the monocomponent Euler equations.
\item To provably guarantee satisfaction of the minimum entropy principle
in the temporal integration of stiff chemical source terms, we extend
DGODE by developing an entropy-stable DGODE based on diagonal-norm
summation-by-parts (SBP) operators. This involves deriving a new entropy-conservative
two-point numerical state function (note that similar entropy-conservative
numerical state functions were derived for the monocomponent Euler,
shallow-water, and ideal magnetohydrodynamics equations by Friedrich
et al.~\citep{Fri19} in the context of an entropy-stable space-time
DG discretization). %
\item We employ the proposed entropy-bounded DG method to robustly and accurately
compute a series of canonical one-dimensional test cases. Mass, energy,
and atom conservation are maintained. The proposed formulation more
effectively suppresses spurious oscillations than the positivity-preserving
DG scheme. In particular, we find that the relative benefit of enforcing
an entropy bound is significantly greater in the multicomponent, thermally
perfect setting than in monocomponent, calorically perfect setting.
\end{itemize}
In Part II~\citep{Chi22_2}, we extend the entropy-bounded DG scheme
to multiple dimensions on arbitrary elements. Our multidimensional
extension is a further generalization of the multidimensional positivity-preserving/entropy-bounded
schemes currently in the literature~\citep{Zha10,Zha12,Lv15_2,Jia18}.
Specifically, restrictions on the numerical flux, physical modeling,
element shape, polynomial order of the geometric approximation, and/or
quadrature rules are relaxed. Complex detonation waves in two and
three dimensions are computed. We find that whereas the positivity-preserving
DG scheme is often not sufficiently stable (even with artificial viscosity
to stabilize strong discontinuities), enforcing an entropy bound enables
robust calculations on relatively coarse meshes. 

\subsection{Background}

Various stabilization strategies for high-order DG schemes have been
introduced in the literature. Artificial viscosity is a popular approach
that very effectively suppresses oscillations and is perfectly compatible
with high polynomial orders, arbitrary elements, and general equation
sets~\citep{Per06,Bar10,Chi19}. However, there are certain limitations
that discourage an overreliance on artificial viscosity for suppressing
\emph{all} instabilities. First, it can significantly pollute accuracy,
especially in smooth regions of the flow. As such, it should ideally
be added only where necessary (e.g., strong shocks that are otherwise
difficult to robustly capture). Second, design of a shock sensor that
can reliably detect discontinuous flow features for general configurations
remains an open problem. Even in a single flow configuration that
involves discontinuities of varying strengths (as is the case for
detonations), it is difficult to detect all such discontinuities and
add the ``right'' amount of artificial viscosity. Third, there is
typically a very strong dependence on tunable parameters. 

A common alternative to artificial viscosity is limiting. WENO-type~\citep{Luo07,Maz19},
TVD/TVB~\citep{Coc89,Coc89_2}, and moment limiters~\citep{Kri07}
are well-known examples. However, it can be difficult to extend these
limiters to arbitrary polynomial orders for both the solution and
geometric approximations, as well as to general equation sets. Furthermore,
these limiters are not guaranteed to yield physically admissible solutions
(i.e., positive density and pressure), a drawback of artificial viscosity
as well. An alternative approach for maintaining robustness is to
align the grid with discontinuities. Recent formulations that do so
in an implicit manner by treating the grid as a variable were developed
by Corrigan et al.~\citep{Cor18} and Zahr and Persson~\citep{Zah18}.
An encouraging preliminary effort to apply implicit shock tracking
to reacting flow is discussed in~\citep{Zah21}, in which supersonic
inviscid reacting flow over a two-dimensional wedge with simple thermodynamics
and chemistry was computed. 

Recently, positivity-preserving DG schemes~\citep{Zha10,Zha11,Zha12,Zha17}
have seen considerable success in solving the monocomponent, nonreacting
Euler equations with explicit time stepping. These schemes prevent
the occurrence of negative densities and pressures under a constraint
on the time step size, a positivity-preserving numerical flux, and
a limiting procedure consisting of a simple linear ``squeezing''
of the solution towards its cell average. The limiting operator is
conservative and maintains order of accuracy for smooth solutions.
However, the limiter is not very effective at dampening oscillations.
The positivity-preserving DG method was extended to the entropy-bounded
DG scheme by Zhang and Shu~\citep{Zha12_2}, in which an additional
limiting step based on a Newton search was introduced to enforce a
global entropy bound.%
. %
{} Note that it is implicitly assumed that an entropy-bounded scheme
is also positivity-preserving. Under their invariant-region-preserving
DG framework, Jiang and Liu~\citep{Zha12_2,Jia18} introduced a more
straightforward limiter to enforce the entropy principle in an algebraic
manner, which is particularly desirable in the multicomponent, thermally
perfect case due to the cost of evaluating complex thermodynamics.
Lv and Ihme~\citep{Lv15_2} further extended the entropy-bounded
DG scheme (for the monocomponent Euler equations) by relaxing restrictions
on the geometry and quadrature rules. They also introduced a local
entropy bound. Numerical tests demonstrated the superiority of the
entropy-bounded DG scheme for suppressing spurious oscillations, compared
to the positivity-preserving DG method. Lv and Ihme later applied
entropy bounding to their reacting flow DG solver~\citep{Lv17},
which utilizes the double-flux approach discussed above. The frozen
thermodynamics and relaxation towards a calorically perfect gas model
circumvent the difficulties of extending the entropy-bounded DG method
to reacting flow with exact thermodynamics and non-calorically-perfect
gases. Furthermore, the physico-mathematical validity of combining
the double-flux model with enforcement of a discrete minimum entropy
principle is not immediately clear. We also note that these positivity-preserving
and entropy-bounded DG methods are related to the recent geometric
quasilinearization framework by Wu and Shu~\citep{Wu21}, as well
as the invariant-domain-preserving schemes based on graph viscosity
by, for example, Guermond et al.~\citep{Gue19} and Pazner~\citep{Paz21},
which employ a convex limiting procedure relying on an iterative line
search.

Before concluding this section, we note several other efforts to compute
chemically reacting flows with DG schemes in a stable manner. Gutierrez-Jorquera
and Kummer~\citep{Gut22} computed steady-state diffusion flames
using a low-Mach pressure-based solver and a one-step reaction mechanism.
May et al.~\citep{May21} simulated steady hypersonic flows with
some high-enthalpy effects using a hybridized DG solver that employs
artificial viscosity for shock capturing. Papoutsakis et al.~\citep{Pap18}
computed chemically reacting hypersonic flow over a double cone with
a TVB limiter; however, only a linear polynomial approximation of
the solution was employed, and there were discrepancies with finite-volume
and experimental results. A number of positivity-preserving DG schemes
for the reacting Euler equations have also been developed. For example,
Zhang and Shu~\citep{Zha11} and Wang et al.~\citep{Wan12} extended
the positivity-preserving DG scheme to handle source terms with explicit
time stepping. Complex thermodynamics and stiff source terms were
not addressed. Du and Yang~\citep{Du19} and Du et al.~\citep{Du19_2}
presented a positivity-preserving DG method based on a new explicit,
exponential Runge-Kutta (RK)/multistep time integration scheme that
can handle stiff source terms~\citep{Hua18}. However, the temporal
order of accuracy is strongly dependent on the initial conditions.
Other positivity-preserving time integrators compatible with stiff
source terms include implicit Patankar-type RK schemes~\citep{Hua19,Hua19_2,Pan21},
which have been applied to finite difference~\citep{Hua19,Hua19_2}
and finite volume~\citep{Pan21} discretizations of the reacting
Euler equations. These schemes are still undergoing development, and
there may be issues with extending them to DG discretizations~\citep{Hua19}.
An additional obstacle of both exponential RK/multistep and Patankar-type
RK schemes is proper enforcement of the minimum entropy principle.
For these reasons and due to its proven past success, we elect to
still employ operator splitting to deal with stiff source terms. Nevertheless,
exponential RK/multistep and Pantankar-type schemes are indeed worthy
of future consideration.

That we use an entropy-stable DG discretization to temporally integrate
the chemical source terms may cause readers to question why an entropy-stable
DG discretization is not also employed for the transport step. Note
that entropy-stable schemes guarantee the global integral of mathematical
entropy to be nonincreasing in time (assuming periodic or entropy-stable
boundary conditions), whereas entropy-bounded schemes enforce, in
a pointwise fashion, the specific thermodynamic entropy to be greater
than some lower bound. Entropy-stable DG schemes for multicomponent
(nonreacting) flows have emerged only recently~\citep{Gou20_2,Ren21}.
For general hyperbolic conservation laws, there are various ways to
achieve entropy stability; two of the most well-known are as follows:
(a) the entropy variables, instead of the conservative variables,
are directly solved for, and (b) SBP operators are used to approximate
discrete derivatives. The first approach, in which the conservative
variables depend implicitly on the polynomial approximation of the
entropy variables, is not appropriate for explicit time stepping,
which we employ in this work (for the transport terms). Furthermore,
with the typical choice of entropy, the entropy variables are undefined
when any of the partial densities vanishes~\citep{Gou20_2}. Conversely,
the second approach, which has surged in popularity in recent years,
is compatible with explicit time stepping and does not assume exact
integration to provably guarantee entropy stability. Belonging to
this category are the methods by Renac~\citep{Ren21} and Peyvan
et al.~\citep{Pey22}. On unstructured grids, particularly those
with (possibly curved) simplicial elements, current SBP-based entropy-stable
methods can become suboptimal~\citep{Che17}, require direct use
of the entropy variables (even if the conservative variables are the
unknowns)~\citep{Cha19,Cha18_2}, and/or significantly increase in
complexity~\citep{Cha19,Cre18,Cha18_2,Che20}. Note that these issues
are not present in the developed entropy-stable DGODE (except a moderate
increase in complexity), which entails a one-dimensional discretization
in time. Due to the above factors, as well as the relative simplicity
of the entropy-bounded DG scheme and its compatibility with the pressure-equilibrium-maintaining
discretization by Johnson and Kercher~\citep{Joh20_2}, we choose
to employ the proposed entropy-bounded method for the transport step.
Nevertheless, we emphasize that we are not ruling out entropy-stable
schemes; advancements to these formulations are rapid and their potential
is evident~\citep{Gas21}. Finally, it should be noted that although
the construction of entropy-stable and entropy-bounded methods rely
on different techniques, entropy stability and entropy boundedness
are not necessarily mutually exclusive. Depending on the type of entropy-stable
scheme, it should, in principle, be feasible to achieve discrete satisfaction
of both properties.

The remainder of this paper is organized as follows. Sections~\ref{sec:governing_equations}
and~\ref{sec:DG-discretization} summarize the governing equations
and basic DG discretization, respectively. Section~\ref{sec:minimum-entropy-principle}
reviews the minimum entropy principle associated with the compressible,
multicomponent, nonreacting Euler equations and extends it to the
reacting Euler equations. The next section presents the positivity-preserving
and entropy-bounded DG formulation for the transport step. We then
discuss the entropy-stable DG discretization for the reaction step
in Section~\ref{sec:entropy-stable-DGODE}. Results for fundamental
nonreacting test cases and one-dimensional detonation-wave simulations
are given in the following section. We close the paper with concluding
remarks.

\section{Governing equations\label{sec:governing_equations}}

The compressible, multicomponent, chemically reacting Euler equations
are given as
\begin{equation}
\frac{\partial y}{\partial t}+\nabla\cdot\mathcal{F}\left(y\right)-\mathcal{S}\left(y\right)=0\label{eq:conservation-law-strong-form}
\end{equation}
where $t\in\mathbb{R}^{+}$ is time, $y(x,t):\mathbb{R}^{d}\times\mathbb{R}^{+}\rightarrow\mathbb{R}^{m\times d}$
is the conservative state vector (with $x=(x_{1},\ldots,x_{d})$ denoting
the physical coordinates), $\mathcal{F}(y):\mathbb{R}^{m}\rightarrow\mathbb{R}^{m\times d}$
is the convective flux, $\mathcal{S}(y):\mathbb{R}^{m}\rightarrow\mathbb{R}^{m}$
is the chemical source term. The state vector is expanded as

\begin{equation}
y=\left(\rho v_{1},\ldots,\rho v_{d},\rho e_{t},C_{1},\ldots,C_{n_{s}}\right)^{T},\label{eq:reacting-navier-stokes-state}
\end{equation}
where $n_{s}$ is the number of species (which yields $m=d+n_{s}+1$),
$\rho$ is density, $v=\left(v_{1},\ldots,v_{d}\right)$ is the velocity,
$e_{t}$ is the mass-specific total energy, and $C=\left(C_{1},\ldots,C_{n_{s}}\right)$
are the species concentrations. The density is computed from the species
concentrations as

\[
\rho=\sum_{i=1}^{n_{s}}\rho_{i}=\sum_{i=1}^{n_{s}}W_{i}C_{i},
\]

\noindent where $\rho_{i}$ is the partial density and $W_{i}$ is
the molecular weight of the $i$th species. The mass fraction and
mole fraction of the $i$th species are defined as 
\[
Y_{i}=\frac{\rho_{i}}{\rho}
\]
and
\[
X_{i}=\frac{C_{i}}{\sum_{i=1}^{n_{s}}C_{i}},
\]
respectively.

\noindent The $k$th spatial convective flux component is written
as
\begin{equation}
\mathcal{F}_{k}^{c}\left(y\right)=\left(\rho v_{k}v_{1}+P\delta_{k1},\ldots,\rho v_{k}v_{d}+P\delta_{kd},v_{k}\left(\rho e_{t}+P\right),v_{k}C_{1},\ldots,v_{k}C_{n_{s}}\right)^{T},\label{eq:reacting-navier-stokes-spatial-convective-flux-component}
\end{equation}
where $P$ is the pressure. The mass-specific total energy is the
sum of the specific internal and kinetic energies, given by

\[
e_{t}=u+\frac{1}{2}\sum_{k=1}^{d}v_{k}v_{k},
\]
where the (mixture-averaged) mass-specific internal energy, $u$,
is the mass-weighted sum of the mass-specific internal energies of
each species:
\[
u=\sum_{i=1}^{n_{s}}Y_{i}u_{i}.
\]
This work assumes thermally perfect gases, with $u_{i}$ given by~\citep{Gio99}
\[
u_{i}=h_{i}-R_{i}T=h_{\mathrm{ref},i}+\int_{T_{\mathrm{ref}}}^{T}c_{p,i}(\tau)d\tau-R_{i}T,
\]
where $h_{i}$ is the mass-specific enthalpy of the $i$th species,
$R_{i}=R^{0}/W_{i}$ (with $R^{0}=8314.4621\,\mathrm{JKmol}^{-1}\mathrm{K}^{-1}$
denoting the universal gas constant), $T$ is the temperature, $T_{\mathrm{ref}}$
is the reference temperature (298.15 K), $h_{\mathrm{ref},i}$ is
the reference-state species formation enthalpy, and $c_{p,i}$ is
the mass-specific heat at constant pressure of the $i$th species.
$c_{p,i}$ is computed from an $n_{p}$-order polynomial as
\begin{equation}
c_{p,i}=\sum_{k=0}^{n_{p}}a_{ik}T^{k},\label{eq:specific_heat_polynomial}
\end{equation}
based on the NASA coefficients~\citep{Mcb93,Mcb02}. The mass-specific
thermodynamic entropy of the mixture is defined as
\[
s=\sum_{i=1}^{n_{s}}Y_{i}s_{i},
\]
with $s_{i}$ given by
\[
s_{i}=s_{i}^{o}-R_{i}\log\frac{P_{i}}{P_{\mathrm{ref}}},\quad s_{i}^{o}=s_{\mathrm{ref},i}^{o}+\int_{T_{\mathrm{ref}}}^{T}\frac{c_{p,i}(\tau)}{\tau}d\tau,
\]
where $s_{\mathrm{ref},i}^{o}$ is the species formation entropy at
the reference temperature and reference pressure ($P_{\mathrm{ref}}=1\text{ atm}$),
$s_{i}^{o}$ denotes the species entropy at atmospheric pressure,
and $P_{i}=C_{i}R^{0}T$ is the partial pressure. $s_{i}$ can also
be expressed as~\citep{Gio99,Gou20,Gou20_2}
\[
s_{i}=s_{\mathrm{ref},i}^{o}+\int_{T_{\mathrm{ref}}}^{T}\frac{c_{v,i}(\tau)}{\tau}d\tau-R_{i}\log\frac{C_{i}}{C_{\mathrm{ref}}},
\]
where $C_{\mathrm{ref}}=P_{\mathrm{ref}}/R^{0}T_{\mathrm{ref}}$ is
the reference concentration and $c_{v,i}=c_{p,i}-R_{i}$ is the mass-specific
heat at constant volume of the $i$th species. Summing up the partial
pressures yields the equation of state for the mixture: 
\begin{equation}
P=R^{0}T\sum_{i=1}^{n_{s}}C_{i}.\label{eq:EOS}
\end{equation}
$u_{i}$, $h_{i}$, and $s_{i}^{o}$ are computed by integrating Equation~(\ref{eq:specific_heat_polynomial})
and incorporating the integration constants in~\citep{Mcb93} and~\citep{Mcb02}.
For example, $u_{i}$ is calculated as
\begin{equation}
u_{i}=b_{i0}+\sum_{k=0}^{n_{p}}\frac{a_{ik}}{k+1}T^{k+1}-R_{i}T=\sum_{k=0}^{n_{p}+1}b_{ik}T^{k},\label{eq:internal-energy-polynomial}
\end{equation}
where $b_{i0}$ is the integration constant and
\[
b_{ik}=\begin{cases}
\frac{a_{i,k-1}}{k}, & k>1\\
a_{i0}-R_{i}, & k=1.
\end{cases}
\]

\subsection{Chemical reaction rates\label{subsec:chemical-reaction-rates}}

The source term in Equation~(\ref{eq:conservation-law-strong-form})
is a smooth function of the state variables, written as~\citep{Kee96}

\begin{equation}
\mathcal{S}\left(y\right)=\left(0,\ldots,0,0,\omega_{1},\ldots,\omega_{n_{s}}\right)^{T},\label{eq:reacting-navier-stokes-source-term}
\end{equation}
where $\omega_{i}$ is the production rate of the $i$th species,
which satisfies mass conservation:
\begin{equation}
\sum_{i=1}^{n_{s}}W_{i}\omega_{i}=0.\label{eq:chemical-reaction-mass-conservation}
\end{equation}
The production rate is computed as
\[
\omega_{i}=\sum_{j=1}^{n_{r}}\nu_{ij}q_{j}.
\]
$n_{r}$ is the number of reactions, $\nu_{ij}=\nu_{ij}^{r}-\nu_{ij}^{f}$
is the difference between the reverse ($\nu_{ij}^{r}$) and the forward
stoichiometric coefficients ($\nu_{ij}^{f}$), and $q_{j}$ is the
rate of progress of the $j$th reaction, computed as
\begin{equation}
q_{j}=k_{j}^{f}\prod_{i=1}^{n_{s}}C_{i}^{\nu_{ij}^{f}}-k_{j}^{r}\prod_{i=1}^{n_{s}}C_{i}^{\nu_{ij}^{r}},\label{eq:chemical-reaction-rate-of-progress}
\end{equation}
where $k_{j}^{f}$ and $k_{j}^{r}$ are the forward and reverse rate
constants, respectively, of the $j$th reaction. The forward and reverse
rate constants are related via the equilibrium constant,
\begin{equation}
K_{j}^{e}=\exp\left(-\frac{\Delta G_{j}'}{R^{0}T}\right)\left(\frac{P_{\mathrm{ref}}}{R^{0}T}\right)^{\sum_{i}\nu_{ij}},\label{eq:equilibrium-constant-1}
\end{equation}
where $\Delta G_{j}'$ is the change in reference-state Gibbs free
energy for the $j$th reaction, given as
\[
\Delta G_{j}'=\sum_{i=1}^{n_{s}}\nu_{ij}W_{i}h_{i}-T\sum_{i=1}^{n_{s}}\nu_{ij}W_{i}s_{i}'.
\]
Introducing the reduced chemical potentials of the $i$th species,
\begin{align*}
\mu_{i} & =\frac{g_{i}}{R^{0}T},\\
\mu_{i}^{\mathrm{u}} & =\mu_{i}-\frac{1}{W_{i}}\log C_{i},
\end{align*}
where $g_{i}=h_{i}-Ts_{i}$ is the Gibbs function of the $i$th species,
the equilibrium constant can also be written as~\citep[Chapter 6.4]{Gio99}
\begin{equation}
K_{j}^{e}=\exp\left(-\sum_{i=1}^{n_{s}}\nu_{ij}W_{i}\mu_{i}^{\mathrm{u}}\right),\label{eq:equilibrium-constant-2}
\end{equation}
There exist various models for approximating the forward rate constants
in Equation~(\ref{eq:chemical-reaction-rate-of-progress}), several
of which will be briefly discussed next.

\subsubsection{Arrhenius reactions}

The Arrhenius form is the most common model for approximating reaction
rates. The forward rate constants are computed as
\[
k_{j}^{f}=A_{j}T^{b_{j}}\exp\left(-\frac{E_{j}}{R^{0}T}\right),
\]
 where $A_{j}>0$ and $b_{j}$ are parameters and $E_{j}\geq0$ is
the activation energy~\citep{Gio99,Kee96}.

\subsubsection{Three-body reactions}

These reactions require a ``third body'' in order to proceed. Dissociation
and recombination reactions are often of this type. The rate of progress
is scaled by a prefactor as~\citep{Kee96}
\[
q_{j}=\left(\sum_{i=1}^{n_{s}}\alpha_{ij}C_{i}\right)\left(k_{j}^{f}\prod_{i=1}^{n_{s}}C_{i}^{\nu_{ij}^{f}}-k_{j}^{r}\prod_{i=1}^{n_{s}}C_{i}^{\nu_{ij}^{r}}\right),
\]
where $\alpha_{ij}$ are the third-body efficiencies.

\subsubsection{Unimolecular/recombination fall-off reactions\label{subsec:fall-off-reactions}}

Unimolecular/recombination fall-off reactions incorporate a dependence
on pressure. In general, this model predicts an increase in the reaction
rate with increasing pressure. For brevity, we drop the $j$ subscript
and $f$ superscript. Given Arrhenius-type low-pressure and high-pressure
limits for the rate constant ($k_{0}$ and $k_{\infty}$, respectively),
$k$ is computed as
\begin{equation}
k=k_{\infty}\left(\frac{P_{r}}{1+P_{r}}\right)F,\label{eq:falloff-reaction}
\end{equation}
where $P_{r}$ is the reduced pressure, defined as
\[
P_{r}=\frac{k_{0}}{k_{\infty}}\sum_{i=1}^{n_{s}}\alpha_{i}C_{i}.
\]
There are different ways to compute $F$ in Equation~(\ref{eq:falloff-reaction}).
With the Lindemann~\citep{Lin22} approach, $F$ is simply unity.
In the Troe~\citep{Gil83} form, $F$ is given by
\[
\log F=\frac{\log F_{\mathrm{cent}}}{1+\left[\frac{\log P_{r}+c_{1}}{c_{2}-c_{3}\left(\log P_{r}+c_{1}\right)}\right]^{2}},
\]
where the definitions of $c_{1}$, $c_{2}$, $c_{3}$, and $F_{\mathrm{cent}}$
are given in~\citep{Gil83}.

\subsubsection{Chemically activated bimolecular reactions}

Reactions of this type are also pressure-dependent, but the reaction
rates typically decrease with increasing pressure. The rate constants
are computed as~\citep{Kee96}
\[
k=k_{0}\left(\frac{1}{1+P_{r}}\right)F,
\]
where $k_{0}$, $P_{r}$, and $F$ are calculated as in Section~\ref{subsec:fall-off-reactions}.

\section{Discontinuous Galerkin discretization\label{sec:DG-discretization}}

In this section, we briefly describe the DG discretization of the
governing equations and review the techniques proposed in~\citep{Joh20_2}
to prevent spurious pressure oscillations in smooth regions of the
flow.

Let $\Omega\subset\mathbb{R}^{d}$ be the $d$-dimensional computational
domain partitioned by $\mathcal{T}$, which consists of non-overlapping
cells $\kappa$ with boundaries $\partial\kappa$. Let $\mathcal{E}$
denote the set of interfaces $\epsilon$, with $\cup_{\epsilon\in\mathcal{E}}\epsilon=\cup_{\kappa\in\mathcal{T}}\partial\kappa$.
$\mathcal{E}$ consists of the interior interfaces,
\[
\epsilon_{\mathcal{I}}\in\mathcal{E_{I}}=\left\{ \epsilon_{\mathcal{I}}\in\mathcal{E}\middlebar\epsilon_{\mathcal{I}}\cap\partial\Omega=\emptyset\right\} ,
\]
and boundary interfaces, 
\[
\epsilon_{\partial}\in\mathcal{E}_{\partial}=\left\{ \epsilon_{\partial}\in\mathcal{E}\middlebar\epsilon_{\partial}\subset\partial\Omega\right\} ,
\]
such that $\mathcal{E}=\mathcal{E_{I}}\cup\mathcal{E}_{\partial}$.
At interior interfaces, there exists $\kappa^{+},\kappa^{-}\in\mathcal{T}$
such that $\epsilon_{\mathcal{I}}=\partial\kappa^{+}\cap\partial\kappa^{-}$.
$n^{+}$ and $n^{-}$ denote the outward facing normal of $\kappa^{+}$
and $\kappa^{-}$, respectively, with $n^{+}=-n^{-}$. The discrete
(finite-dimensional) subspace $V_{h}^{p}$ over $\mathcal{T}$ is
defined as
\begin{eqnarray}
V_{h}^{p} & = & \left\{ \mathfrak{v}\in\left[L^{2}\left(\Omega\right)\right]^{m}\middlebar\forall\kappa\in\mathcal{T},\left.\mathfrak{v}\right|_{\kappa}\in\left[\mathcal{P}_{p}(\kappa)\right]^{m}\right\} ,\label{eq:discrete-subspace}
\end{eqnarray}
where, for $d=1$, $\mathcal{P}_{p}(\kappa)$ is the space of polynomial
functions of degree no greater than $p$ in $\kappa$. %
{} For $d>1$, the choice of polynomial space typically depends on the
element type~\citep{Har13}.

The semi-discrete form of the governing equations (Equation~(\ref{eq:conservation-law-strong-form}))
is given as: find $y\in V_{h}^{p}$ such that
\begin{gather}
\sum_{\kappa\in\mathcal{T}}\left(\frac{\partial y}{\partial t},\mathfrak{v}\right)_{\kappa}-\sum_{\kappa\in\mathcal{T}}\left(\mathcal{F}\left(y\right),\nabla\mathfrak{v}\right)_{\kappa}+\sum_{\epsilon\in\mathcal{E}}\left(\mathcal{F}^{\dagger}\left(y^{+},y^{-},n\right),\left\llbracket \mathfrak{v}\right\rrbracket \right)_{\mathcal{E}}-\sum_{\kappa\in\mathcal{T}}\left(\mathcal{S}\left(y\right),\mathfrak{v}\right)_{\kappa}=0\qquad\forall\:\mathfrak{v}\in V_{h}^{p},\label{eq:semi-discrete-form}
\end{gather}
where $\left(\cdot,\cdot\right)$ denotes the inner product, $\mathcal{F}^{\dagger}\left(y^{+},y^{-},n\right)$
is the numerical flux, and $\left\llbracket \cdot\right\rrbracket $
is the jump operator, given by $\left\llbracket \mathfrak{v}\right\rrbracket =\mathfrak{v}^{+}-\mathfrak{v}^{-}$
at interior interfaces and $\left\llbracket \mathfrak{v}\right\rrbracket =\mathfrak{v}^{+}$
at boundary interfaces. Applying a standard, fully explicit time stepping
scheme to Equation~(\ref{eq:semi-discrete-form}) would yield an
exceedingly small time step due to the stiff chemical source terms.
As such, operator splitting is employed to decouple the temporal integration
of the convection operator from that of the source term. Specifically,
we apply Strang splitting~\citep{Str68} over a given interval $(t_{0},t_{0}+\Delta t]$
as
\begin{align}
\frac{\partial y}{\partial t}+\nabla\cdot\mathcal{F}\left(y\right)=0 & \textup{ in }\Omega\times\left(t_{0},t_{0}+\nicefrac{\Delta t}{2}\right],\label{eq:strang-splitting-1}\\
\frac{\partial y}{\partial t}-\mathcal{S}\left(y\right)=0 & \textup{ in }\left(t_{0},t_{0}+\Delta t\right],\label{eq:strang-splitting-2}\\
\frac{\partial y}{\partial t}+\nabla\cdot\mathcal{F}\left(y\right)=0 & \textup{ in }\Omega\times\left(t_{0}+\nicefrac{\Delta t}{2},t_{0}+\Delta t\right],\label{eq:strang-splitting-3}
\end{align}
where Equations~(\ref{eq:strang-splitting-1}) and~(\ref{eq:strang-splitting-3})
are integrated in time with an explicit RK-type scheme, while Equation~(\ref{eq:strang-splitting-2})
is solved using a fully implicit, temporal DG discretization for ODEs
(DGODE). Details on DGODE and its extension to entropy-stable DGODE
are given in Section~\ref{sec:entropy-stable-DGODE}. More sophisticated
operator-splitting schemes can be employed as well~\citep{Wu19}. 

Unless otherwise specified, the volume and surface terms in Equation~(\ref{eq:semi-discrete-form})
are evaluated using a quadrature-free approach~\citep{Atk96,Atk98}.
Throughout this work, we employ a nodal basis, such that the element-local
polynomial approximation of the solution is expanded as
\begin{equation}
y_{\kappa}=\sum_{j=1}^{n_{b}}y_{\kappa}(x_{j})\phi_{j},\label{eq:solution-approximation}
\end{equation}
where $n_{b}$ is the number of basis functions, $\left\{ \phi_{1},\ldots,\phi_{n_{b}}\right\} $
are the basis functions, and $\left\{ x_{1},\ldots,x_{n_{b}}\right\} $
are the node coordinates. Its average over $\kappa$ is given by
\begin{equation}
\overline{y}_{\kappa}=\frac{1}{\left|\kappa\right|}\int_{\kappa}ydx,\label{eq:solution-element-average}
\end{equation}
where $\left|\kappa\right|$ is the volume of $\kappa$. In the evaluation
of the second and third integrals in Equation~(\ref{eq:semi-discrete-form}),
the nonlinear convective flux can be approximated as
\begin{equation}
\mathcal{F_{\kappa}}\approx\sum_{k=1}^{n_{c}}\mathcal{F}\left(y_{\kappa}\left(x_{k}\right)\right)\varphi_{k},\label{eq:flux-projection}
\end{equation}
where $n_{c}\geq n_{b}$ and $\left\{ \varphi_{1},\ldots,\varphi_{n_{c}}\right\} $
is a set of (potentially different) polynomial basis functions. If
$n_{c}=n_{b}$ and the integration points are included in the set
of solution nodes (e.g., Gauss-Lobatto points for tensor-product elements),
then pressure equilibrium is trivially maintained~\citep{Joh20_2}.
However, over-integration (i.e., $n_{c}>n_{b}$) is often necessary
to minimize aliasing errors and improve stability. Unfortunately,
standard over-integration, as defined in Equation~(\ref{eq:flux-projection}),
causes a loss of pressure equilibrium and generation of spurious pressure
oscillations at material interfaces~\citep{Joh20_2}. Instead, Johnson
and Kercher~\citep{Joh20_2} proposed the following approximation
of the convective flux:
\begin{equation}
\mathcal{F_{\kappa}}\approx\sum_{k=1}^{n_{c}}\mathcal{F}\left(\widetilde{y}_{\kappa}\left(x_{k}\right)\right)\varphi_{k},\label{eq:modified-flux-projection}
\end{equation}
where $\widetilde{y}:\mathbb{R}^{m}\times\mathbb{R}\rightarrow\mathbb{R}^{m}$
is a modified state defined as
\begin{equation}
\widetilde{y}\left(y,\widetilde{P}\right)=\left(\rho v_{1},\ldots,\rho v_{d},\widetilde{\rho u}\left(C_{1},\ldots,C_{n_{s}},\widetilde{P}\right)+\frac{1}{2}\sum_{k=1}^{d}\rho v_{k}v_{k},C_{1},\ldots,C_{n_{s}}\right)^{T}.\label{eq:interpolated-state-modified}
\end{equation}
$\widetilde{P}$ is a polynomial approximation of the pressure that
interpolates onto the span of $\left\{ \phi_{1},\ldots,\phi_{n_{b}}\right\} $
as
\[
\widetilde{P}_{\kappa}=\sum_{j=1}^{n_{b}}P\left(y_{\kappa}\left(x_{j}\right)\right)\phi_{j},
\]
and the modified internal energy, $\widetilde{\rho u}$, is evaluated
from the modified pressure and unmodified species concentrations.
Interpolating the convective flux (including the numerical flux function)
in this manner achieves pressure equilibrium both internally and between
adjacent elements. Of course, with finite resolution, slight deviations
from pressure equilibrium are inevitable; nevertheless, apart from
severely underresolved computations, these deviations remain small
and do not generate large-scale pressure oscillations that cause the
solver to crash, which is not the case if standard flux interpolation~(\ref{eq:flux-projection})
is employed. Additional information on the basic DG discretization,
enforcement of boundary conditions, and nonlinear flux interpolation,
as well as a detailed discussion of the conditions under which pressure
oscillations are generated, can be found in~\citep{Joh20_2}.

\section{Minimum entropy principle\label{sec:minimum-entropy-principle}}

It is well-known that in the presence of discontinuities, weak solutions
to general systems of hyperbolic conservation laws, including the
multicomponent Euler equations, are not unique~\citep{Har83_3}.
As such, physical solutions are typically identified as those that
satisfy entropy conditions, written as (in the absence of source terms)
\begin{equation}
\frac{\partial U}{\partial t}+\nabla\cdot\mathcal{F}^{s}\leq0,\label{eq:entropy-condition}
\end{equation}
where $U(y):\mathbb{R}^{m}\rightarrow\mathbb{R}$ is a given convex
(mathematical) entropy function and $\mathcal{F}^{s}(y):\mathbb{R}^{m}\rightarrow\mathbb{R}^{d}$
is the corresponding spatial entropy flux satisfying
\[
\mathsf{v}^{T}\frac{\partial\mathcal{F}}{\partial y}=\frac{\partial\mathcal{F}^{s}}{\partial y},
\]
with $\mathsf{v}$, the entropy variables, defined as
\[
\mathsf{v}=\left(\frac{\partial U}{\partial y}\right)^{T}.
\]
The mapping from the conservative variables to the entropy variables
is one-to-one and symmetrizes the system~\citep{Moc80}. \emph{Entropy
solutions} are weak solutions that satisfy~(\ref{eq:entropy-condition})
for all entropy/entropy-flux pairs. We also introduce the entropy
potential and the corresponding entropy flux potential:
\begin{equation}
\left(\mathcal{U},\mathcal{F}^{p}\right)=\left(\mathsf{v}^{T}y-U,\mathsf{v}^{T}\mathcal{F}-\mathcal{F}^{s}\right),
\end{equation}
which will be used in Section~\ref{sec:entropy-stable-DGODE}. For
the multicomponent Euler equations, $U=-\rho s$ and $\mathcal{F}^{s}=-\rho sv$
form a common admissible entropy/entropy-flux pair~\citep{Gou20,Gio99},
assuming $C_{i}>0$ and $T>0$. Note the distinction between the mathematical
entropy and the thermodynamic entropy (per unit volume); they are
typically of opposite sign. The conservation equation (for smooth
solutions) for $U=-\rho s$ can be obtained by first combining the
Gibbs relation,
\[
Tds=du-\sum_{i=1}^{n_{s}}g_{i}dY_{i}-\frac{P}{\rho^{2}}d\rho,
\]
with the transport equations for species mass fractions, density,
and internal energy~\citep{Gou20,Gio99},
\[
\begin{aligned}\frac{DY_{i}}{Dt} & =0,\;i=1,\ldots,n_{s},\\
\frac{D\rho}{Dt} & +\rho\nabla\cdot v=0,\\
\frac{Du}{Dt} & +\frac{P}{\rho}\nabla\cdot v=0,
\end{aligned}
\]
to yield the specific entropy transport equation,
\begin{equation}
\frac{Ds}{Dt}=0.\label{eq:entropy-transport-equation}
\end{equation}
Then, using conservation of mass,
\[
\frac{\partial\rho}{\partial t}+\nabla\cdot\left(\rho v\right)=0,
\]
the conservation equation for $U=-\rho s$ is obtained:
\[
\frac{\partial\rho s}{\partial t}+\nabla\cdot\left(\rho sv\right)=0.
\]

\subsection{Review: Minimum entropy principle in the compressible, multicomponent,
nonreacting Euler equations\label{subsec:minimum-entropy-principle-nonreacting}}

Gouasmi et al.~\citep{Gou20} recently proved a minimum entropy principle
satisfied by entropy solutions to the multicomponent, nonreacting
Euler equations, which means that the spatial minimum of the specific
thermodynamic entropy is a nondecreasing function of time. In this
subsection, we summarize the main steps of the proof, which itself
builds on the proof by Tadmor~\citep{Tad86} of a minimum entropy
principle in the monocomponent Euler equations. 

Integrating the inequality~(\ref{eq:entropy-condition}) over $\Omega$
and assuming nonnegative net spatial entropy outflux across $\partial\Omega$
yields
\[
\frac{d}{dt}\int_{\Omega}U\left(y(x,t)\right)dx\leq0.
\]
Integrating in time then gives
\[
\int_{\Omega}U\left(y(x,t)\right)dx\leq\int_{\Omega}U\left(y(x,0)\right)dx.
\]
Tadmor~\citep{Tad84}, however, showed that by integrating~(\ref{eq:entropy-condition})
over the truncated cone $\mathsf{C}=\left\{ \left|x\right|\leq R+v_{\max}(t-\tau)|0\leq\tau\leq t\right\} $,
a more local inequality can be written:
\begin{equation}
\int_{\left|x\right|\leq R}U\left(y(x,t)\right)dx\leq\int_{\left|x\right|\leq R+v_{\mathrm{max}}t}U\left(y(x,0)\right)dx,\label{eq:entropy-inequality-integral-local}
\end{equation}
where $v_{\max}$ is the maximum speed in the domain at $t=0$. If
we consider entropy/entropy-flux pairs of the form 
\begin{equation}
\left(U,\mathcal{F}^{s}\right)=\left(-\rho f(s),-\rho vf(s)\right),
\end{equation}
where $f$ is a smooth function of $s$, (\ref{eq:entropy-inequality-integral-local})
then becomes
\begin{equation}
\int_{\left|x\right|\leq R}\rho(x,t)\cdot f\left(s\left(y(x,t)\right)\right)dx\geq\int_{\left|x\right|\leq R+v_{\mathrm{max}}t}\rho(x,0)\cdot f\left(s\left(y(x,0)\right)\right)dx.\label{eq:entropy-inequality-inject-f}
\end{equation}
Consider the following choice for $f(s)$:
\[
f_{0}(s)=\min\left\{ s-s_{0},0\right\} ,
\]
where $s_{0}$ is the essential infimum of the specific thermodynamic
entropy in the subdomain $\Omega_{R}=\left\{ \left|x\right|\leq R+v_{\max}t\right\} $,
\[
s_{0}=\underset{\left|x\right|\leq R+v_{\max}t}{\text{Ess inf}}s(x,0).
\]
Although $f_{0}(s)$ is not a smooth function of $s$, it can be written
as the limit of a sequence of smooth functions, $f_{0}(s)=\underset{\epsilon\rightarrow0}{\lim}f_{\epsilon}(s)$,
where $f_{\epsilon}(s)$ is defined as~\citep{Gou20}
\[
f_{\epsilon}(s)=\int_{-\infty}^{\infty}f_{0}(s-\mathfrak{s})g_{\epsilon}(\mathfrak{s})d\mathfrak{s},
\]
with
\[
g_{\epsilon}(\mathfrak{s})=\frac{1}{\epsilon}\frac{\exp\left(-\frac{\mathfrak{s}^{2}}{\epsilon^{2}}\right)}{\sqrt{\pi}},\;\epsilon>0.
\]
The first and second derivatives of $f_{\epsilon}(s)$ satisfy the
following conditions:
\[
\frac{df_{\epsilon}}{ds}>0,\frac{d^{2}f_{\epsilon}}{ds^{2}}<0.
\]
Gouasmi et al.~\citep{Gou20} proved the key result that the entropy/entropy-flux
pairs $\left(U,\mathcal{F}^{s}\right)=\left(-\rho f_{\epsilon}(s),-\rho vf_{\epsilon}(s)\right)$,
with $\epsilon>0$, are admissible. In particular, they showed that
a conservation equation (for smooth solutions) for said pairs can
be obtained and that the entropy functions are convex with respect
to the conservative variables. 

With $U=-\rho f_{0}(s)$, the inequality~(\ref{eq:entropy-inequality-inject-f})
can then be written as
\[
\int_{\left|x\right|\leq R}\rho(x,t)\cdot\min\left\{ s(x,t)-s_{0},0\right\} dx\geq\int_{\left|x\right|\leq R+v_{\mathrm{max}}t}\rho(x,0)\cdot\min\left\{ s(x,0)-s_{0},0\right\} dx,
\]
where the RHS is zero by definition of $s_{0}$ and the LHS is nonpositive,
yielding, for $\left|x\right|\leq R$,
\begin{equation}
\min\left\{ s(x,t)-s_{0},0\right\} =0.
\end{equation}
As a direct result, we obtain, for $\left|x\right|\leq R$, 
\begin{equation}
s(x,t)\geq s_{0}=\underset{\left|x\right|\leq R+v_{\max}t}{\text{Ess inf}}s(x,0),\label{eq:minimum-entropy-principle}
\end{equation}
which is the minimum entropy principle for the compressible, multicomponent
Euler equations. Note that only the entropy inequalities associated
with $U=-\rho f_{\epsilon}(s)$ need to be satisfied for a minimum
entropy principle to hold.

\subsection{Minimum entropy principle in the compressible, multicomponent, reacting
Euler equations\label{subsec:minimum-entropy-principle-reacting}}

We now extend the result in the previous subsection to the reacting
Euler equations, where the only difference is the inclusion of the
chemical source terms (Equation~(\ref{eq:reacting-navier-stokes-source-term})).
The presence of the chemical source terms, which are smooth functions
of only the state variables, modifies the entropy inequality~(\ref{eq:entropy-condition})
satisfied by entropy solutions as~\citep{Cha10,Bou04,Rug89}
\begin{equation}
\frac{\partial U}{\partial t}+\nabla\cdot\mathcal{F}^{s}\leq\mathsf{v}^{T}\mathcal{S},\label{eq:entropy-condition-with-source}
\end{equation}
where the RHS represents the mathematical entropy production rate
due to the source term. If $\mathsf{v}^{T}\mathcal{S}\leq0$ (i.e.,
the entropy production is nonpositive), the entropy inequality in
(\ref{eq:entropy-condition}) for the homogeneous system is recovered.
The local inequality in (\ref{eq:entropy-inequality-integral-local})
can then be obtained by again integrating over $\mathsf{C}=\left\{ \left|x\right|\leq R+v_{\max}(t-\tau)|0\leq\tau\leq t\right\} $.
If, in particular, $\mathsf{v}^{T}\mathcal{S}\leq0$ for entropy functions
of the form $U=-\rho f_{\epsilon}(s),\:\forall\epsilon>0$, the remaining
arguments in Section~\ref{subsec:minimum-entropy-principle-nonreacting}
can be used to establish the same minimum entropy principle in Equation~(\ref{eq:minimum-entropy-principle})
for the reacting Euler equations. As such, we focus on showing $\mathsf{v}^{T}\mathcal{S}\leq0$
with $U=-\rho f_{\epsilon}(s),\:\forall\epsilon>0$. 

Consider again entropy/entropy-flux pairs of the form $\left(U,\mathcal{F}^{s}\right)=\left(-\rho f(s),-\rho vf(s)\right)$.
The corresponding entropy variables are given as
\[
\mathsf{v}=\begin{pmatrix}\frac{df}{ds}\frac{v_{1}}{T}\\
\vdots\\
\frac{df}{ds}\frac{v_{d}}{T}\\
-\frac{df}{ds}\frac{1}{T}\\
W_{1}\frac{df}{ds}\left(\frac{g_{1}-\frac{1}{2}\sum_{k=1}^{d}v_{k}v_{k}}{T}+s\right)-W_{1}f\\
\vdots\\
W_{n_{s}}\frac{df}{ds}\left(\frac{g_{n_{s}}-\frac{1}{2}\sum_{k=1}^{d}v_{k}v_{k}}{T}+s\right)-W_{n_{s}}f
\end{pmatrix},
\]
which differ slightly from the entropy variables derived by Gouasmi
et al.~\citep{Gou20} since they used partial densities instead of
species concentrations in the vector of state variables. The entropy
production rate due to chemical reactions, $\mathsf{v}^{T}\mathcal{S}$,
is then written as
\begin{align*}
\mathsf{v}^{T}\mathcal{S} & =\sum_{i=1}^{n_{s}}\left[W_{i}\frac{df}{ds}\left(\frac{g_{i}-\frac{1}{2}\sum_{k=1}^{d}v_{k}v_{k}}{T}+s\right)-W_{i}f\right]\omega_{i}\\
 & =\frac{df}{ds}\sum_{i=1}^{n_{s}}W_{i}\omega_{i}\frac{g_{i}}{T}+\left(-\frac{1}{2}\frac{\sum_{k=1}^{d}v_{k}v_{k}}{T}\frac{df}{ds}+s\frac{df}{ds}-f\right)\sum_{i=1}^{n_{s}}W_{i}\omega_{i}\\
 & =\frac{df}{ds}\sum_{i=1}^{n_{s}}W_{i}\omega_{i}\frac{g_{i}}{T},
\end{align*}
where the last equality is due to mass conservation, as given by Equation~(\ref{eq:chemical-reaction-mass-conservation}).
Since $\frac{df_{\epsilon}}{ds}>0,\:\forall\epsilon>0$, a minimum
entropy principle holds under the condition
\begin{equation}
\sum_{i=1}^{n_{s}}W_{i}\omega_{i}\frac{g_{i}}{T}\leq0.\label{eq:entropy-production-due-to-chemical-reactions-inequality}
\end{equation}
The term on the LHS, $\sum_{i=1}^{n_{s}}W_{i}\omega_{i}g_{i}/T$,
is precisely the entropy production rate for $U=-\rho s$, which Giovangigli~\citep[Chapter 6.4]{Gio99}
already showed to be nonpositive. For completeness, we review the
proof here.

Let $\mu$ and $\mu^{\mathrm{u}}$ denote the following vectors of
the reduced chemical potentials:
\begin{align*}
\mu & =\left(\mu_{1},\ldots,\mu_{n_{s}}\right)^{T},\\
\mu^{\mathrm{u}} & =\left(\mu_{1}^{\mathrm{u}},\ldots,\mu_{n_{s}}^{\mathrm{u}}\right)^{T}.
\end{align*}
We also define $\mathcal{W}$ as the matrix with the molecular weights
along the diagonal: 
\[
\mathcal{W}=\mathrm{diag}\left(W_{1},\ldots,W_{n_{s}}\right).
\]
The equilibrium constant can then be rewritten as
\begin{equation}
\log K_{j}^{e}=-\left(\mathcal{W}\nu_{j}\right)^{T}\mu^{\mathrm{u}}=\left(\mathcal{W}\nu_{j}^{f}\right)^{T}\mu^{\mathrm{u}}-\left(\mathcal{W}\nu_{j}^{r}\right)^{T}\mu^{\mathrm{u}},\label{eq:equilibrium-constant-3}
\end{equation}
where
\begin{align*}
\nu_{j}^{f} & =\left(\nu_{1j}^{f},\ldots,\nu_{n_{s}j}^{f}\right)^{T},\\
\nu_{j}^{r} & =\left(\nu_{1j}^{r},\ldots,\nu_{n_{s}j}^{r}\right)^{T},\\
\nu_{j} & =\nu_{j}^{r}-\nu_{j}^{f}.
\end{align*}
Using Equation~(\ref{eq:equilibrium-constant-3}), Giovangigli~\citep[Chapter 6.4]{Gio99}
introduced a new reaction constant, given by
\[
\log K_{j}^{s}=\log K_{j}^{f}-\left(\mathcal{W}\nu_{j}^{f}\right)^{T}\mu^{\mathrm{u}}=\log K_{j}^{r}-\left(\mathcal{W}\nu_{j}^{r}\right)^{T}\mu^{\mathrm{u}}.
\]
The rate of progress of the $j$th reaction can then be expressed
as
\begin{align*}
q_{j} & =K_{j}^{s}\left[\exp\left(\left(\mathcal{W}\nu_{j}^{f}\right)^{T}\mu^{\mathrm{u}}\right)\prod_{i=1}^{n_{s}}C_{i}^{\nu_{ij}^{f}}-\exp\left(\left(\mathcal{W}\nu_{j}^{r}\right)^{T}\mu^{\mathrm{u}}\right)\prod_{i=1}^{n_{s}}C_{i}^{\nu_{ij}^{r}}\right],\\
 & =K_{j}^{s}\left[\exp\left(\left(\mathcal{W}\nu_{j}^{f}\right)^{T}\mu\right)-\exp\left(\left(\mathcal{W}\nu_{j}^{r}\right)^{T}\mu\right)\right].
\end{align*}
We rewrite the entropy production rate for $U=-\rho s$ as
\begin{align*}
\sum_{i=1}^{n_{s}}W_{i}\omega_{i}\frac{g_{i}}{T} & =R^{0}\sum_{i=1}^{n_{s}}\sum_{j=1}^{n_{r}}W_{i}\mu_{i}\nu_{ij}q_{j}\\
 & =-R^{0}\sum_{j=1}^{n_{r}}\mu^{T}\mathcal{W}\left(\nu_{j}^{f}-\nu_{j}^{r}\right)q_{j}\\
 & =-R^{0}\sum_{j=1}^{n_{r}}K_{j}^{s}\left(\mu^{T}\mathcal{W}\nu_{j}^{f}-\mu^{T}\mathcal{W}\nu_{j}^{r}\right)\left[\exp\left(\mu^{T}\mathcal{W}\nu_{j}^{f}\right)-\exp\left(\mu^{T}\mathcal{W}\nu_{j}^{r}\right)\right],\\
 & \leq0.
\end{align*}
Note that this is equivalent to a nonnegative production rate of \emph{thermodynamic}
entropy per unit volume. Since the condition in~(\ref{eq:entropy-production-due-to-chemical-reactions-inequality})
is satisfied, a minimum entropy principle in the compressible, multicomponent,
reacting Euler equations holds. Another consequence of~(\ref{eq:entropy-production-due-to-chemical-reactions-inequality})
is that the entropy production rate for any entropy function of the
form $U=-\rho f(s)$ is nonpositive provided that $\frac{df}{ds}\geq0$.

\section{Transport step: Entropy-bounded discontinuous Galerkin scheme in
one dimension\label{sec:transport-step-EBDG}}

In this section, we detail the entropy-bounded DG methodology for
solving Equations~(\ref{eq:strang-splitting-1}) and~(\ref{eq:strang-splitting-3})
(i.e., the explicit time integrations without source terms in the
operator splitting strategy) while accounting for the modified flux
interpolation in Equation~(\ref{eq:modified-flux-projection}). We
build on related entropy-bounded DG schemes for the monocomponent
Euler equations~\citep{Zha12_2,Lv15_2,Jia18}. These schemes are
formulated as extensions of positivity-preserving DG methods~\citep{Zha10,Zha11,Zha17}
since the thermodynamic entropy is well-defined only for positive
densities and pressures and enforcement of an entropy constraint can
be straightforwardly incorporated into the positivity-preserving framework.
The one-dimensional entropy-bounded DG scheme is presented in Section~(\ref{subsec:entropy-bounded-DG-1D}),
then extended to multiple dimensions in Part II~\citep{Chi22_2}.

\subsection{Preliminaries~\label{subsec:preliminaries-EBDG-1D}}

Let $\mathcal{G}_{\sigma}$ denote the following set:
\begin{equation}
\mathcal{G_{\sigma}}=\left\{ y\mid C_{1}>0,\ldots,C_{n_{s}}>0,\rho u^{*}>0,s\geq\sigma\right\} ,
\end{equation}
where $\sigma\in\mathbb{R}$ and $u^{*}$ is the ``shifted'' internal
energy~\citep{Hua19}, computed as
\begin{equation}
u^{*}=u-u_{0}=u-\sum_{i=1}^{n_{s}}Y_{i}b_{i0},\label{eq:shifted-internal-energy}
\end{equation}
such that $u^{*}>0$ if and only if $T>0$, provided $c_{v,i}>0,\:i=1,\ldots,n_{s}$~\citep{Gio99}.
Pressure is then also positive. Note that $C_{i}>0,\;\forall i$,
implies $\rho>0$. This set is similar to that in~\citep{Du19,Du19_2},
except with the additional entropy constraint. In~\ref{sec:supporting-lemmas-specific-entropy},
we review two important lemmas from~\citep{Zha12_2}, the first of
which establishes the quasi-concavity of $s(y)$ (which holds in the
multicomponent case as well) and the second of which states $s\left(\overline{y}_{\kappa}\right)\geq\min_{x\in\kappa}s\left(y(x)\right)$.
In~\ref{sec:concavity-of-shifted-internal-energy}, we show that
$\rho u^{*}(y)$ is a concave function of the state. For a given $\sigma$,
$\mathcal{G}_{\sigma}$ is then a convex set~\citep{Zha11,Zha17,Hua19,Wu21_2}.
We assume that the exact solution to the classical Riemann problem
with initial data 
\[
y\left(x,0\right)=\begin{cases}
y_{1}, & x<0\\
y_{2}, & x>0
\end{cases}
\]
is an entropy solution that preserves positivity. Then, by Lemmas~\ref{lem:specific-entropy-of-solution-average}
and~\ref{lem:concavity-of-shifted-internal-energy}, $\mathcal{G}_{\sigma}$
is an invariant set~\citep{Gue19,Gue16_2}, i.e., $y_{1}\in\mathcal{G}_{\sigma}$
and $y_{2}\in\mathcal{G}_{\sigma}$ imply that the average of the
exact Riemann solution over a domain that includes the Riemann fan~\citep{Gue16_2}
is in $\mathcal{G}_{\sigma}$. 

Consider the following three-point system arising from a $p=0$, element-local
DG discretization with forward Euler time stepping:
\begin{gather}
y_{\kappa}^{j+1}=y_{\kappa}^{j}-\frac{\Delta t}{h}\left[\mathcal{F}^{\dagger}\left(y_{\kappa}^{j},y_{\kappa_{L}}^{j},-1\right)+\mathcal{F}^{\dagger}\left(y_{\kappa}^{j},y_{\kappa_{R}}^{j},1\right)\right],\label{eq:three-point-system}
\end{gather}
where $j$ indexes the time step, $h$ is the element size, and $\kappa_{L}$
and $\kappa_{R}$ are the elements to the left and right of $\kappa$,
respectively. Let $\lambda$ be an upper bound on the maximum wave
speed of the system. Under the following condition,
\begin{equation}
\frac{\Delta t\lambda}{h}\leq\frac{1}{2},\label{eq:p0-time-step-constraint}
\end{equation}
$y_{\kappa}^{j},y_{\kappa_{L}}^{j},y_{\kappa_{R}}^{j}\in\mathcal{G}_{\sigma}$
implies $y_{\kappa}^{j+1}\in\mathcal{G}_{\sigma}$ if certain \emph{invariant-region-preserving}
numerical fluxes are employed~\citep{Jia18}. In particular, $y_{\kappa}^{j+1}$
satisfies~\citep{Jia18,Wu21_2}
\begin{equation}
s\left(y_{\kappa}^{j+1}\right)\geq\min\left\{ s\left(y_{\kappa_{L}}^{j}\right),s\left(y_{\kappa}^{j}\right),s\left(y_{\kappa_{R}}^{j}\right)\right\} .\label{eq:p0-minimum-entropy-principle}
\end{equation}
The Godunov, Lax-Friedrichs, HLL, and HLLC fluxes qualify~\citep{Jia18}
(see also~\citep{Tad86,Lv15_2,Gou20,Per96} for additional proofs
regarding the Godunov and/or Lax-Friedrichs fluxes), some of which
allow for less restrictive time-step-size constraints than~(\ref{eq:p0-time-step-constraint}).
The proofs often rely on the notion that $\mathcal{G}_{\sigma}$ is
an invariant set, which itself invokes the aforementioned assumption
that the exact Riemann solution is an entropy solution satisfying
the positivity property. Two exceptions are the proof by Zhang and
Shu~\citep{Zha10} of the positivity property of the Lax-Friedrichs
flux and the proof by Lax~\citep{Lax71} that the Lax-Friedrichs
flux satisfies a discrete cell entropy inequality for all entropy/entropy-flux
pairs, from which a sharper version of the inequality~(\ref{eq:p0-minimum-entropy-principle})
follows~\citep{Tad86,Gou20} (note that the corresponding time-step-size
constraints for these two proofs are not necessarily the same as~(\ref{eq:p0-time-step-constraint})).
Unless otherwise specified, we employ the HLLC numerical flux~\citep{Tor13}. 

The three-point system~(\ref{eq:three-point-system}) will be crucial
in the construction of a positivity-preserving and entropy-bounded
DG scheme for $p>0$. Specifically, we will show in the following
subsection that the element average of the solution (for $p>0$) at
the next time step, $\overline{y}_{\kappa}^{j+1}$, can be expressed
as a convex combination of both pointwise values of $y_{\kappa}^{j}(x)$
and three-point systems involving pointwise values of $y_{\kappa}^{j}(x)$.
If all of said pointwise values of $y_{\kappa}^{j}(x)$ are in $\mathcal{G}_{\sigma}$,
then $\overline{y}_{\kappa}^{j+1}$ will also be in $\mathcal{G}_{\sigma}$
under a time-step-size constraint. A simple limiter, described in
Section~\ref{subsec:limiting-procedure}, is applied to ensure that
the pointwise values of $y_{\kappa}^{j}(x)$ are in $\mathcal{G}_{\sigma}$.

\subsection{Entropy-bounded, high-order discontinuous Galerkin method in one
dimension\label{subsec:entropy-bounded-DG-1D}}

Suppose $\kappa=\left[x_{L},x_{R}\right]$. Let $x_{q}$ and $w_{q}$
denote the quadrature points and weights, respectively, of a quadrature
rule with $x_{q}\in\kappa$ , $w_{q}>0$, and $\sum_{q=1}^{n_{q}}w_{q}=1$,
where $n_{q}$ is the number of quadrature points/weights. This set
of quadrature points does not need to include the endpoints, and the
quadrature rule need not be explicitly used to evaluate any integrals
in Equation~(\ref{eq:semi-discrete-form}). For now, we assume that
standard flux interpolation, as in Equation~(\ref{eq:flux-projection}),
is employed; Section~\ref{subsec:modified-flux-interpolation-1D}
discusses how to account for the modified flux interpolation in Equation~(\ref{eq:modified-flux-projection}).
As in~\citep{Lv15_2}, provided that the quadrature rule is sufficiently
accurate, the element-averaged solution in Equation~(\ref{eq:solution-element-average})
can be expanded as
\begin{align}
\overline{y}_{\kappa} & =\sum_{q=1}^{n_{q}}w_{q}y_{\kappa}(x_{q})\nonumber \\
 & =\sum_{q=1}^{n_{q}}\theta_{q}y_{\kappa}(x_{q})+\theta_{L}y_{\kappa}(x_{L})+\theta_{R}y_{\kappa}(x_{R}).\label{eq:element-average-convex-combination}
\end{align}
If the set of quadrature points includes the endpoints, then we can
simply take 
\[
\theta_{q}=\begin{cases}
w_{q} & x_{q}\neq x_{L},x_{q}\neq x_{R}\\
0 & \mathrm{otherwise}
\end{cases}
\]
and 
\[
\theta_{L}=w_{L},\quad\theta_{R}=w_{R},
\]
where $w_{L}$ and $w_{R}$ are the quadrature weights at the left
and right endpoints, respectively. If the set of quadrature points
does not include the endpoints, then we can instead take
\[
\theta_{q}=w_{q}-\theta_{L}\psi_{q}\left(x_{L}\right)-\theta_{R}\psi_{q}\left(x_{R}\right),
\]
where $\left\{ \psi_{1},\ldots,\psi_{n_{d}}\right\} $, with $n_{b}\leq n_{d}\leq n_{q}$,
is a set of Lagrange basis functions whose nodes are located at a
subset of the quadrature points, while $\psi_{v}=0$ for $v=n_{d}+1,\ldots,n_{q}$,
such that~\citep{Lv15_2}
\begin{align*}
\sum_{q=1}^{n_{q}}\theta_{q}y_{\kappa}(x_{q}) & =\sum_{q=1}^{n_{q}}\left[w_{q}-\theta_{L}\psi_{q}\left(x_{L}\right)-\theta_{R}\psi_{q}\left(x_{R}\right)\right]y_{\kappa}(x_{q})\\
 & =\sum_{q=1}^{n_{q}}w_{q}y_{\kappa}(x_{q})-\theta_{L}\sum_{q=1}^{n_{q}}y_{\kappa}(x_{q})\psi_{q}\left(x_{L}\right)-\theta_{R}\sum_{q=1}^{n_{q}}y_{\kappa}(x_{q})\psi_{q}\left(x_{R}\right)\\
 & =\sum_{q=1}^{n_{q}}w_{q}y_{\kappa}(x_{q})-\theta_{L}y_{\kappa}(x_{L})+\theta_{R}y_{\kappa}(x_{R}).
\end{align*}
$\theta_{L}$ and $\theta_{R}$ will be related to a constraint on
the time step size later in this section. The positivity of the quadrature
weights guarantees the existence of positive $\theta_{L}$ and $\theta_{R}$
that yield $\theta_{q}\geq0,\;q=1,\ldots,n_{q}$~\citep{Lv15_2}.
Furthermore, $\sum_{q}\theta_{q}+\theta_{L}+\theta_{R}=1$ since $\sum_{q}\psi_{q}=1$.
Define $\partial\mathcal{D}_{\kappa}=\left\{ x_{L,}x_{R}\right\} $,
and let $\mathcal{D}_{\kappa}$ denote the set of points at which
the state is evaluated in Equation~(\ref{eq:element-average-convex-combination}):
\[
\mathcal{D_{\kappa}}=\partial\mathcal{D}_{\kappa}\bigcup\left\{ x_{q},q=1,\ldots,n_{q}\right\} =\left\{ x_{L,}x_{R},x_{q},q=1,\ldots,n_{q}\right\} .
\]

Applying forward Euler time stepping to Equation~(\ref{eq:semi-discrete-form})
and taking $\mathfrak{v}$ to be a vector of ones (i.e., $\mathfrak{v}\in V_{h}^{0}$)
gives the fully discrete scheme satisfied by the element averages~\citep{Zha12_2,Lv15_2}:
\begin{eqnarray}
\overline{y}_{\kappa}^{j+1} & = & \overline{y}_{\kappa}^{j}-\frac{\Delta t}{h}\left[\mathcal{F}^{\dagger}\left(y_{\kappa}^{j}(x_{L}),y_{\kappa_{L}}^{j}(x_{L}),-1\right)+\mathcal{F}^{\dagger}\left(y_{\kappa}^{j}(x_{R}),y_{\kappa_{R}}^{j}(x_{R}),1\right)\right]\label{eq:fully-discrete-form-average-unexpanded}\\
 & = & \sum_{q=1}^{n_{q}}\theta_{q}y_{\kappa}^{j}(x_{q})+\theta_{L}y_{\kappa}^{j}(x_{L})-\frac{\Delta t}{h}\left[\mathcal{F}^{\dagger}\left(y_{\kappa}^{j}(x_{L}),y_{\kappa_{L}}^{j}(x_{L}),-1\right)+\mathcal{F}^{\dagger}\left(y_{\kappa}^{j}(x_{L}),y_{\kappa}^{j}(x_{R}),1\right)\right]\nonumber \\
 &  & +\theta_{R}y_{\kappa}^{j}(x_{R})-\frac{\Delta t}{h}\left[\mathcal{F}^{\dagger}\left(y_{\kappa}^{j}(x_{R}),y_{\kappa}^{j}(x_{L}),-1\right)+\mathcal{F}^{\dagger}\left(y_{\kappa}^{j}(x_{R}),y_{\kappa_{R}}^{j}(x_{R}),1\right)\right],\label{eq:fully-discrete-form-average}
\end{eqnarray}
where the second equality is due to the conservation property of the
numerical flux:
\[
\mathcal{F}^{\dagger}\left(y_{\kappa}^{j}(x_{L}),y_{\kappa}^{j}(x_{R}),1\right)=-\mathcal{F}^{\dagger}\left(y_{\kappa}^{j}(x_{R}),y_{\kappa}^{j}(x_{L}),-1\right).
\]
Note that Equations~(\ref{eq:fully-discrete-form-average-unexpanded})
and~(\ref{eq:fully-discrete-form-average}) hold regardless of whether
the integrals in Equation~(\ref{eq:semi-discrete-form}) are evaluated
using conventional quadrature or a quadrature-free implementation~\citep{Atk96,Atk98}.
Though the forward Euler time integration scheme is used here, strong-stability-preserving
Runge-Kutta (SSPRK) methods~\citep{Got01,Spi02}, which are convex
combinations of forward Euler steps, are compatible as well. Equation~(\ref{eq:fully-discrete-form-average})
then leads to the following theorem, where we use $y_{\kappa}^{-}$
to denote the exterior state along $\partial\kappa$.
\begin{thm}[\citep{Zha10,Zha12_2,Lv15_2}]
\label{thm:CFL-condition-1D}If $y_{\kappa}^{j}(x)\in\mathcal{G}_{\sigma},\;\forall x\in\mathcal{D_{\kappa}}$,
and $y_{\kappa}^{-,j}\in\mathcal{G}_{\sigma},\;\forall x\in\partial\mathcal{D}_{\kappa}$,
with
\begin{equation}
\sigma\leq\min\left\{ \min\left\{ s\left(y_{\kappa}^{j}(x)\right)\vert x\in\mathcal{D_{\kappa}}\right\} ,\min\left\{ s\left(y_{\kappa}^{-,j}(x)\right)\vert x\in\mathcal{\partial D_{\kappa}}\right\} \right\} ,\label{eq:sigma-definition}
\end{equation}
 then $\overline{y}_{\kappa}^{j+1}$ in Equation~(\ref{eq:fully-discrete-form-average-unexpanded})
is also in $\mathcal{G}_{\sigma}$ under the constraint
\begin{equation}
\frac{\Delta t\lambda}{h}\leq\frac{1}{2}\min\left\{ \theta_{L},\theta_{R}\right\} \label{eq:CFL-condition}
\end{equation}
and the conditions
\begin{equation}
\theta_{L}>0,\theta_{R}>0,\theta_{q}\geq0,q=1,\ldots,n_{q}.\label{eq:theta-conditions}
\end{equation}
\end{thm}

\begin{proof}
The proof follows the same procedure as in~\citep{Zha10},~\citep{Zha12_2},~\citep{Lv15_2},
and related papers, which we review here. We first rewrite Equation~(\ref{eq:fully-discrete-form-average})
as
\[
\overline{y}_{\kappa}^{j+1}=\sum_{q=1}^{n_{q}}\theta_{q}y_{\kappa}^{j}(x_{q})+\theta_{L}y_{\kappa,s1}^{j+1}+\theta_{R}y_{\kappa,s2}^{j+1},
\]
where
\begin{align*}
y_{\kappa,s1}^{j+1} & =y_{\kappa}^{j}(x_{L})-\frac{\Delta t}{\theta_{L}h}\left[\mathcal{F}^{\dagger}\left(y_{\kappa}^{j}(x_{L}),y_{\kappa_{L}}^{j}(x_{L}),-1\right)+\mathcal{F}^{\dagger}\left(y_{\kappa}^{j}(x_{L}),y_{\kappa}^{j}(x_{R}),1\right)\right],\\
y_{\kappa,s2}^{j+1} & =y_{\kappa}^{j}(x_{R})-\frac{\Delta t}{\theta_{R}h}\left[\mathcal{F}^{\dagger}\left(y_{\kappa}^{j}(x_{R}),y_{\kappa}^{j}(x_{L}),-1\right)+\mathcal{F}^{\dagger}\left(y_{\kappa}^{j}(x_{R}),y_{\kappa_{R}}^{j}(x_{R}),1\right)\right].
\end{align*}
As such, $\overline{y}_{\kappa}^{j+1}$ is a convex combination of
$y_{\kappa}^{j}(x_{q})$ evaluated at $x_{q}$ and two three-point
systems of the type~(\ref{eq:three-point-system}). Under the conditions~(\ref{eq:CFL-condition})
and~(\ref{eq:theta-conditions}), $y_{\kappa,s1}^{j+1}$ and $y_{\kappa,s2}^{j+1}$
are both in $\mathcal{G}_{\sigma}$. It then follows that $\overline{y}_{\kappa}^{j+1}\in\mathcal{G}_{\sigma}$.

\end{proof}
\begin{rem}
A direct result of Theorem~\ref{thm:CFL-condition-1D} is that
\[
s\left(\overline{y}_{\kappa}^{j+1}\right)\geq\min\left\{ \min\left\{ s\left(y_{\kappa}^{j}(x)\right)\vert x\in\mathcal{D_{\kappa}}\right\} ,\min\left\{ s\left(y_{\kappa}^{-,j}(x)\right)\vert x\in\mathcal{\partial D_{\kappa}}\right\} \right\} .
\]
 
\end{rem}

According to the inequality~(\ref{eq:CFL-condition}), the upper
bound on the time step size is proportional to $\min\left\{ \theta_{L},\theta_{R}\right\} $.
For Gauss-Lobatto quadrature, $\theta_{L}$ and $\theta_{R}$ are
both equal to the quadrature weight corresponding to either endpoint.
See~\citep{Lv15_2} for information about Gauss-Legendre quadrature,
as well as a discussion on how to find the maximum allowable value
of $\min\left\{ \theta_{L},\theta_{R}\right\} $ for general quadrature
rules. %

To complete the construction of an entropy-bounded, high-order DG
scheme, we need to enforce not only the positivity of $y_{\kappa}^{j+1}(x),\:\forall x\in\mathcal{D_{\kappa}}$,
for all $\kappa\in\Omega$, but also $s\left(y^{j+1}(x)\right)>s_{b},\:\forall x\in\mathcal{D_{\kappa}}$,
for all $\kappa\in\Omega$, where $s_{b}$ is a lower bound on the
specific thermodynamic entropy. $s_{b}$ can vary among elements and
time steps, and the prescription of $s_{b}$ should be motivated by
the physical principles examined in Section~\ref{sec:minimum-entropy-principle}.
We will discuss $s_{b}$ in more detail in Section~\ref{sec:entropy-bound}.
In other words, we enforce $y_{\kappa}^{j+1}(x)\in\mathcal{G}_{s_{b}},\:\forall x\in\mathcal{D_{\kappa}}$,
which is done via a simple limiting procedure that will be described
in Section~\ref{subsec:limiting-procedure}. %

In practice, we relax some of the requirements introduced thus far.
First, $\lambda$ is computed in a local (instead of global) manner
and is calculated as the maximum value of $\left|v\right|+c$, where
$c$ is the speed of sound, over the points of interest. However,
$\left|v\right|+c$ does not bound the wave speeds arising from the
interactions between states at interfaces. A similar remark can be
made for most invariant-region-preserving numerical flux functions,
which typically require wave-speed estimates. %
{} Simple algorithms for bounding the wave speeds in the monocomponent
case have been developed~\citep{Gue16,Tor20}; extending these to
the multicomponent Euler equations may indeed be worthy of future
investigation, with \citep{Fro19} as one example. Second, we introduce
$\chi_{\sigma}=\rho s-\rho\sigma$, which is concave with respect
to the state~\citep{Jia18}, and revise the definition of $\mathcal{G}_{\sigma}$
as
\begin{equation}
\mathcal{G}_{\sigma}=\left\{ y\mid\rho>0,\rho u^{*}>0,C_{1}\geq0,\ldots,C_{n_{s}}\geq0,\chi_{\sigma}\geq0\right\} ,
\end{equation}
where $\chi_{\sigma}\geq0$ is a reformulation of $s\geq\sigma$ and
the species concentrations are now allowed to be equal to zero. From
a practical standpoint, allowing $C_{i}=0$ is necessary since the
concentrations are frequently zero in many reacting flow problems
of interest. Unfortunately, entropy functions of the form $U=-\rho f_{\epsilon}(s)$
and $U=-\rho s$ are no longer convex if any of the concentrations
is zero~\citep{Gou20,Gou20_2}. Furthermore, the specific thermodynamic
entropy becomes ill-defined. Nevertheless, by making use of $0\log0=0$~\citep[Chapter 6]{Gio99},
$\rho s$ and thus $\chi_{\sigma}$ remain well-defined. The entire
methodology developed here also remains well-defined, unlike entropy-stable
schemes that rely on the entropy variables associated with $U=-\rho s$.
Throughout this work, we did not encounter any major issues associated
with relaxing the two aforementioned requirements. One potential reason
is that $\overline{y}_{\kappa}^{j+1}$ can be in $\mathcal{G}_{\sigma}$
even if $\mathcal{G}_{\sigma}$ is not convex and/or some of the conditions
in Theorem~\ref{thm:CFL-condition-1D} are not satisfied. Furthermore,
in this work, we choose $\mathrm{CFL}=0.1$ to maintain low temporal
errors, which in general yields a smaller time step size than necessary.
Should issues emerge in future work, they can likely be alleviated
by adaptively decreasing the time step size~\citep{Zha17}.

Finally, we remark that $\mathcal{D}_{\kappa}$ is simply the set
of points at which limiting should be applied. Specifically, the interior
quadrature points in $\mathcal{D}_{\kappa}$ need not be explicitly
used in numerical integrations in Equation~(\ref{eq:semi-discrete-form});
if they are indeed not, the actual integration points are added to
$\mathcal{D}_{\kappa}$ as well~\citep{Zha17}.

\subsubsection{Limiting procedure\label{subsec:limiting-procedure}}

In this subsection, we describe the limiting procedure to ensure $y_{\kappa}^{j+1}(x)\in\mathcal{G}_{s_{b}},\:\forall x\in\mathcal{D_{\kappa}}$.
It is assumed that $\overline{y}_{\kappa}^{j+1}(x)\in\mathcal{G}_{s_{b}}$.
For brevity, we drop the $j+1$ superscript and $\kappa$ subscript
in this discussion. The limiting operator is of the same form as in
\citep{Wan12}, \citep{Zha17}, \citep{Jia18}, \citep{Wu21_2}, and
related papers.
\begin{enumerate}
\item First, positivity of density is enforced. Specifically, if $\rho(x)>\epsilon,\:\forall x\in\mathcal{D}_{\kappa}$,
where $\epsilon$ is a small positive number (e.g., $\epsilon=10^{-10}$),
then set $C_{i}^{(1)}=C_{i}=\sum_{j=1}^{n_{b}}C_{i}(x_{j})\phi_{j},i=1,\ldots,n_{s}$;
otherwise, compute
\[
C_{i}^{(1)}=\overline{C}_{i}+\theta^{(1)}\left(C_{i}-\overline{C}_{i}\right),\;i=1,\ldots,n_{s},
\]
with
\[
\theta^{(1)}=\frac{\rho(\overline{y})-\epsilon}{\rho(\overline{y})-\underset{x\in\mathcal{D}}{\min}\rho(y(x))}.
\]
\item Next, nonnegativity of the species concentrations is enforced. If
$C_{i}^{(1)}(x)\geq0,\:\forall x\in\mathcal{D}_{\kappa}$, then set
$C_{i}^{(2)}=C_{i}^{(1)},i=1,\ldots,n_{s}$; otherwise, compute
\[
C_{i}^{(2)}=\overline{C}_{i}+\theta^{(2)}\left(C_{i}^{(1)}-\overline{C}_{i}\right),\;i=1,\ldots,n_{s},
\]
with
\[
\theta^{(2)}=\frac{\overline{C}_{i}}{\overline{C}_{i}-\underset{x\in\mathcal{D}}{\min}C_{i}^{(1)}(x)}.
\]
Let $y^{(2)}=\left(\rho v_{1},\ldots,\rho v_{d},\rho e_{t},C_{1}^{(2)},\ldots,C_{n_{s}}^{(2)}\right)$.
\item Positivity of $\rho u^{*}(y)$ is then enforced. If $\rho u^{*}\left(y^{(2)}(x)\right)>\epsilon,\:\forall x\in\mathcal{D}_{\kappa}$,
then set $y^{(3)}=y^{(2)}$; otherwise, compute
\[
y^{(3)}=\overline{y}+\theta^{(3)}\left(y^{(2)}-\overline{y}\right),
\]
with
\[
\theta^{(3)}=\frac{\rho u^{*}(\overline{y})-\epsilon}{\rho u^{*}(\overline{y})-\underset{x\in\mathcal{D}}{\min}\rho u^{*}(y^{(2)}(x))}.
\]
It can be shown that $\rho u^{*}(y^{(3)}(x))>0,\:\forall x\in\mathcal{D}_{\kappa}$
by concavity~\citep{Wan12,Zha17}. The ``positivity-preserving limiter''
refers to the limiting procedure up to this point. The ``entropy
limiter'' corresponds to the following step (in addition to the above
steps).
\item Finally, the entropy constraint is enforced. If $\chi\left(y^{(3)}(x)\right)\geq0,\:\forall x\in\mathcal{D}_{\kappa}$,
then set $y^{(4)}=y^{(3)}$; otherwise, compute
\[
y^{(4)}=\overline{y}+\theta^{(4)}\left(y^{(3)}-\overline{y}\right),
\]
with
\[
\theta^{(4)}=\frac{\chi(\overline{y})}{\chi(\overline{y})-\underset{x\in\mathcal{D}}{\min}\chi(y^{(3)}(x))}.
\]
It can be shown that $s\left(y^{(4)}(x)\right)\geq s_{b},\:\forall x\in\mathcal{D}_{\kappa}$,
by concavity of $\chi$~\citep{Jia18,Wu21_2}.
\end{enumerate}
$y^{(4)}$ then replaces $y$ as the solution. %
{} The limiting operator is conservative, maintains stability, and in
general preserves the formal order of accuracy for smooth solutions~\citep{Zha10,Zha17,Zha12_2,Lv15_2,Jia18}.
There is extensive empirical evidence demonstrating preservation of
accuracy (see previously cited references, as well as \citep{Du19_2},
\citep{Wu21_2}, and related papers). However, the order of accuracy
can potentially deteriorate when the element average is close to the
boundary of $\mathcal{G}_{s_{b}}$~\citep{Zha17,Jia18}. Furthermore,
the linear scaling is not expected to suppress all oscillations~\citep{Zha10,Jia18,Lv15_2,Wu21_2}.
The limiting procedure described here is applied at the end of every
RK stage. 

\subsubsection{Modified flux interpolation\label{subsec:modified-flux-interpolation-1D}}

We now discuss how to account for over-integration with the modified
flux interpolation in Equation~(\ref{eq:modified-flux-projection}).
The scheme satisfied by the element averages becomes
\begin{eqnarray}
\overline{y}_{\kappa}^{j+1} & = & \overline{y}_{\kappa}^{j}-\frac{\Delta t}{h}\left[\mathcal{F}^{\dagger}\left(\widetilde{y}_{\kappa}^{j}(x_{L}),\widetilde{y}_{\kappa_{L}}^{j}(x_{L}),-1\right)+\mathcal{F}^{\dagger}\left(\widetilde{y}_{\kappa}^{j}(x_{R}),\widetilde{y}_{\kappa_{R}}^{j}(x_{R}),1\right)\right]\nonumber \\
 & = & \sum_{q=1}^{n_{q}}\theta_{q}y_{\kappa}^{j}(x_{q})+\theta_{L}y_{\kappa}^{j}(x_{L})-\frac{\Delta t}{h}\left[\mathcal{F}^{\dagger}\left(\widetilde{y}_{\kappa}^{j}(x_{L}),\widetilde{y}_{\kappa_{L}}^{j}(x_{L}),-1\right)+\mathcal{F}^{\dagger}\left(\widetilde{y}_{\kappa}^{j}(x_{L}),\widetilde{y}_{\kappa}^{j}(x_{R}),1\right)\right]\nonumber \\
 &  & +\theta_{R}y_{\kappa}^{j}(x_{R})-\frac{\Delta t}{h}\left[\mathcal{F}^{\dagger}\left(\widetilde{y}_{\kappa}^{j}(x_{R}),\widetilde{y}_{\kappa}^{j}(x_{L}),-1\right)+\mathcal{F}^{\dagger}\left(\widetilde{y}_{\kappa}^{j}(x_{R}),\widetilde{y}_{\kappa_{R}}^{j}(x_{R}),1\right)\right].\label{eq:fully-discrete-form-average-modified}
\end{eqnarray}
Equation~(\ref{eq:fully-discrete-form-average-modified}) can be
rewritten as
\[
\overline{y}_{\kappa}^{j+1}=\sum_{q=1}^{n_{q}}\theta_{q}y_{\kappa}^{j}(x_{q})+\theta_{L}y_{\kappa,s3}^{j+1}+\theta_{R}y_{\kappa,s4}^{j+1},
\]
where
\begin{align*}
y_{\kappa,s3}^{j+1} & =y_{\kappa}^{j}(x_{L})-\frac{\Delta t}{\theta_{L}h}\left[\mathcal{F}^{\dagger}\left(\widetilde{y}_{\kappa}^{j}(x_{L}),\widetilde{y}_{\kappa_{L}}^{j}(x_{L}),-1\right)+\mathcal{F}^{\dagger}\left(\widetilde{y}_{\kappa}^{j}(x_{L}),\widetilde{y}_{\kappa}^{j}(x_{R}),1\right)\right],\\
y_{\kappa,s4}^{j+1} & =y_{\kappa}^{j}(x_{R})-\frac{\Delta t}{\theta_{R}h}\left[\mathcal{F}^{\dagger}\left(\widetilde{y}_{\kappa}^{j}(x_{R}),\widetilde{y}_{\kappa}^{j}(x_{L}),-1\right)+\mathcal{F}^{\dagger}\left(\widetilde{y}_{\kappa}^{j}(x_{R}),\widetilde{y}_{\kappa_{R}}^{j}(x_{R}),1\right)\right],
\end{align*}
which are not necessarily of the type~(\ref{eq:three-point-system})
since in general, $y_{\kappa}^{j}(x_{L})\neq\widetilde{y}_{\kappa}^{j}(x_{L})$
and $y_{\kappa}^{j}(x_{R})\neq\widetilde{y}_{\kappa}^{j}(x_{R})$.
The incompatibility is a result of expressing $\overline{y}_{\kappa}$
as a convex combination of pointwise values of $y_{\kappa}(x)$ (as
opposed to $\widetilde{y}_{\kappa}(x)$). Unfortunately, the element
average of $\widetilde{y}_{\kappa}$, denoted $\overline{\widetilde{y}}_{\kappa}$,
is not necessarily equal to $\overline{y}_{\kappa}$; consequently,
$\overline{y}_{\kappa}$ cannot be directly written as a convex combination
of pointwise values of $\widetilde{y}_{\kappa}(x)$. However, if the
set of solution nodes includes the endpoints (e.g., equidistant or
Gauss-Lobatto points), then Equation~(\ref{eq:fully-discrete-form-average-modified})
recovers Equation~(\ref{eq:fully-discrete-form-average}) since $y_{\kappa}^{j}(x_{L})=\widetilde{y}_{\kappa}^{j}(x_{L})$
and $y_{\kappa}^{j}(x_{R})=\widetilde{y}_{\kappa}^{j}(x_{R})$. The
previous analysis, including Theorem~\ref{thm:CFL-condition-1D},
then holds. As such, for a nodal set that includes the endpoints,
the modified flux interpolation in Equation~(\ref{eq:modified-flux-projection})
does not introduce any additional difficulties to the discussed framework.
Note also that the second term in Equation~(\ref{eq:semi-discrete-form})
(i.e., the volumetric flux integral) does not factor into the scheme
satisfied by the element averages, so the modified flux interpolation
can be freely employed in said integral.

\subsubsection{Lower bound on specific thermodynamic entropy\label{sec:entropy-bound}}

We consider two options for specifying $s_{b}$. The first option
is a \emph{global} entropy bound~\citep{Zha12_2,Jia18,Wu21_2}:
\begin{equation}
s_{b}(y)=\min\left\{ s\left(y(x)\right)\vert x\in\Omega\right\} ,\label{eq:global-entropy-bound-exact}
\end{equation}
which can be evaluated once based on the initial condition, $y_{0}(x)$,
or updated at each time step~\citep{Gou20}. Instead of calculating
the true minimum via, for example, Newton's method, we either rely
on user-specified information or compute
\begin{equation}
s_{b}(y)=\min\left\{ s\left(y(x)\right)\vert x\in\bigcup_{\kappa\in\mathcal{T}}\mathcal{D_{\kappa}}\right\} .\label{eq:global-entropy-bound}
\end{equation}

The second option is a \emph{local} entropy bound, which should satisfy
\[
s_{b,\kappa}^{j+1}(y)\leq\min\left\{ \min\left\{ s\left(y_{\kappa}^{j}(x)\right)\vert x\in\mathcal{D_{\kappa}}\right\} ,\min\left\{ s\left(y_{\kappa}^{-,j}(x)\right)\vert x\in\mathcal{\partial D_{\kappa}}\right\} \right\} ,
\]
in order to ensure compatibility with Theorem~\ref{thm:CFL-condition-1D}
and the limiting procedure, which enforces $y_{\kappa}^{j+1}(x)\in\mathcal{G}_{s_{b,\kappa}^{j+1}},\:\forall x\in\mathcal{D_{\kappa}}$.
Lv and Ihme~\citep{Lv15_2} introduced the following local entropy
bound:
\begin{equation}
s_{b,\kappa}^{j+1}(y)=\min\left\{ \min\left\{ s\left(y_{\kappa}^{j}(x)\right)\vert x\in\kappa\right\} ,\min\left\{ s\left(y_{\kappa}^{-,j}(x)\right)\vert x\in\partial\kappa\right\} \right\} .\label{eq:lv-local-entropy-bound}
\end{equation}
They demonstrated that the local entropy bound~(\ref{eq:lv-local-entropy-bound})
can more effectively dampen overshoots and undershoots when the entropy
varies significantly throughout the domain (e.g., when multiple discontinuities
are present). %
However, we find that Equation~(\ref{eq:lv-local-entropy-bound})
is often too restrictive, as will be illustrated in Section~\ref{subsec:thermal-bubble}. 

Here, we employ a different local entropy bound, 
\begin{equation}
s_{b,\kappa}^{j+1}(y)=\min\left\{ s\left(y^{j}(x)\right)\vert x\in\kappa\cup\kappa_{L}\cup\kappa_{R}\right\} ,\label{eq:local-entropy-bound-exact}
\end{equation}
which is based on the local minimum entropy principle satisfied by
exact entropy solutions (or discrete entropy solutions in the limit
of infinite resolution). Specifically, the inequality
\begin{equation}
s(y(x,t))\geq\min\left\{ s\left(y\left(x,t_{0}\right)\right)\vert x\in\kappa\cup\kappa_{L}\cup\kappa_{R}\right\} ,\quad t\in\left[t_{0},t_{0}+\Delta t\right],\label{eq:semi-discrete-minimum-entropy-principle}
\end{equation}
can be considered the semi-discrete analog of~(\ref{eq:minimum-entropy-principle})
with $R=h/2$ and $\kappa=\left[-h/2,h/2\right]$, under the condition

\begin{equation}
\frac{\Delta tv_{\max}}{h}\leq1.\label{eq:CFL-condition-local-entropy-bound}
\end{equation}
The RHS of Equation~(\ref{eq:local-entropy-bound-exact}) is the
fully discrete analog of the RHS of the inequality~(\ref{eq:semi-discrete-minimum-entropy-principle}).
This property is imposed on the discrete solution as a means to suppress
instabilities. A similar bound was employed by Dzanic and Witherden~\citep{Dza22}
in their entropy-based filtering framework. Note that if $\Delta t$
satisfies~(\ref{eq:CFL-condition}), then it also satisfies~(\ref{eq:CFL-condition-local-entropy-bound})
since
\[
\Delta t\leq\frac{h}{2\lambda}\min\left\{ \theta_{L},\theta_{R}\right\} \leq\frac{h}{2\lambda}\leq\frac{h}{\lambda}\leq\frac{h}{v_{\max}}.
\]

An alternate viewpoint draws from the generalized Riemann problem
(GRP)~\citep{Tor13}, which differs from the \emph{classical} Riemann
problem by allowing for source terms and piecewise smooth (as opposed
to piecewise constant) initial conditions. Suppose that $y_{\kappa_{L}}$
and $y_{\kappa}$ form the initial conditions of a GRP centered at
$x_{L}$, while $y_{\kappa}$ and $y_{\kappa_{R}}$ form the initial
conditions of a GRP centered at $x_{R}$. Under the (potentially strong)
assumption that exact solutions ($y_{L}^{\mathrm{GRP}}$ and $y_{R}^{\mathrm{GRP}}$,
respectively) to the GRPs exist, as well as the assumption that those
solutions satisfy all entropy inequalities, the exact solution in
$\kappa$, $y_{\kappa}^{\mathrm{ex}}$, arising from the GRPs satisfies~\citep{Har83_3}
\[
\int_{\kappa}U\left(y_{\kappa}^{\mathrm{ex}}(x,t)\right)\leq\int_{\kappa}U\left(y_{\kappa}(x,t_{0})\right)dx-\int_{t_{0}}^{t}\mathcal{F}^{s}\left(y_{R}^{\mathrm{GRP}}\left(x_{R},\tau\right)\right)d\tau+\int_{t_{0}}^{t}\mathcal{F}^{s}\left(y_{L}^{\mathrm{GRP}}\left(x_{L},\tau\right)\right)d\tau,
\]
at least before the local Riemann problems interact~\citep{Tor02}.
With $\left(U,\mathcal{F}^{s}\right)=\left(-\rho f_{0}(s),-\rho v_{1}f_{0}(s)\right)$
and~$s_{0}=\min\left\{ s\left(y\left(x,t_{0}\right)\right)\vert x\in\kappa\cup\kappa_{L}\cup\kappa_{R}\right\} $,
the above inequality becomes
\begin{eqnarray*}
\int_{\kappa}\rho\left(y_{\kappa}^{\mathrm{ex}}(x,t)\right)f_{0}\left(s\left(y_{\kappa}^{\mathrm{ex}}(x,t)\right)\right)dx & \geq & \int_{\kappa}\rho\left(y_{\kappa}(x,t_{0})\right)\min\left\{ s\left(y_{\kappa}(x,t_{0})\right)-s_{0},0\right\} dx\\
 &  & -\int_{t_{0}}^{t}\rho v_{1}\left(y_{R}^{\mathrm{GRP}}\left(x_{R},\tau\right)\right)\min\left\{ s\left(y_{R}^{\mathrm{GRP}}\left(x_{R},\tau\right)\right)-s_{0},0\right\} dx\\
 &  & +\int_{t_{0}}^{t}\rho v_{1}\left(y_{L}^{\mathrm{GRP}}\left(x_{L},\tau\right)\right)\min\left\{ s\left(y_{L}^{\mathrm{GRP}}\left(x_{L},\tau\right)\right)-s_{0},0\right\} dx,
\end{eqnarray*}
where the first term on the RHS clearly vanishes and the second and
third terms vanish due to Equation~(\ref{eq:minimum-entropy-principle}).
As such, we have $s\left(y_{\kappa}^{\mathrm{ex}}(x,t)\right)\geq s_{0}$.
This property is then imposed on the discrete solution via the local
entropy bound in Equation~(\ref{eq:local-entropy-bound-exact}).

To our knowledge, it is unclear whether the assumption that exact
solutions to the GRP exist is valid (even for calorically perfect
gases, let alone mixtures of thermally perfect gases). Semi-analytical
methods for solving GRPs have been developed; see~\citep[Chapter 19]{Tor13}
for a detailed description and~\citep{Mon12} for a comparison of
such methods. Discussions on exact solutions to the classical Riemann
problem for mixtures of thermally perfect gases can be found in~\citep{Bec01,Bec10}.

To avoid finding the true minimum over a given element via an iterative
procedure, we relax Equation~(\ref{eq:local-entropy-bound-exact})
and instead compute 
\begin{equation}
s_{b,\kappa}^{j+1}(y)=\mathsf{c}\min\left\{ s\left(y^{j}(x)\right)\vert x\in\mathcal{D}_{\kappa}\cup\mathcal{D}_{\kappa_{L}}\cup\mathcal{D}_{\kappa_{R}}\right\} ,\label{eq:local-entropy-bound}
\end{equation}
where $\mathsf{c}\in(0,1]$ is a relaxation parameter. The entropy
limiter in Section~\ref{subsec:limiting-procedure} remains valid
since
\[
s\left(\overline{y}_{\kappa}^{j+1}\right)\geq\min\left\{ s\left(y^{j}(x)\right)\vert x\in\mathcal{D_{\kappa}}\cup\mathcal{D}_{\kappa}^{-}\right\} \geq\min\left\{ s\left(y^{j}(x)\right)\vert x\in\mathcal{D}_{\kappa}\cup\mathcal{D}_{\kappa_{L}}\cup\mathcal{D}_{\kappa_{R}}\right\} \geq s_{b,\kappa}^{j+1}(y).
\]
For $\mathsf{c}<1$, it can be useful to simultaneously account for
a well-defined global entropy bound, $s_{b}$, as
\begin{equation}
s_{b,\kappa}^{j+1}(y)=\max\left\{ s_{b},\mathsf{c}\min\left\{ s\left(y^{j}(x)\right)\vert x\in\mathcal{D}_{\kappa}\cup\mathcal{D}_{\kappa_{L}}\cup\mathcal{D}_{\kappa_{R}}\right\} \right\} .\label{eq:local-global-entropy-bound}
\end{equation}
In this work, we simply choose $\mathsf{c}=1$. An alternative approach
for estimating the true minimum in an algebraic manner can be found
in~\citep{Lv15_2}. Note that Equation~(\ref{eq:local-entropy-bound})
with $\mathsf{c}=1$ yields the true minimum for $p=1$ if the element
endpoints are the solution nodes and included in $\mathcal{D}_{\kappa}$.
For $\kappa=\left[x_{L},x_{R}\right]$, the basis functions are given
by
\[
\phi_{1}(x)=\frac{x_{R}-x}{x_{R}-x_{L}},\phi_{2}(x)=\frac{x-x_{L}}{x_{R}-x_{L}},
\]
such that $\phi_{1}+\phi_{2}=1$ and $\phi_{i}(x)\in(0,1),\:\forall x\in\left(x_{L},x_{R}\right)$.
Therefore, $y_{\kappa}(x)=y_{\kappa}\left(x_{1}\right)\phi_{1}(x)+y_{\kappa}\left(x_{2}\right)\phi_{2}(x)$
is a convex combination of the endpoint values of the solution. By
Lemma~\ref{lem:quasi-concavity-specific-entropy}, $\min_{x}s\left(y_{\kappa}(x)\right)=\min\left\{ s\left(y_{\kappa}\left(x_{1}\right)\right),s\left(y_{\kappa}\left(x_{2}\right)\right)\right\} $.
This is similarly true in two and three dimensions.

\subsubsection{Artificial viscosity\label{subsec:artificial-viscosity}}

As will be demonstrated in Section~\ref{subsec:shock-tube}, the
proposed entropy-bounded DG method does not completely suppress smaller-scale
oscillations, especially in the presence of flow-field discontinuities.
Therefore, artificial viscosity is employed in certain test cases
in Section~\ref{sec:results-1D} to more effectively dampen such
oscillations. Specifically, the following dissipation term is added
to the LHS of Equation~(\ref{eq:semi-discrete-form})~\citep{Har13}:
\begin{equation}
-\sum_{\kappa\in\mathcal{T}}\left(\mathcal{F}^{\mathrm{AV}}\left(y,\nabla y\right),\nabla\mathfrak{v}\right)_{\kappa},\label{eq:artificial-viscosity-integral}
\end{equation}
where
\[
\mathcal{F}^{\mathrm{AV}}(y,\nabla y)=\nu_{\mathrm{AV}}\nabla y,
\]
with $\nu_{\mathrm{AV}}\geq0$ denoting the artificial viscosity.
The artificial viscosity is computed as~\citep{Joh20_2}
\[
\nu_{\mathrm{AV}}=\left(C_{\mathrm{AV}}+S_{\mathrm{AV}}\right)\left(\frac{h^{2}}{p+1}\left|\frac{\partial T}{\partial y}\cdot\frac{\mathcal{R}\left(y,\nabla y\right)}{T}\right|\right).
\]
{} $C_{\mathrm{AV}}$ is a user-defined coefficient, $S_{\mathrm{AV}}$
is a shock sensor based on intra-element variations~\citep{Chi19},
and $\mathcal{R}\left(y,\nabla y\right)$ is the strong form of the
residual~(\ref{eq:conservation-law-strong-form}). Note that Equation~(\ref{eq:fully-discrete-form-average-unexpanded})
(the scheme satisfied by the element averages) remains the same since
the dissipation term~(\ref{eq:artificial-viscosity-integral}) vanishes
for $\mathfrak{v}\in V_{h}^{0}$. Therefore, Theorem~\ref{thm:CFL-condition-1D}
still holds. This type of artificial viscosity was found to effectively
suppress spurious oscillations in the vicinity of flow-field discontinuities
in multicomponent reacting flows~\citep{Joh20_2}. However, we remark
that the artificial-viscosity formulation presented here is not the
focus of this paper; other types of artificial viscosity or limiters
can be employed to dampen the small-scale instabilities that the linear-scaling
limiter fails to cure, provided that the element-local averages are
unmodified.

\section{Reaction step: Entropy-stable discontinuous Galerkin method for ODE
integration\label{sec:entropy-stable-DGODE}}

In this section, we describe the entropy-stable DG discretization
of Equation~(\ref{eq:strang-splitting-2}) (i.e., the ordinary differential
equation (ODE) with stiff chemical source terms). We build on DGODE,
the (non-entropy-stable) DG method for ODE integration described in~\citep{Joh20_2}.
The local, semi-discrete integral form of Equation~(\ref{eq:strang-splitting-2})
is given by
\begin{gather}
\int_{\kappa}\mathfrak{v}^{T}\frac{\partial y}{\partial t}dx-\int_{\kappa}\mathfrak{v}^{T}\mathcal{S}(y)dx=0.\label{eq:semi-discrete-form-ode}
\end{gather}
Approximating $\mathcal{S}\left(y\right)$ locally as a polynomial
in $V_{h}^{p}$,
\[
\mathcal{S}_{\kappa}\approx\sum_{j=1}^{n_{b}}\mathcal{S}\left(y\left(x_{j}\right)\right)\phi_{j},
\]
we can write
\[
\frac{d}{dt}y_{\kappa}\left(x_{j},t\right)-\mathcal{S}\left(y_{\kappa}\left(x_{j},t\right)\right)=0,\quad j=1,\ldots,n_{b},
\]
which is a spatially decoupled system of ODEs advanced at the solution
nodes from $t=t_{0}$ to $t=t_{f}$. Our goal here is to ensure
\begin{align*}
y_{\kappa}\left(x_{j},t_{f}\right) & \in\mathcal{G}_{s\left(y_{\kappa}\left(x_{j},t_{0}\right)\right)},\quad j=1,\ldots,n_{b}.
\end{align*}
Assuming a Gauss-Lobatto nodal set, since $\overline{y}_{\kappa}(t_{f})$
is a convex combination of the nodal values, we have $\overline{y}_{\kappa}(t_{f})\in\mathcal{G}_{s_{b}}$,
where $s_{b}$ is now given by
\[
s_{b}=\min_{j=1,\ldots,n_{b}}s\left(y_{\kappa}\left(x_{j},t_{0}\right)\right).
\]
The limiting procedure described in Section~\ref{subsec:limiting-procedure}
can then be applied to enforce $y_{\kappa}\left(x,t_{f}\right)\in\mathcal{G}_{s_{b}},\:\forall x\in\mathcal{D}_{\kappa}$
(unless $\mathcal{D}_{\kappa}=\left\{ x_{j},j=1,\ldots,n_{b}\right\} $,
in which case the limiting procedure is unnecessary).

In the following, we drop the ``$\kappa$'' and ``$j$'' subscripts,
yielding
\begin{equation}
\frac{dy}{dt}-S(y)=0,\label{eq:ode}
\end{equation}
which is the system of ODEs solved at each node. Note that the formulation
described here is slightly different from that in~\citep{Lv15}.

\subsection{Review: DGODE\label{subsec:DGODE_review}}

We first briefly review DGODE, referred to as ``standard DGODE,''
which deals with the following one-dimensional DG discretization in
time of Equation~(\ref{eq:ode}):

\begin{gather}
N_{h}\left(y,\mathfrak{v}\right)=\sum_{\epsilon\in\mathcal{E}}\left(y^{\dagger}\left(y^{+},y^{-},n\right),\left\llbracket \mathfrak{v}\right\rrbracket \right)_{\mathcal{E}}-\sum_{\kappa\in\mathcal{T}}\left(y,\frac{d\mathfrak{v}}{dt}\right)_{\kappa}-\sum_{\kappa\in\mathcal{T}}\left(\mathcal{S}\left(y\right),\mathfrak{v}\right)_{\kappa}=0\qquad\forall\mathfrak{v}\in V_{h}^{p},\label{eq:dg-ode-weak-form-discretized}
\end{gather}
where $\epsilon$, $\mathcal{E}$, $\kappa$, $\mathcal{T}$, and
$V_{h}^{p}$ are temporal analogs of the spatial counterparts defined
in Section~\ref{sec:DG-discretization}. Specifically, $\mathcal{E}$
is the set of temporal interfaces $\epsilon$, $\mathcal{T}$ is the
set of cells $\kappa$ that partitions the computational domain, $\Omega=\left(t_{0},t_{f}\right)=\left(t_{0},t_{0}+\Delta t\right)$,
and $V_{h}^{p}$ is the discrete subspace defined similarly to Equation~(\ref{eq:discrete-subspace}).
Note that ``cell'' in this context corresponds to a sub-time-step
(referred to in this section as simply ``time step'') of size $h\in(0,\Delta t]$.
Equation~(\ref{eq:dg-ode-weak-form-discretized}) is obtained by
integrating Equation~(\ref{eq:ode}) over each cell, performing integration
by parts on the time-derivative terms, and summing over the domain.
On interior faces, the numerical temporal flux (henceforth referred
to as the ``numerical state''), $y^{\dagger}$, is defined as the
upwind flux function, 

\begin{align}
y^{\dagger}\left(y^{+},y^{-},n\right)=\begin{cases}
y^{+} & \textup{ if }n\geq0\\
-y^{-} & \textup{ if }n<0
\end{cases}, & \textup{ on }\epsilon\qquad\forall\epsilon\in\mathcal{E_{I}}.\label{eq:upwind-numerical-flux}
\end{align}
On exterior interfaces, the numerical state is defined as
\begin{align}
y^{\dagger}\left(y^{+},y^{-},n\right)=\begin{cases}
y^{+} & \textup{ if }n\geq0\\
-y_{\partial}\left(y^{+}\right) & \textup{ if }n<0
\end{cases}, & \textup{ on }\epsilon\qquad\forall\epsilon\in\mathcal{E}_{\partial},\label{eq:upwind-numerical-flux-boundary}
\end{align}
where $y_{\partial}\left(y^{+}\right)$ is a prescribed boundary state.
At the inflow interface, located at $t=t_{0}$, $y_{\partial}=y_{0}$,
which is the initial condition. At the outflow interface, located
at $t=t_{f}$, $y_{\partial}=y^{+}$ (i.e., no boundary condition
is imposed).

Let $m=\dim V_{h}^{p}$ and $\left(\phi_{1},\ldots,\phi_{m}\right)$
be a basis for $V_{h}^{p}$. The discrete residual, $\mathcal{R}=\left(\mathcal{R}_{1}\left(y\right),\ldots,\mathcal{R}_{m}\left(y\right)\right)$
is defined as 
\begin{equation}
\mathcal{R}_{i}\left(y\right)=\sum_{\epsilon\in\mathcal{E}}\left(y^{\dagger}\left(y^{+},y^{-},n\right),\left\llbracket \phi_{i}\right\rrbracket \right)_{\mathcal{E}}-\sum_{\kappa\in\mathcal{T}}\left(y,\frac{d\phi{}_{i}}{dt}\right)_{\kappa}-\sum_{\kappa\in\mathcal{T}}\left(\mathcal{S}\left(y\right),\phi_{i}\right)_{\kappa}\label{eq:dg-ode-residual}
\end{equation}
for $i=1,\ldots,m$. We can then recast~(\ref{eq:dg-ode-weak-form-discretized})
as
\[
\mathcal{R}\left(y\right)=0,
\]
which is solved for $y$ via Newton's method. With initial guess $y^{0}$,
the $k$th update, $y^{k}$, is computed by solving the linear system
\begin{equation}
\frac{d}{dy}\mathcal{R}\left(y^{k}\right)\left(y^{k+1}-y^{k}\right)=-\mathcal{R}\left(y^{k}\right)\label{eq:dg-ode-newtons-method}
\end{equation}
recursively until a convergence criterion is satisfied.

Johnson and Kercher~\citep{Joh20_2} introduced a simple yet effective
$hp$-adaptation strategy to control the time step, $h$, and polynomial
degree, $p$. They applied the strategy to efficiently and accurately
integrate chemical systems with complex mechanisms. Here, we focus
on $h$-adaptation with uniform $p$. W define the norm of the local
error estimate as
\[
\mathrm{err}_{h}=\left\Vert \frac{\mathcal{R}(y)}{\epsilon_{\mathrm{abs}}+\epsilon_{\mathrm{rel}}\left|y\right|}\right\Vert ,
\]
where $\epsilon_{\mathrm{abs}}$ and $\epsilon_{\mathrm{rel}}$ are
user-specified absolute and relative tolerances, respectively. %
The convergence criterion is
\begin{equation}
\mathrm{err}_{h}<1.\label{eq:convergence-criterion}
\end{equation}
If~(\ref{eq:convergence-criterion}) is satisfied within a user-specified
number of Newton iterations, the solution is updated and a new time
step is determined using Gustafsson's method~\citep{Hai96}. Otherwise,
the time step is reduced by a factor of ten.

In this work, given that implicit time stepping schemes with order
greater than one are typically not unconditionally positivity-preserving~\citep{For11},
we introduce an additional convergence criterion that improves stability
and more robustly maintains conservation of mass:
\begin{equation}
C_{i}\geq0,\quad i=1,\ldots,n_{s}.\label{eq:concentration-criterion-exact}
\end{equation}
Specifically, we require the solution at the end of each time step
to satisfy~(\ref{eq:concentration-criterion-exact}). In our experience,
this additional criterion does not significantly restrict the time
step size. %
{} Furthermore, the above criterion can be relaxed. Instead of requiring
pointwise\emph{ }nonnegativity of the species concentrations, we can
simply require that the spatial averages of the concentrations over
the element be nonnnegative. The positivity-preserving limiter in
Section~\ref{subsec:limiting-procedure} can then be applied to guarantee
nonnegativity of the species concentrations in a pointwise manner.
Though not pursued in this work, the conservative and positivity-preserving
projection method by Sandu~\citep{San01} can also be employed when~(\ref{eq:concentration-criterion-exact})
is violated. Another possible convergence criterion is $T>0$, but
this is almost never a concern.

\subsection{Entropy stability: Preliminaries}

Entropy stability in this ODE setting is defined as 
\begin{equation}
U\left(y\left(t_{f}\right)\right)\leq U\left(y_{0}\right).\label{eq:entropy-stability-definition}
\end{equation}
Here, we consider the entropy function $U=-\rho s$. Combined with
discrete mass conservation (i.e., $\rho$ is constant), we have
\begin{equation}
s\left(y\left(t_{f}\right)\right)\geq s\left(y_{0}\right).\label{eq:entropy-stability-specific-thermodynamic-entropy}
\end{equation}
In other words, entropy stability implies that the specific thermodynamic
entropy is nondecreasing in time, which is in line with the minimum
entropy principle. 

Standard DGODE, as in Section~\ref{subsec:DGODE_review}, does not
necessarily satisfy the inequality~(\ref{eq:entropy-stability-definition}).
In the following, we introduce a modified DG discretization that is
guaranteed to be entropy-stable.

\subsection{Entropy-stable DGODE with summation-by-parts property}

We work with the following strong-form discretization, obtained by
performing integration by parts on the temporal-derivative term in
Equation~(\ref{eq:dg-ode-weak-form-discretized}):

\begin{gather}
N_{h}\left(y,\mathfrak{v}\right)=\sum_{\kappa\in\mathcal{T}}\left[\left(y^{\dagger}\left(y^{+},y^{-},n\right)-n\cdot y^{+},\mathfrak{v}^{+}\right)_{\partial\kappa}+\left(\frac{dy}{dt},\mathfrak{v}\right)_{\kappa}\right]-\sum_{\kappa\in\mathcal{T}}\left(\mathcal{S}\left(y\right),\mathfrak{v}\right)_{\kappa}=0\qquad\forall\mathfrak{v}\in V_{h}^{p}.\label{eq:dg-ode-strong-form-discretized}
\end{gather}
It is well-known that a collocated DG scheme (i.e., solution nodes
and integration points are the same) with Gauss-Lobatto points possesses
the diagonal-norm summation-by-parts (SBP) property~\citep{Che17,Gas16}.
Consider the discrete mass matrix and discrete derivative matrix derived
from Gauss-Lobatto collocation, given by
\[
\mathsf{M}_{ij}=w_{i}\delta_{ij},\quad\mathsf{D}_{ij}=\ell'_{j}(\xi_{i}),
\]
where $w_{i}$ is the $i$th Gauss-Lobatto weight, $\ell_{j}$ is
the $j$th Lagrange basis polynomial, and $\xi\in[0,1]$ is the reference
coodinate. $\mathsf{M}$ and $\mathsf{D}$ satisfy the SBP property,
\[
\mathsf{Q}+\mathsf{Q}^{T}=\mathsf{B},
\]
where $\mathsf{Q}=\mathsf{MD}$ and $\mathsf{B}=\mathrm{diag}(-1,0,\ldots,0,1)$.
The following element-local discrete form is then obtained~\citep{Fri19}:
\begin{equation}
\begin{aligned} & \left[y^{\dagger}\left(y_{\kappa}^{+},y_{\kappa}^{-},n\right)-n\cdot y_{\kappa}^{+}\right]^{T}\mathfrak{v}_{\kappa}^{+}\biggr|_{\xi=0}+\left[y^{\dagger}\left(y_{\kappa}^{+},y_{\kappa}^{-},n\right)-n\cdot y_{\kappa}^{+}\right]^{T}\mathfrak{v}_{\kappa}^{+}\biggr|_{\xi=1}\\
 & +\sum_{i=1}^{n_{b}}w_{i}\left[\sum_{j=1}^{n_{b}}\mathsf{D}_{ij}y_{\kappa}(t_{j})-h\mathcal{S}\left(y_{\kappa}(t_{i})\right)\right]^{T}\mathfrak{v}_{\kappa}(t_{i})=0.
\end{aligned}
\label{eq:dg-ode-local-discrete-form}
\end{equation}
Note that in~\citep{Joh20_2}, standard DGODE was solved using a
quadrature-free approach~\citep{Atk96,Atk98}. The discrete form~(\ref{eq:dg-ode-local-discrete-form}),
though similar, specifically invokes quadrature in order to exploit
the SBP property. 

We replace the temporal-derivative interpolation operator in Equation~(\ref{eq:dg-ode-local-discrete-form})
with a specific temporal-derivative projection operator as~\citep{Fri19}
\[
\sum_{j=1}^{n_{b}}\mathsf{D}_{ij}y_{\kappa}(t_{j})\rightarrow2\sum_{j=1}^{n_{b}}\mathsf{D}_{ij}y^{\ddagger}(y_{\kappa}(t_{i}),y_{\kappa}(t_{j})),
\]
where $y^{\ddagger}(y_{1,}y_{2})$ is a two-point numerical state
function (distinct from $y^{\dagger}$) that is consistent and symmetric.
Equation~(\ref{eq:dg-ode-local-discrete-form}) thus becomes
\begin{equation}
\begin{aligned} & \left[y^{\dagger}\left(y_{\kappa}^{+},y_{\kappa}^{-},n\right)-n\cdot y_{\kappa}^{+}\right]^{T}\mathfrak{v}_{\kappa}^{+}\biggr|_{\xi=0}+\left[y^{\dagger}\left(y_{\kappa}^{+},y_{\kappa}^{-},n\right)-n\cdot y_{\kappa}^{+}\right]^{T}\mathfrak{v}_{\kappa}^{+}\biggr|_{\xi=1}\\
 & +\sum_{i=1}^{n_{b}}w_{i}\left[2\sum_{j=1}^{n_{b}}\mathsf{D}_{ij}y^{\ddagger}(y_{\kappa}(t_{i}),y_{\kappa}(t_{j}))-h\mathcal{S}\left(y_{\kappa}(t_{i})\right)\right]^{T}\mathfrak{v}_{\kappa}(t_{i})=0.
\end{aligned}
\label{eq:dg-ode-local-discrete-form-temporal-projection}
\end{equation}
With the mean-value numerical state, 
\begin{equation}
y^{\ddagger}(y_{L},y_{R})=\frac{y_{L}+y_{R}}{2},\label{eq:mean-value-numerical-state}
\end{equation}
Equation~(\ref{eq:dg-ode-local-discrete-form-temporal-projection})
recovers Equation~(\ref{eq:dg-ode-local-discrete-form})~\citep{Gas16,Che17}.

The discretization~(\ref{eq:dg-ode-local-discrete-form-temporal-projection})
combined with the mean-value numerical state is not necessarily entropy-stable.
Although, as will be shown in Section~\ref{subsec:discrete-temporal-entropy-analysis},
the source-term discretization results in destruction of mathematical
entropy, the temporal-derivative operator may cause production of
mathematical entropy. Entropy stability is achieved only if the entropy
destruction due to the source term outweighs the entropy production
due to the temporal-derivative operator. In the following, by replacing
the mean-value numerical state with a more appropriate numerical state,
we will obtain a temporal DG scheme that is guaranteed to be entropy-stable.

\subsubsection{Entropy-conservative numerical state function}

A key ingredient in the development of an entropy-stable DG scheme
is a two-point numerical state function that satisfies the following
condition~\citep{Fri19}:
\begin{equation}
(\mathsf{v}_{R}-\mathsf{v}_{L})^{T}y^{\ddagger}(y_{L},y_{R})=\mathcal{U}_{R}-\mathcal{U}_{L}.\label{eq:numerical-state-entropy-conservative-condition}
\end{equation}
A numerical state that satisfies~(\ref{eq:numerical-state-entropy-conservative-condition})
is \emph{entropy-conservative}. Note that an entropy-conservative
numerical \emph{flux} satisfies an analogous condition:
\begin{equation}
(\mathsf{v}_{R}-\mathsf{v}_{L})^{T}\mathcal{F}^{\ddagger}(y_{L},y_{R})=\mathcal{F}_{R}^{p}-\mathcal{F}_{L}^{p}.\label{eq:numerical-flux-entropy-conservative-condition}
\end{equation}
Gouasmi et al.~\citep{Gou20_2} derived a simple, closed-form entropy-conservative
numerical flux for the multicomponent Euler equations with the entropy
function $U=-\rho s$. They built on the techniques by Roe~\citep{Roe06}
originally used to construct an entropy-conservative numerical flux
for the monocomponent Euler equations. Said techniques rely on~(\ref{eq:numerical-flux-entropy-conservative-condition})
as a starting point. An analogous procedure, instead using~(\ref{eq:numerical-state-entropy-conservative-condition})
as a starting point, is employed here to derive an entropy-conservative
numerical state function. 

With $\left\llbracket \cdot\right\rrbracket $ denoting the jump operator,
the entropy conservation condition~(\ref{eq:numerical-state-entropy-conservative-condition})
can be expressed as
\begin{equation}
\left\llbracket \mathsf{v}\right\rrbracket {}^{T}y^{\ddagger}=\left\llbracket \mathcal{U}\right\rrbracket .\label{eq:numerical-state-entropy-conservative-condition-jump-form}
\end{equation}
For consistency with~\citep{Gou20_2}, we first work with a re-ordered
state vector where the species concentrations are replaced with partial
densities:
\begin{equation}
y=\left(\rho_{1},\ldots,\rho_{n_{s}},\rho v_{1},\ldots,\rho v_{d},\rho e_{t}\right)^{T}.\label{eq:state-vector-Gouasmi}
\end{equation}
We employ the notation
\[
y^{\ddagger}=\left(y_{1,1}^{\ddagger},\ldots,y_{1,n_{s}}^{\ddagger},y_{2,1}^{\ddagger},\ldots,y_{2,d}^{\ddagger},y_{3}^{\ddagger}\right),
\]
where the first $n_{s}$ components correspond to the partial densities,
the next $d$ components correspond to the momentum, and the last
component corresponds to the total energy. Let $z$ denote the vector
\[
z=\left(\rho_{1},\ldots,\rho_{n_{s}},v_{1},\ldots,v_{d},1/T\right)=\left(z_{1,1},\ldots,z_{1,n_{s}},z_{2,1},\ldots,z_{2,d},z_{3}\right).
\]
With the entropy function $U=-\rho s$, the entropy variables and
entropy potential are given by
\[
\mathsf{v}=\left(\frac{g_{1}-\frac{1}{2}\sum_{k=1}^{d}v_{k}v_{k}}{T},\ldots,\frac{g_{n_{s}}-\frac{1}{2}\sum_{k=1}^{d}v_{k}v_{k}}{T},\frac{v_{1}}{T},\ldots,\frac{v_{d}}{T},-\frac{1}{T},\right)^{T},\quad\mathcal{U}=\sum_{i=1}^{n_{s}}W_{i}\rho_{i}.
\]
Next, we introduce the arithmetic mean, logarithmic mean, and product
operators,
\begin{align}
\average{\alpha} & =\frac{\alpha_{L}+\alpha_{R}}{2},\nonumber \\
\alpha^{\ln} & =\begin{cases}
\alpha_{L}, & \text{if }\alpha_{L}=\alpha_{R}\\
0, & \text{if }\alpha_{L}=0\text{ or }\alpha_{R}=0\\
\frac{\alpha_{R}-\alpha_{L}}{\ln\alpha_{R}-\ln\alpha_{L}}, & \text{otherwise},
\end{cases}\label{eq:numerical-state-operators}\\
\alpha^{\times} & =\alpha_{L}\alpha_{R}\nonumber 
\end{align}
which are equipped with the identities
\begin{align*}
\average{\alpha\beta} & =\alpha\left\llbracket \beta\right\rrbracket +\beta\left\llbracket \alpha\right\rrbracket \\
\left\llbracket \ln\alpha\right\rrbracket  & =\frac{\left\llbracket \alpha\right\rrbracket }{\alpha^{\ln}}\\
\left\llbracket \alpha\right\rrbracket  & =-\alpha^{\times}\left\llbracket \frac{1}{\alpha}\right\rrbracket ,\alpha\neq0.
\end{align*}
Note that the three operators in~(\ref{eq:numerical-state-operators})
are symmetric; the first two are also consistent (the product operator
is consistent with $\alpha^{2}$). To compute the logarithmic mean
in a numerically stable manner, we employ the procedure by Ismail
and Roe~\citep{Ism09}. The jump in the entropy potential can then
be written as
\begin{equation}
\left\llbracket \mathcal{U}\right\rrbracket =\sum_{i=1}^{n_{s}}W_{i}\left\llbracket \rho_{i}\right\rrbracket =\sum_{i=1}^{n_{s}}W_{i}\left\llbracket z_{1,i}\right\rrbracket ,\label{eq:jump-in-entropy-potential}
\end{equation}
while the jump in the entropy variables, after some algebraic manipulation,
can be expressed as~\citep{Gou20_2}
\begin{equation}
\left\llbracket \mathsf{v}\right\rrbracket =\begin{pmatrix}\left\llbracket z_{3}\right\rrbracket \mathsf{X}_{1}+\left\llbracket z_{1,1}\right\rrbracket \frac{W_{1}}{\rho_{1}^{\ln}}-\sum_{k=1}^{d}\left\llbracket z_{2,k}\right\rrbracket \average{\frac{1}{T}}\average{v_{k}}\\
\vdots\\
\left\llbracket z_{3}\right\rrbracket \mathsf{X}_{n_{s}}+\left\llbracket z_{1,n_{s}}\right\rrbracket \frac{W_{n_{s}}}{\rho_{n_{s}}^{\ln}}-\sum_{k=1}^{d}\left\llbracket z_{2,k}\right\rrbracket \average{\frac{1}{T}}\average{v_{k}}\\
\average{z_{3}}\left\llbracket z_{2,1}\right\rrbracket +\average{z_{2,1}}\left\llbracket z_{3}\right\rrbracket \\
\vdots\\
\average{z_{3}}\left\llbracket z_{2,d}\right\rrbracket +\average{z_{2,d}}\left\llbracket z_{3}\right\rrbracket \\
-\left\llbracket z_{3}\right\rrbracket 
\end{pmatrix},\label{eq:jump-in-entropy-variables}
\end{equation}
where 
\[
\mathsf{X}_{i}=b_{i0}+\frac{b_{i1}}{\left(1/T\right)^{\ln}}+\sum_{r=2}^{n_{p}+1}b_{ir}\left(\mathsf{f}_{r-1}(T)T^{\times}\right)-\frac{1}{2}\sum_{k=1}^{d}\left\{ \left\{ v_{d}^{2}\right\} \right\} ,
\]
with $\mathsf{f}_{r}(\alpha)$ denoting a special averaging operator,
consistent with $\alpha^{r-1}$, that satisfies $\left\llbracket \alpha^{r}\right\rrbracket =r\mathsf{f}_{r}(\alpha)\left\llbracket \alpha\right\rrbracket $.
For $r=1,2,3,4$, $\mathsf{f}_{r}(\alpha)$ is defined as
\begin{align*}
\mathsf{f}_{1}(\alpha) & =1,\\
\mathsf{f}_{2}(\alpha) & =\average{\alpha},\\
\mathsf{f}_{3}(\alpha) & =\frac{2}{3}\average{\alpha}\average{\alpha}+\frac{1}{3}\average{\alpha^{2}},\\
\mathsf{f}_{4}(\alpha) & =\average{\alpha}\average{\alpha^{2}}.
\end{align*}
$\mathsf{f}_{r}(\alpha)$ can be derived for $r>4$ as well~\citep{Gou20_2}.
Using Equations~(\ref{eq:jump-in-entropy-potential}) and~(\ref{eq:jump-in-entropy-variables}),
we rewrite the entropy conservation condition in time~(\ref{eq:numerical-state-entropy-conservative-condition-jump-form})
as a requirement that a linear combination of the jumps in the components
of $z$ equals zero:
\begin{align*}
 & \sum_{i=1}^{n_{s}}\left\llbracket z_{1,i}\right\rrbracket \left(y_{1,i}^{\ddagger}\frac{W_{i}}{\rho_{i}^{\ln}}-W_{i}\right)+\sum_{k=1}^{d}\left\llbracket z_{2,k}\right\rrbracket \left(y_{2,k}^{\ddagger}\average{z_{3}}-\sum_{i=1}^{n_{s}}y_{1,i}^{\ddagger}\average{\frac{1}{T}}\average{v_{k}}\right)\\
 & +\left\llbracket z_{3}\right\rrbracket \left(\sum_{i=1}^{n_{s}}y_{1,i}^{\ddagger}\mathsf{X}_{i}+\sum_{k=1}^{d}y_{2,k}^{\ddagger}\average{z_{2,k}}-y_{3}^{\ddagger}\right)=0.
\end{align*}
Invoking the independence of these jumps yields a system of $m$ equations:
\begin{align*}
y_{1,i}^{\ddagger}\frac{W_{i}}{\rho_{i}^{\ln}}-W_{i} & =0,\quad i=1,\ldots,n_{s},\\
y_{2,k}^{\ddagger}\average{z_{3}}-\sum_{i=1}^{n_{s}}y_{1,i}^{\ddagger}\average{\frac{1}{T}}\average{v_{k}} & =0,\quad k=1,\ldots,d,\\
\sum_{i=1}^{n_{s}}y_{1,i}^{\ddagger}\mathsf{X}_{i}+\sum_{k=1}^{d}y_{2,k}^{\ddagger}\average{z_{2,k}}-y_{3}^{\ddagger} & =0.
\end{align*}
Solving for the components of $y^{\ddagger}$ then gives

\begin{equation}
y^{\ddagger}=\begin{pmatrix}\rho_{1}^{\ln}\\
\vdots\\
\rho_{n_{s}}^{\ln}\\
\sum_{i=1}^{n_{s}}\rho_{i}^{\ln}\average{v_{1}}\\
\vdots\\
\sum_{i=1}^{n_{s}}\rho_{i}^{\ln}\average{v_{d}}\\
\sum_{k=1}^{d}\sum_{i=1}^{n_{s}}\rho_{i}^{\ln}\average{v_{k}}\average{v_{k}}+\sum_{i=1}^{n_{s}}\rho_{i}^{\ln}\left[b_{i0}+\frac{b_{i1}}{\left(1/T\right)^{\ln}}+\sum_{r=2}^{n_{p}+1}b_{ir}\left(\mathsf{f}_{r-1}(T)T^{\times}\right)-\frac{1}{2}\sum_{k=1}^{d}\average{v_{k}^{2}}\right]
\end{pmatrix}.\label{eq:numerical-state-EC}
\end{equation}
Note that Equation~(\ref{eq:numerical-state-EC}) corresponds to
the state vector~(\ref{eq:state-vector-Gouasmi}). A simple re-ordering
and linear mapping~\citep{Gou20_2} yields

\begin{equation}
y^{\ddagger}=\begin{pmatrix}\sum_{i=1}^{n_{s}}\rho_{i}^{\ln}\average{v_{1}}\\
\vdots\\
\sum_{i=1}^{n_{s}}\rho_{i}^{\ln}\average{v_{d}}\\
\sum_{k=1}^{d}\sum_{i=1}^{n_{s}}\rho_{i}^{\ln}\average{v_{k}}\average{v_{k}}+\sum_{i=1}^{n_{s}}\rho_{i}^{\ln}\left[b_{i0}+\frac{b_{i1}}{\left(1/T\right)^{\ln}}+\sum_{r=2}^{n_{p}+1}b_{ir}\left(\mathsf{f}_{r-1}(T)T^{\times}\right)-\frac{1}{2}\sum_{k=1}^{d}\average{v_{k}^{2}}\right]\\
C_{1}^{\ln}\\
\vdots\\
C_{n_{s}}^{\ln}
\end{pmatrix},\label{eq:numerical-state-EC-transformed}
\end{equation}
which corresponds to the original state vector~(\ref{eq:reacting-navier-stokes-state}).
We then have the following theorem.
\begin{thm}
The two-point numerical state function in Equation~(\ref{eq:numerical-state-EC-transformed})
is entropy-conservative, consistent, and symmetric.
\end{thm}

\begin{proof}
The numerical state~(\ref{eq:numerical-state-EC-transformed}) is
entropy-conservative by construction. Since all of the introduced
operators are symmetric, the numerical state is symmetric. Finally,
recognizing that $\mathsf{f}_{r-1}(\alpha)\alpha^{\times}$ is consistent
with $\alpha^{r}$, taking the left and right states to be the same
yields
\[
y^{\ddagger}=\begin{pmatrix}\rho v_{1}\\
\vdots\\
\rho v_{d}\\
\sum_{k=1}^{d}\rho v_{k}v_{k}+\sum_{i=1}^{n_{s}}\rho_{i}\left[b_{i0}+b_{i1}T+\sum_{r=2}^{n_{p}+1}b_{ir}T^{r}-\frac{1}{2}\sum_{k=1}^{d}v_{k}^{2}\right]\\
C_{1}\\
\vdots\\
C_{n_{s}}
\end{pmatrix}=\begin{pmatrix}\rho v_{1}\\
\vdots\\
\rho v_{d}\\
\rho u+\frac{1}{2}\rho\sum_{k=1}^{d}v_{k}v_{k}\\
C_{1}\\
\vdots\\
C_{n_{s}}
\end{pmatrix}=y.
\]
Therefore, it is also consistent.
\end{proof}
\begin{rem}
The derived entropy-conservative numerical state function~(\ref{eq:numerical-state-EC-transformed})
can be directly used in entropy-stable space-time DG schemes based
on SBP operators~\citep{Fri19}.
\end{rem}

\subsubsection{\label{subsec:discrete-temporal-entropy-analysis}Discrete temporal
entropy analysis}

We analyze the entropy stability of the proposed temporal DG discretization
in the following theorem. 
\begin{thm}
Consider the DG discretization~(\ref{eq:dg-ode-local-discrete-form-temporal-projection}).
Assume that $y^{\dagger}$ is the upwind numerical state and $y^{\ddagger}$
is the entropy-conservative numerical state in~(\ref{eq:numerical-state-EC-transformed}).
Then, for the entropy function $U=-\rho s$, the resulting DG discretization
is entropy-stable.
\end{thm}

\begin{proof}
Take $\mathfrak{v}_{\kappa}$ in~(\ref{eq:dg-ode-local-discrete-form-temporal-projection})
to be the polynomial interpolant of the entropy variables, such that
\[
\mathfrak{v}_{\kappa}\left(t_{j}\right)=\mathsf{v}\left(y_{\kappa}\left(t_{j}\right)\right),\quad j=1,\ldots,n_{b},
\]
which results in
\begin{equation}
\begin{split} & \sum_{\kappa\in\mathcal{T}}\left[\left[y^{\dagger}\left(y_{\kappa}^{+},y_{\kappa}^{-},n\right)-n\cdot y_{\kappa}^{+}\right]^{T}\mathsf{v}\left(y_{\kappa}^{+}\right)\biggr|_{\xi=0}+\left[y^{\dagger}\left(y_{\kappa}^{+},y_{\kappa}^{-},n\right)-n\cdot y_{\kappa}^{+}\right]^{T}\mathsf{v}\left(y_{\kappa}^{+}\right)\biggr|_{\xi=1}\right]\\
 & +\sum_{\kappa\in\mathcal{T}}\left[\sum_{i=1}^{n_{b}}w_{i}\left(2\sum_{j=1}^{n_{b}}\mathsf{D}_{ij}y^{\ddagger}(y_{\kappa}(t_{i}),y_{\kappa}(t_{j}))-h\mathcal{S}\left(y_{\kappa}(t_{i})\right)\right)^{T}\mathsf{v}\left(y_{\kappa}(t_{i})\right)\right]=0.
\end{split}
\label{eq:dg-ode-local-discrete-form-temporal-projection-entropy-variables}
\end{equation}
We introduce $\mathcal{A}_{\kappa}$ and $\mathcal{B}_{\kappa}$,
defined as
\begin{align*}
\mathcal{A}_{\kappa}= & \left[y^{\dagger}\left(y^{+},y^{-},n\right)-n\cdot y^{+}\right]^{T}\mathsf{v}\left(y_{\kappa}^{+}\right)\biggr|_{\xi=0}+\left[y^{\dagger}\left(y^{+},y^{-},n\right)-n\cdot y^{+}\right]^{T}\mathsf{v}\left(y_{\kappa}^{+}\right)\biggr|_{\xi=1}\\
 & +\sum_{i=1}^{n_{b}}w_{i}\left[2\sum_{j=1}^{n_{b}}\mathsf{D}_{ij}y^{\ddagger}(y(t_{i}),y(t_{j}))\right]^{T}\mathsf{v}\left(y_{\kappa}(t_{i})\right)\\
\mathcal{B}_{\kappa}= & -\sum_{i=1}^{n_{b}}w_{i}hS\left(y_{\kappa}(t_{i})\right)^{T}\mathsf{v}\left(y_{\kappa}(t_{i})\right),
\end{align*}
such that Equation~(\ref{eq:dg-ode-local-discrete-form-temporal-projection-entropy-variables})
can be rewritten as
\begin{equation}
\sum_{\kappa\in\mathcal{T}}\mathcal{A}_{\kappa}+\sum_{\kappa\in\mathcal{T}}\mathcal{B}_{\kappa}=0.\label{eq:entropy-stability-proof-A-plus-B}
\end{equation}
By invoking the SBP property and the fact that the upwind numerical
state function is an entropy-stable numerical state (i.e., it satisfies
$\left\llbracket \mathsf{v}\right\rrbracket {}^{T}y^{\dagger}\leq\left\llbracket \mathcal{U}\right\rrbracket $)~\citep{Fri19},
we obtain the inequality
\begin{equation}
U\left(y\left(t_{f}\right)\right)-U\left(y_{0}\right)\leq\sum_{\kappa\in\mathcal{T}}\mathcal{A}_{\kappa},\label{eq:entropy-stability-proof-A-inequality}
\end{equation}
the proof of which is very similar to that in~\citep[Theorem 1]{Fri19}
(just without any spatial component) and is therefore not included
here. It remains to analyze $\sum_{\kappa\in\mathcal{T}}\mathcal{B}_{\kappa}$.
The quantity $S\left(y_{\kappa}(t_{i})\right)^{T}\mathsf{v}\left(y_{\kappa}(t_{i})\right)$
is simply the pointwise entropy production rate due to the chemical
source terms, which, as demonstrated in Section~\ref{subsec:minimum-entropy-principle-reacting},
is nonpositive. Since $w_{i}>0$ and $h>0$, we have 
\begin{equation}
\sum_{\kappa\in\mathcal{T}}\mathcal{B}_{\kappa}\geq0.\label{eq:entropy-stability-proof-B-inequality}
\end{equation}
Combining Equations~(\ref{eq:entropy-stability-proof-A-plus-B}),~(\ref{eq:entropy-stability-proof-A-inequality}),
and~(\ref{eq:entropy-stability-proof-B-inequality}) gives the inequality~(\ref{eq:entropy-stability-definition}),
which completes the proof.
\end{proof}
In the following, we refer to the proposed entropy-stable DG discretization
as simply ``entropy-stable DGODE.''
\begin{rem}
Entropy-stable DGODE is well-defined for zero concentrations. However,
during early iterations of Newton's method, negative concentrations
can occur, making the formulation ill-defined. As a simple remedy,
we use the solution obtained with the mean-value numerical state~(\ref{eq:mean-value-numerical-state})
as an initial guess, which we find to be sufficiently robust since
this initial guess is generally close to the entropy-stable DGODE
solution (recall that taking $y^{\ddagger}$ to be the mean-value
numerical state essentially recovers standard DGODE). Other, more
sophisticated approaches can be employed as well.
\end{rem}

\begin{rem}
In general, entropy-stable DGODE is more expensive than standard DGODE.
Furthermore, when using the latter, we find that the thermodynamic
entropy produced by the source terms typically outweighs any thermodynamic
entropy destruction caused by the non-entropy-stable mean-value numerical
state, at least for the reaction mechanisms used in this study (though
it is important to note that this observation may not hold true for
significantly different mechanisms). As such, in practice, the following
approach can be employed to maximize efficiency:
\end{rem}

\begin{itemize}
\item Compute a solution with standard DGODE. 
\item If $s\left(y\left(t_{f}\right)\right)\geq s\left(y_{0}\right)$, then
proceed. Otherwise, compute a solution with entropy-stable DGODE.
\end{itemize}
In Section~\ref{subsec:detonation-results-1D}, however, in order
to test the formulation, we calculate spatially one-dimensional detonation
waves with entropy-stable DGODE alone.
\begin{rem}
To reduce computational cost, it is common practice, especially for
large-scale simulations, to approximate a given chemical reaction
as two irreversible forward reactions (as opposed to having forward
and reverse reaction rates). However, this strategy can cause appreciable
entropy violations~\citep{Gio99,Sla11}. A detailed investigation
of whether the gains in speed outweigh the loss in accuracy is outside
the scope of this study. If such a strategy is employed, then reverting
to standard DGODE is the natural course of action. Though not considered
in this work, an alternative approach is to use the least-squares-based
method by Ream et al.~\citep{Rea18} to generate chemical mechanisms
involving irreversible reactions that do not cause entropy violations.
\end{rem}

\section{One-dimensional results\label{sec:results-1D}}

We consider three one-dimensional test cases. The first one involves
the advection of a thermal bubble. The second case is a shock-tube
problem with multiple flow discontinuities. These first two tests
comprise nonreacting multicomponent flows. The final case explores
sustained detonations formed via an overdriven initialization. Stiff
chemical reactions are present in this test. The SSPRK2 time integration
scheme with $\mathrm{CFL}=0.1$ (based on the linear-stability constraint)
is employed throughout. All simulations are performed using a modified
version of the JENRE\textregistered~Multiphysics Framework~\citep{Cor18_SCITECH,Joh20_2}
that incorporates the developments and extensions described in this
work.

\subsection{Thermal bubble\label{subsec:thermal-bubble}}

We use this smooth flow problem to assess the grid convergence of
the entropy-bounded DG method (without artificial viscosity). The
order of accuracy of the limiting procedure in Section~\ref{subsec:limiting-procedure}
is of particular interest. We also compare the local entropy bound
in~\citep{Lv15_2} to that proposed here (Equation~(\ref{eq:local-entropy-bound})).
The initial conditions are as follows:

\begin{eqnarray}
v_{1} & = & 1\textrm{ m/s},\nonumber \\
Y_{H_{2}} & = & \frac{1}{2}\left[1-\tanh\left(|x|-10\right)\right],\nonumber \\
Y_{O_{2}} & = & 1-Y_{H_{2}},\label{eq:thermal-bubble}\\
T & = & 1200-900\tanh\left(|x|-10\right)\textrm{ K},\nonumber \\
P & = & 1\textrm{ bar}.\nonumber 
\end{eqnarray}
The computational domain is $\Omega=[-25,25]\:\mathrm{m}$, with periodic
boundaries. Four element sizes are considered: $h$, $h/2$, $h/4$,
and $h/8$, where $h=2\:\mathrm{m}$. Smaller element sizes are not
investigated since the limiters are not activated for such fine resolutions
and optimal order of accuracy without limiting was already demonstrated
in~\citep{Joh20_2}. The $L^{2}$ error at $t=5\:\mathrm{s}$ is
calculated in terms of the following normalized state variables:
\[
\widehat{\rho v}_{k}=\frac{1}{\sqrt{\rho_{r}P_{r}}}\rho v_{k},\quad\widehat{\rho e}_{t}=\frac{1}{P_{r}}\rho e_{t},\quad\widehat{C}_{i}=\frac{R^{0}T_{r}}{P_{r}}C_{i},
\]
where $T_{r}=1000\,\mathrm{K}$, $\rho_{r}=1\,\mathrm{kg\cdot}\mathrm{m}^{-3}$,
and $P_{r}=101325\,\mathrm{Pa}$ are reference values. The results
of the convergence study for $p=1$ to $p=3$ are displayed in Figure~\ref{fig:thermal_bubble_convergence}.
The dashed lines represent the theoretical convergence rates. The
``$\times$'' marker indicates that the positivity-preserving limiter
is activated, the ``$\Circle$'' marker indicates that the entropy
limiter is activated, and the ``$\triangle$'' marker indicates
that neither limiter is activated. If both limiters are activated,
then the corresponding markers are superimposed as ``$\otimes$''.
For $h$ and $h/2$, both the positivity-preserving and entropy limiters
are engaged regardless of $p$; for $h/4$ and $p=1$, only the positivity-preserving
limiter is engaged. That the limiters are no longer activated for
well-resolved solutions is a desirable property. In general, optimal
order of accuracy is recovered. Suboptimal accuracy is observed for
the coarser resolutions with $p=1$, likely because the asymptotic
regime is not yet reached. 
\begin{figure}[H]
\begin{centering}
\includegraphics[width=0.6\columnwidth]{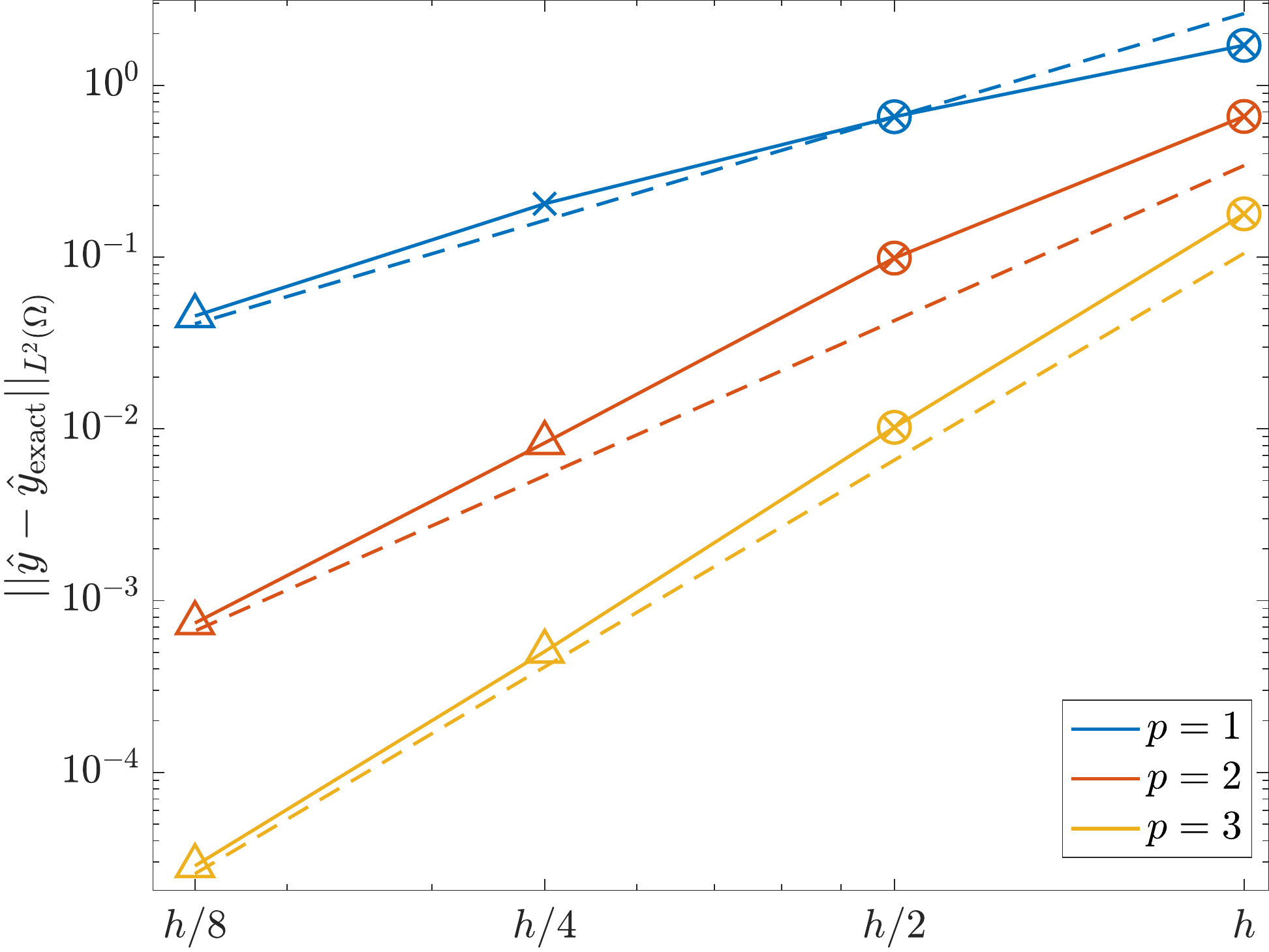}
\par\end{centering}
\caption{\label{fig:thermal_bubble_convergence}Convergence under grid refinement,
with $h=2\:\mathrm{m}$, for the one-dimensional thermal bubble test
case. The $L^{2}$ error of the normalized state with respect to the
exact solution at $t=5\:\mathrm{s}$ is computed. The dashed lines
represent the theroretical convergence rates. The ``$\times$''
marker indicates that the positivity-preserving limiter is activated,
the ``$\Circle$'' marker indicates that the entropy limiter is
activated, and the ``$\triangle$'' marker indicates that neither
limiter is activated. If both limiters are activated, then the corresponding
markers are superimposed as ``$\otimes$''.}
\end{figure}

Figure~\ref{fig:thermal_bubble_T_local_EB} compares the temperature
profile at $t=50\:s$ computed with the local entropy bound in~\citep{Lv15_2}
(``Old'') to that computed with the local entropy bound in Equation~(\ref{eq:local-entropy-bound})
(``New''). The exact solution is the same as the initial condition.
The element size is $h/4$ and the polynomial order is $p=1$. The
reason for choosing $p=1$ is to eliminate any ambiguity associated
with algebraically estimating the spatial minimum of the specific
thermodynamic entropy, as discussed in Section~\ref{sec:entropy-bound}.
Though oscillations are present in both cases, they are significantly
larger with the local entropy bound in~\citep{Lv15_2}, suggesting
that it may be overly restrictive. The local Lax-Friedrichs flux is
employed in this comparison. With the selected parameters, the differences
between the two local entropy bounds are more pronounced for this
numerical flux than the HLLC flux. Nevertheless, such discrepancies
are observed for the latter flux function in other configurations
as well, especially with higher polynomial orders.
\begin{figure}[H]
\begin{centering}
\includegraphics[width=0.6\columnwidth]{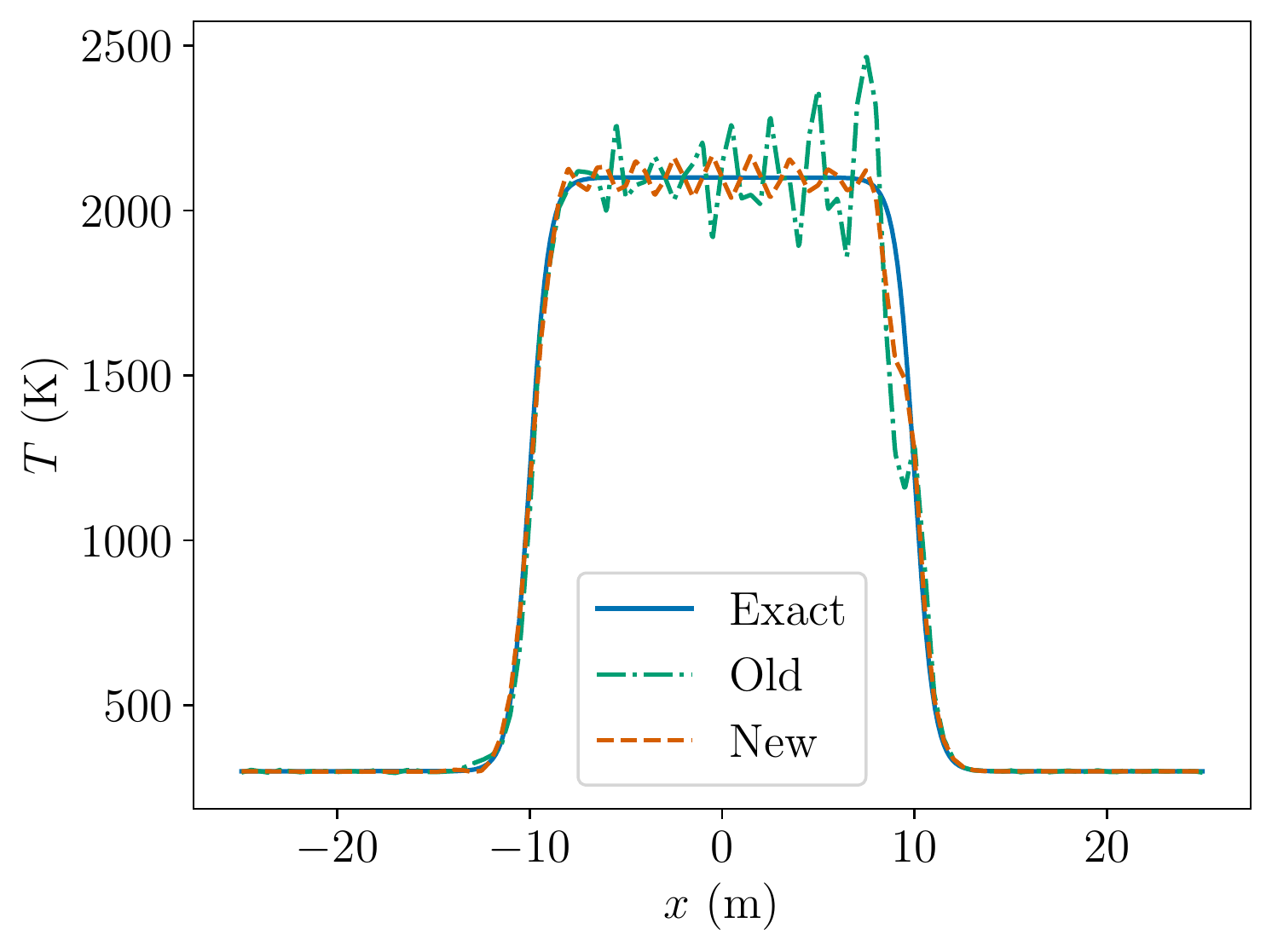}
\par\end{centering}
\caption{\label{fig:thermal_bubble_T_local_EB}Temperature profiles at $t=50\:s$
computed with the local entropy bound in~\citep{Lv15_2} (``Old'')
and the local entropy bound in Equation~(\ref{eq:local-entropy-bound})
(``New''). The exact solution is the same as the initial condition.
The element size is $h/4$ and the polynomial order is $p=1$.}
\end{figure}

\subsection{Shock tube\label{subsec:shock-tube}}

This test case was first presented by Houim and Kuo~\citep{Hou11}
and computed as well by Johnson and Kercher~\citep{Joh20_2}. The
goals here are to compare the stabilization properties of the limiters
and to illustrate the benefits of the local entropy bound. The initial
conditions are given by
\begin{equation}
\left(v_{1},T,P,Y_{N_{2}},Y_{He}\right)=\begin{cases}
\left(0\text{ m/s},300\text{ K},1\text{ atm},1,0\right), & x\geq0.4\\
\left(0\text{ m/s},300\text{ K},10\text{ atm},0,1\right), & x<0.4
\end{cases}.\label{eq:shock-tube-IC-Houim}
\end{equation}
The computational domain is $\Omega=[0,1]\text{ m}$, with walls at
both ends. Figure~\ref{fig:shock_tube_Houim} shows the mass fraction,
pressure, temperature, and entropy profiles at $t=300\;\mu\mathrm{s}$
for $p=3$ and 200 elements. We present results for only the positivity-preserving
limiter (referred to as ``PPL'') and for both the positivity-preserving
and entropy limiters with the local entropy bound in Equation~(\ref{eq:local-entropy-bound})
(referred to as ``local EB''). Also included is a reference solution
computed with $p=2$, 2000 elements, and artificial viscosity, which
corresponds to the configuration in~\citep{Joh20_2}. Artificial
viscosity is not employed in the coarser cases in order to isolate
the effects of the limiters. The mass fraction profiles are well-captured
for both types of limiting. However, although the solver does not
crash, the positivity-preserving limiter by itself fails to suppress
large-scale oscillations and significant overshoots/undershoots in
the pressure and temperature distributions. While instabilities are
still present with the local entropy limiter, they are of substantially
smaller magnitude. The entropy distribution obtained with the local
entropy limiter is very similar to that of the reference solution,
whereas the positivity-preserving limiter generates a notable overshoot
and undershoot at the shock. For the given flow conditions, the global
entropy bound yields very similar results (not shown for brevity)
to the local entropy bound since the specific thermodynamic entropy
in the vicinity of the shock, which is where much of the limiting
occurs, is close to the global minimum. 

The instabilities observed in the ``PPL'' case in Figures~\ref{fig:shock_tube_P_Houim}
and~\ref{fig:shock_tube_T_Houim} are considerably larger than those
typically observed in shock-tube solutions obtained with the positivity-preserving
limiter in the monocomponent, calorically perfect case~\citep{Zha10,Zha12_2,Zha17}.
This difference reflects the numerical challenges associated with
not only multicomponent mixtures, but also variable thermodynamics.
In a similar vein, the relative benefit of the entropy limiter (compared
to the positivity-preserving limiter) seems significantly greater
in the multicomponent, thermally perfect case than in the monocomponent,
calorically perfect case. 

\begin{figure}[H]
\subfloat[\label{fig:shock_tube_Y_Houim}Mass fractions.]{\includegraphics[width=0.45\columnwidth]{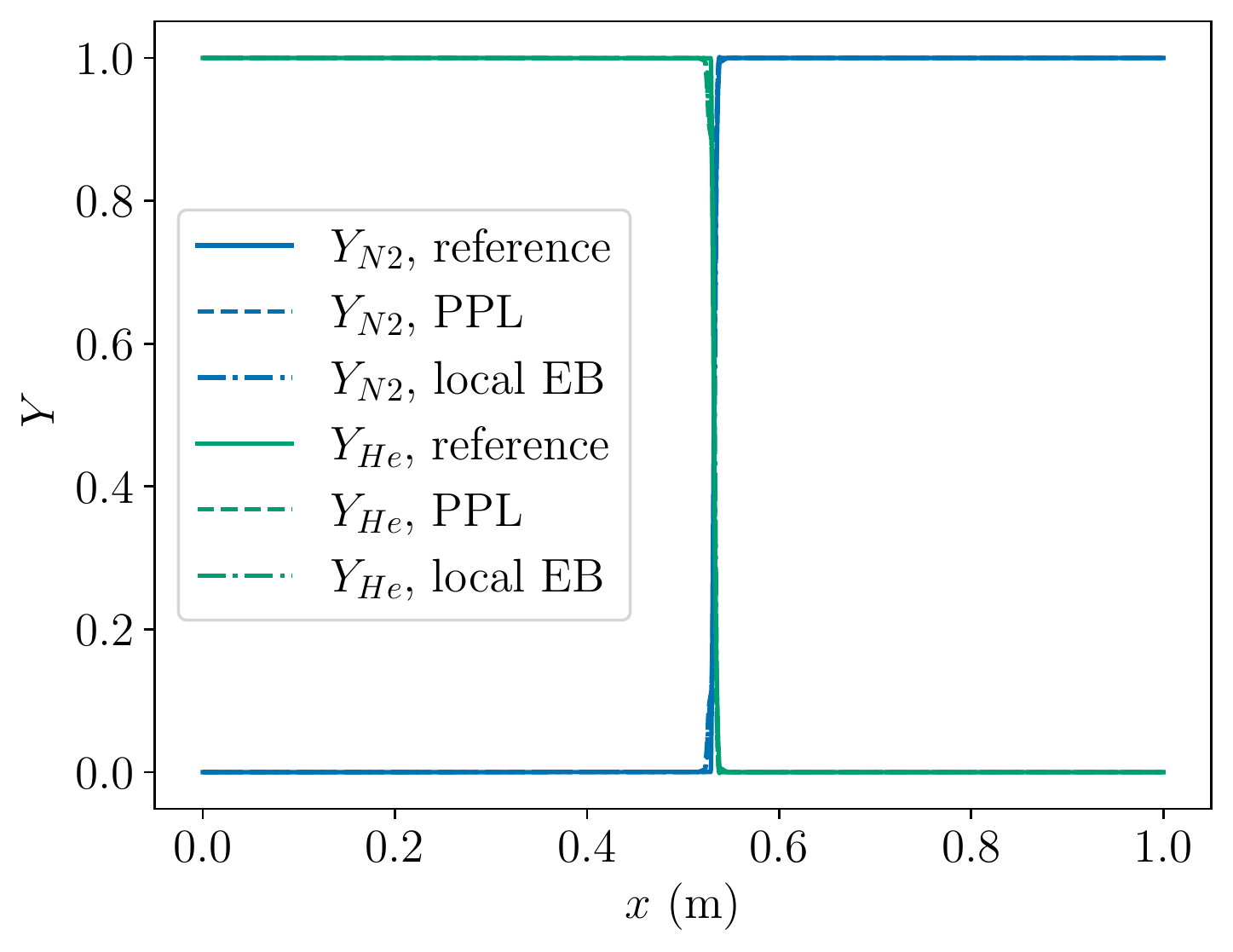}}\hfill{}\subfloat[\label{fig:shock_tube_P_Houim}Pressure.]{\includegraphics[width=0.45\columnwidth]{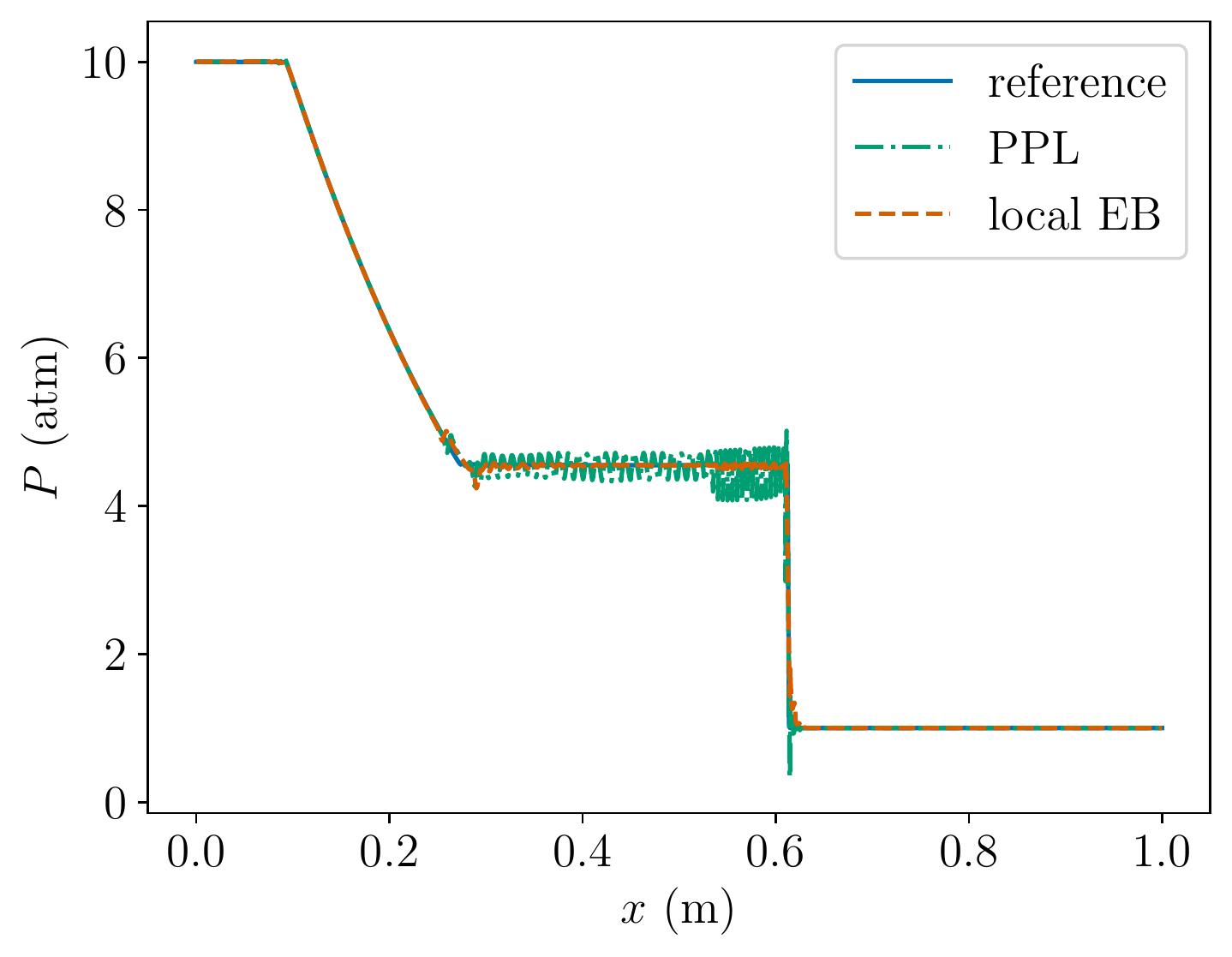}}\hfill{}\subfloat[\label{fig:shock_tube_T_Houim}Temperature.]{\includegraphics[width=0.48\columnwidth]{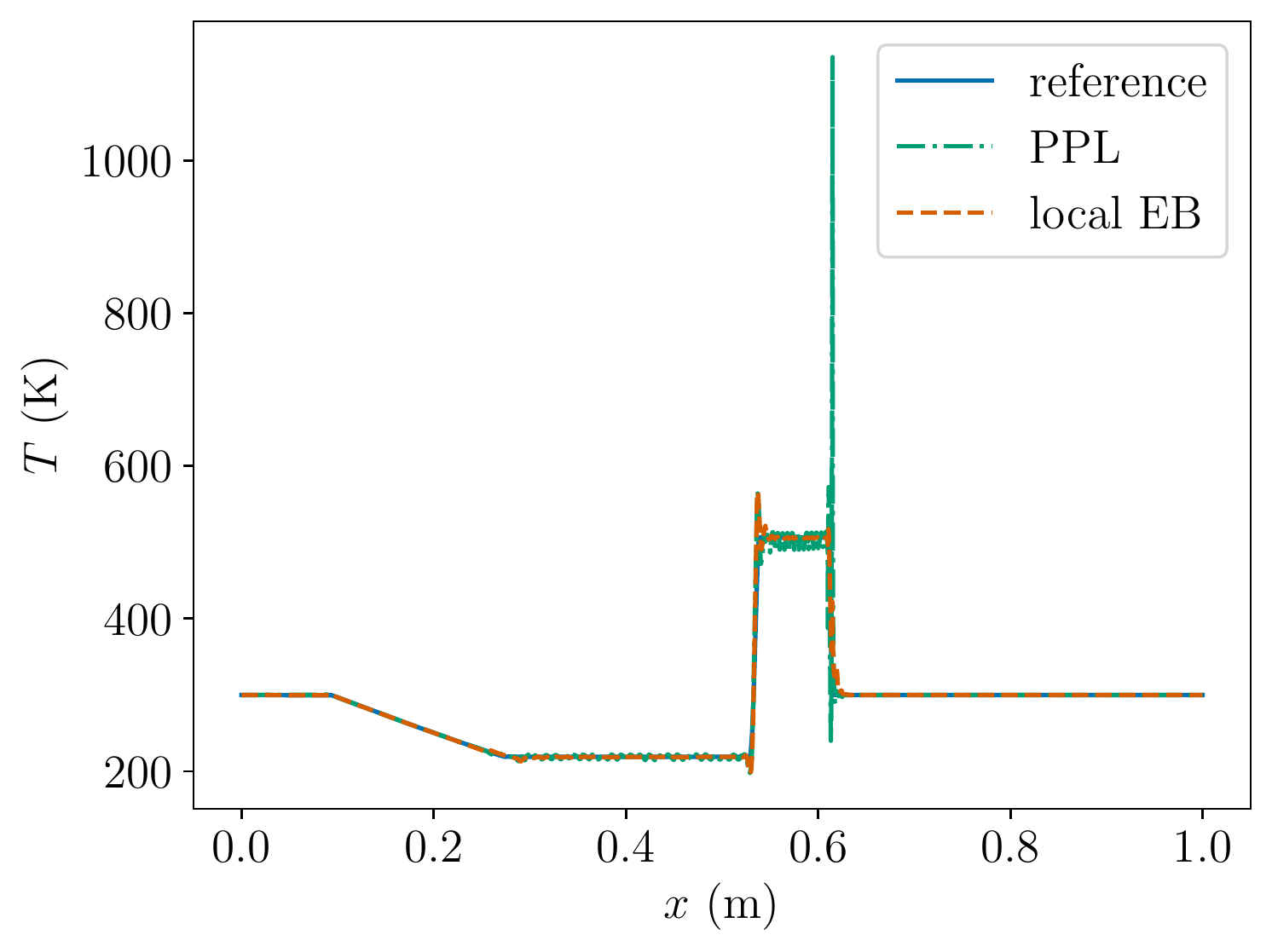}}\hfill{}\subfloat[\label{fig:shock_tube_s_Houim}Specific thermodynamic entropy.]{\includegraphics[width=0.45\columnwidth]{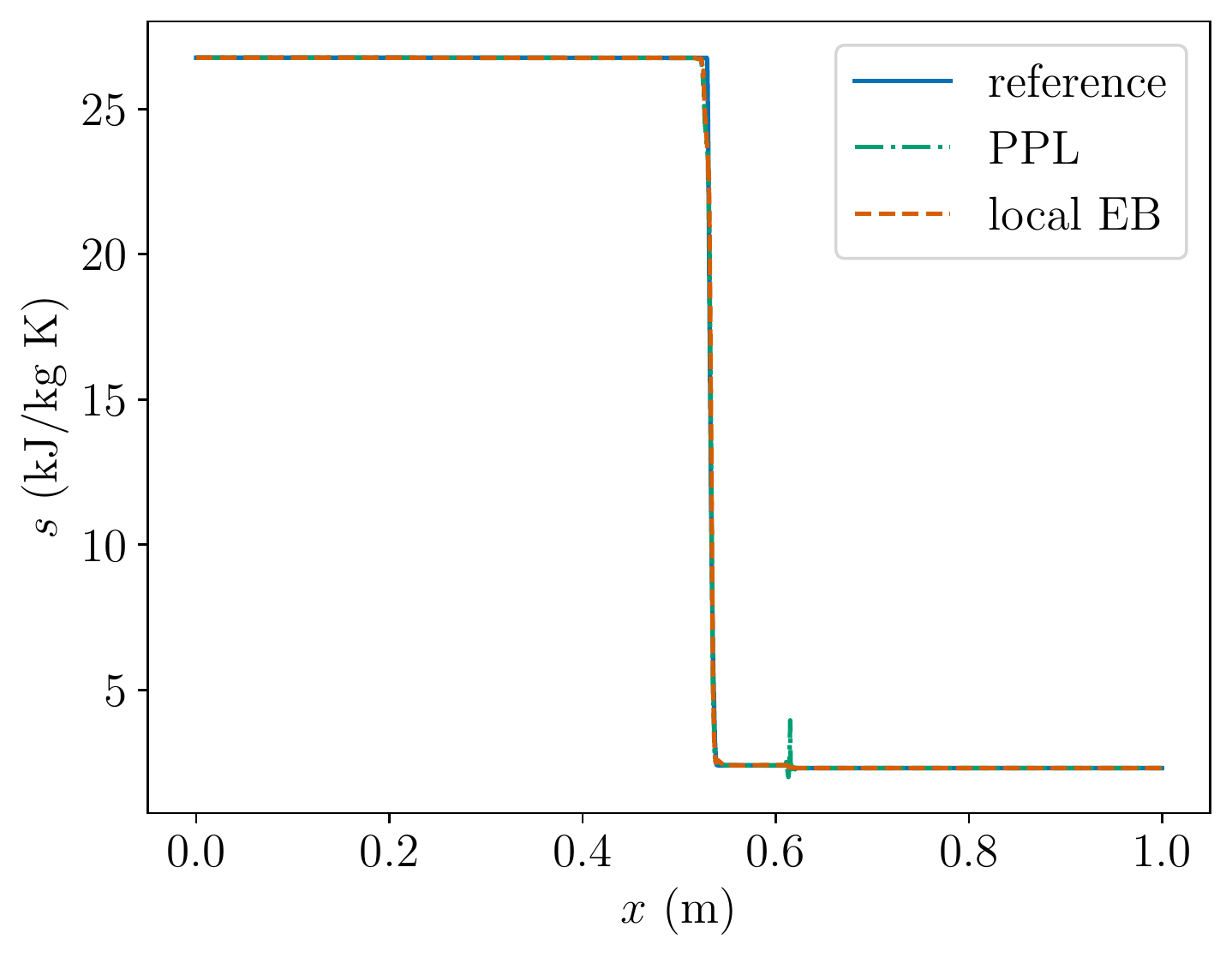}}

\caption{\label{fig:shock_tube_Houim} Results for $p=3$ solutions on 200
elements without artificial viscosity for the one-dimensional, multicomponent
shock-tube problem with initialization in Equation~(\ref{eq:shock-tube-IC-Houim}).
``PPL'' corresponds to the positivity-preserving limiter by itself,
and ``local EB'' refers to both the positivity-preserving and entropy
limiters with the local entropy bound in Equation~(\ref{eq:local-entropy-bound}).
The reference solution~\citep{Joh20_2} is computed with $p=2$,
2000 elements, and artificial viscosity.}
\end{figure}

Figure~\ref{fig:shock_tube_conservation_error_Houim} presents the
percent error in conservation of mass, energy, and atomic elements
for the ``local EB'' case as a representative example, calculated
every $0.3\;\mu\mathrm{s}$ (for a total of 1000 samples). $\mathsf{N}_{N}$
and $\mathsf{N}_{He}$ denote the total numbers of nitrogen and helium
atoms in the mixture. The error remains close to machine precision,
confirming that the proposed formulation is conservative. Also included
is the error in mass conservation (calculated every time step) for
a solution computed without the positivity-preserving and entropy
limiters, but instead with a simple clipping procedure in which negative
species concentrations are set to zero, a strategy employed by many
reacting-flow solvers. The error increases rapidly to non-negligible
values until the solver diverges.

\begin{figure}[H]
\begin{centering}
\includegraphics[width=0.6\columnwidth]{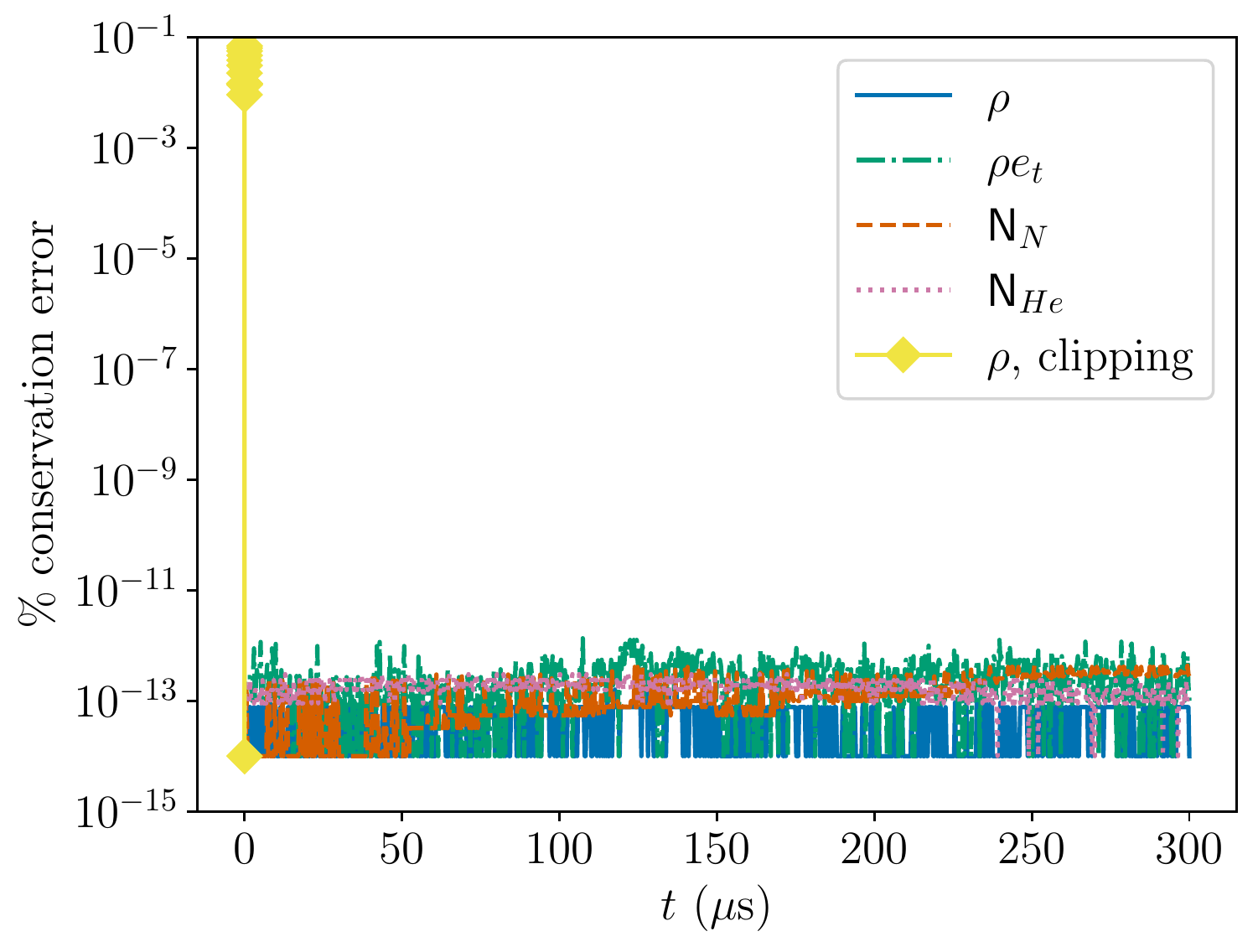}
\par\end{centering}
\caption{\label{fig:shock_tube_conservation_error_Houim}Percent error in conservation
of mass, energy, and atomic elements for the ``local EB'' case in
Figure~\ref{fig:shock_tube_Houim}, computed with $p=3$ on 200 elements.
The initial conditions for this one-dimensional, multicomponent shock-tube
problem are given in Equation~(\ref{eq:shock-tube-IC-Houim}). Also
included is the error in mass conservation for a solution computed
without the positivity-preserving and entropy limiters, but instead
with a simple clipping procedure in which negative species concentrations
are set to zero.}
\end{figure}

Next, we recompute this problem with artificial viscosity to confirm
adequate suppression of small-scale oscillations. Figure~\ref{fig:shock_tube_AV_Houim}
presents the results for $C_{\mathrm{AV}}=1$. The instabilities observed
in Figure~\ref{fig:shock_tube_Houim} are largely eliminated by the
artificial viscosity. As shown in Figure~\ref{fig:shock_tube_T_AV_Houim},
a temperature undershoot at the shock emerges when only the positivity-preserving
limiter is used, but is suppressed by the entropy limiter. The artificial
viscosity causes some smearing of the solution at the contact. Note
that without any limiting, negative species concentrations occur for
the $p=3$, 200-element cases, even with artificial viscosity. As
will be further discussed in Part II~\citep{Chi22_2}, artificial
viscosity alone, or even when combined with solely the positivity-preserving
limiter, does not provide sufficient stabilization in simulations
of complex detonation waves on relatively coarse meshes. In these
simulations, enforcement of the entropy principle is critical for
robustness.

\begin{figure}
\subfloat[\label{fig:shock_tube_Y_AV_Houim}Mass fractions.]{\includegraphics[width=0.45\columnwidth]{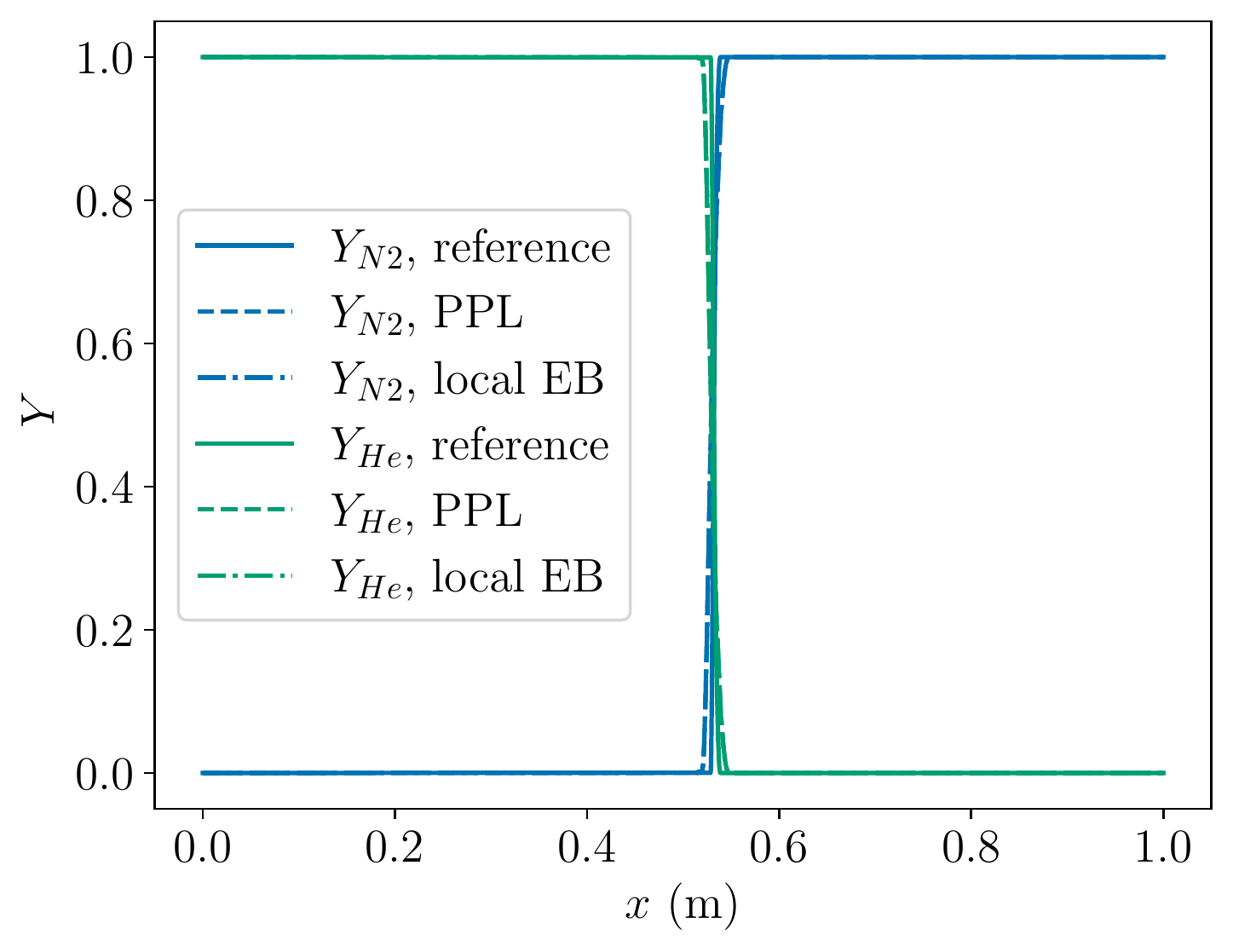}}\hfill{}\subfloat[\label{fig:shock_tube_P_AV_Houim}Pressure.]{\includegraphics[width=0.45\columnwidth]{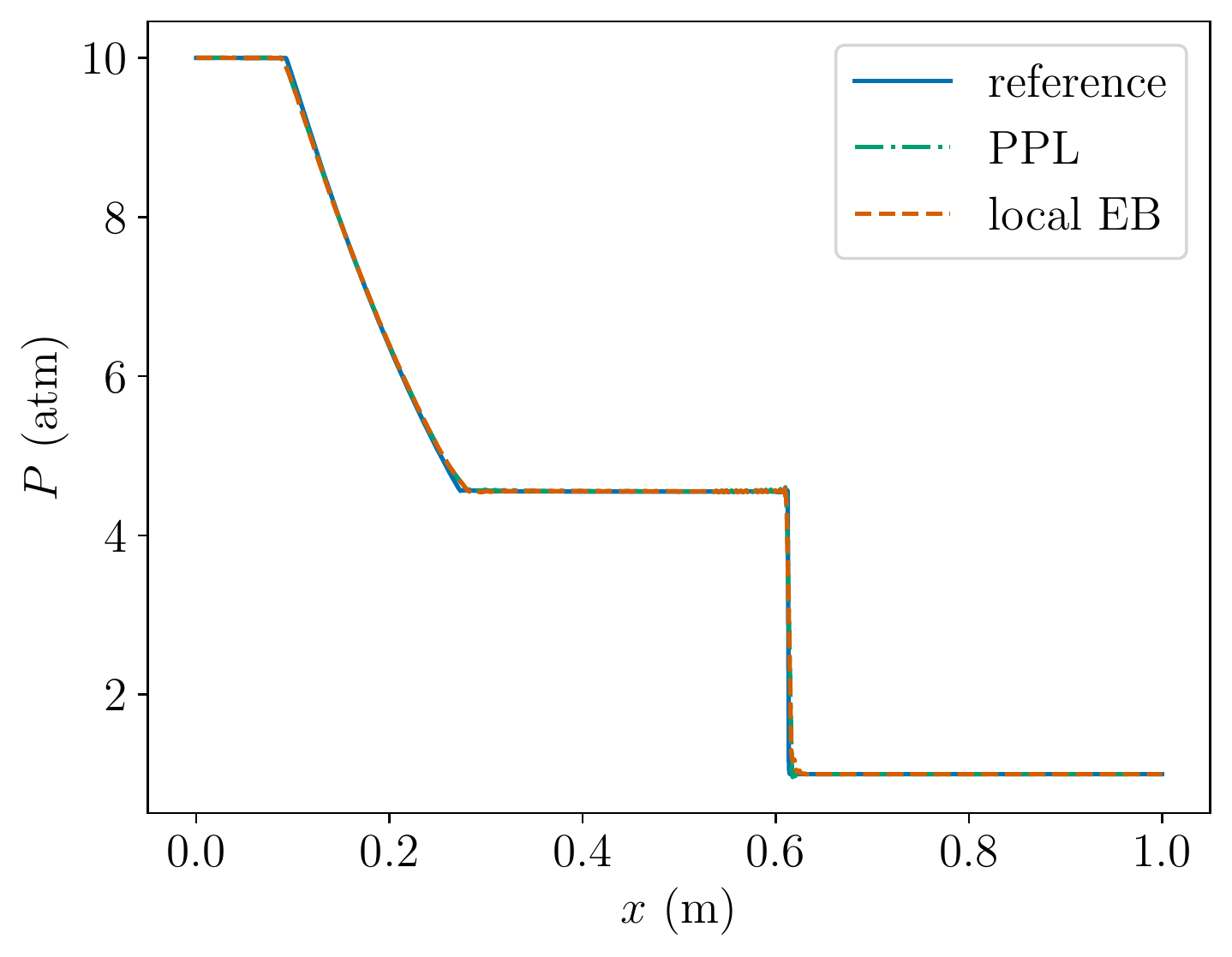}}\hfill{}\subfloat[\label{fig:shock_tube_T_AV_Houim}Temperature.]{\includegraphics[width=0.46\columnwidth]{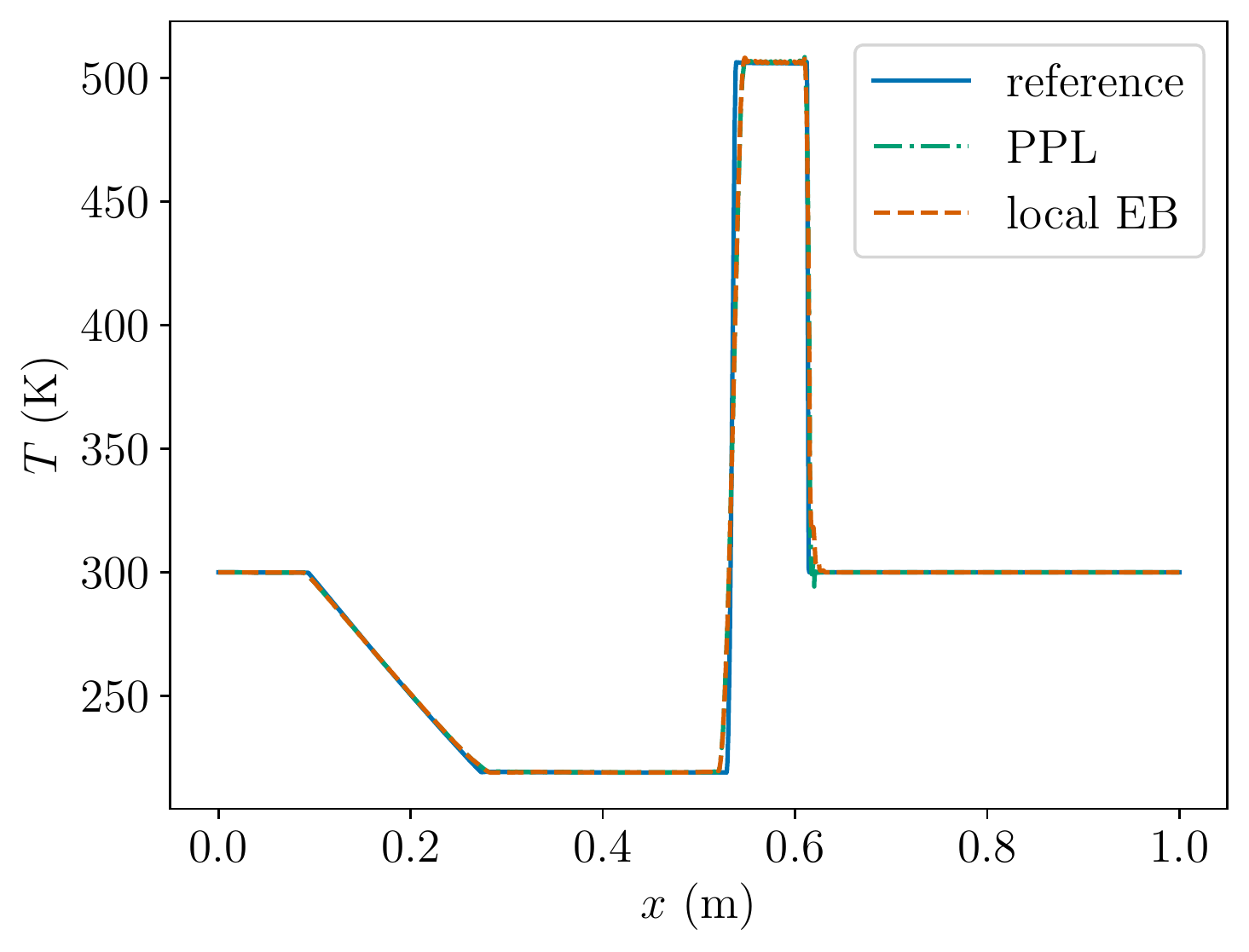}}\hfill{}\subfloat[\label{fig:shock_tube_s_AV_Houim}Specfic thermodynamic entropy.]{\includegraphics[width=0.45\columnwidth]{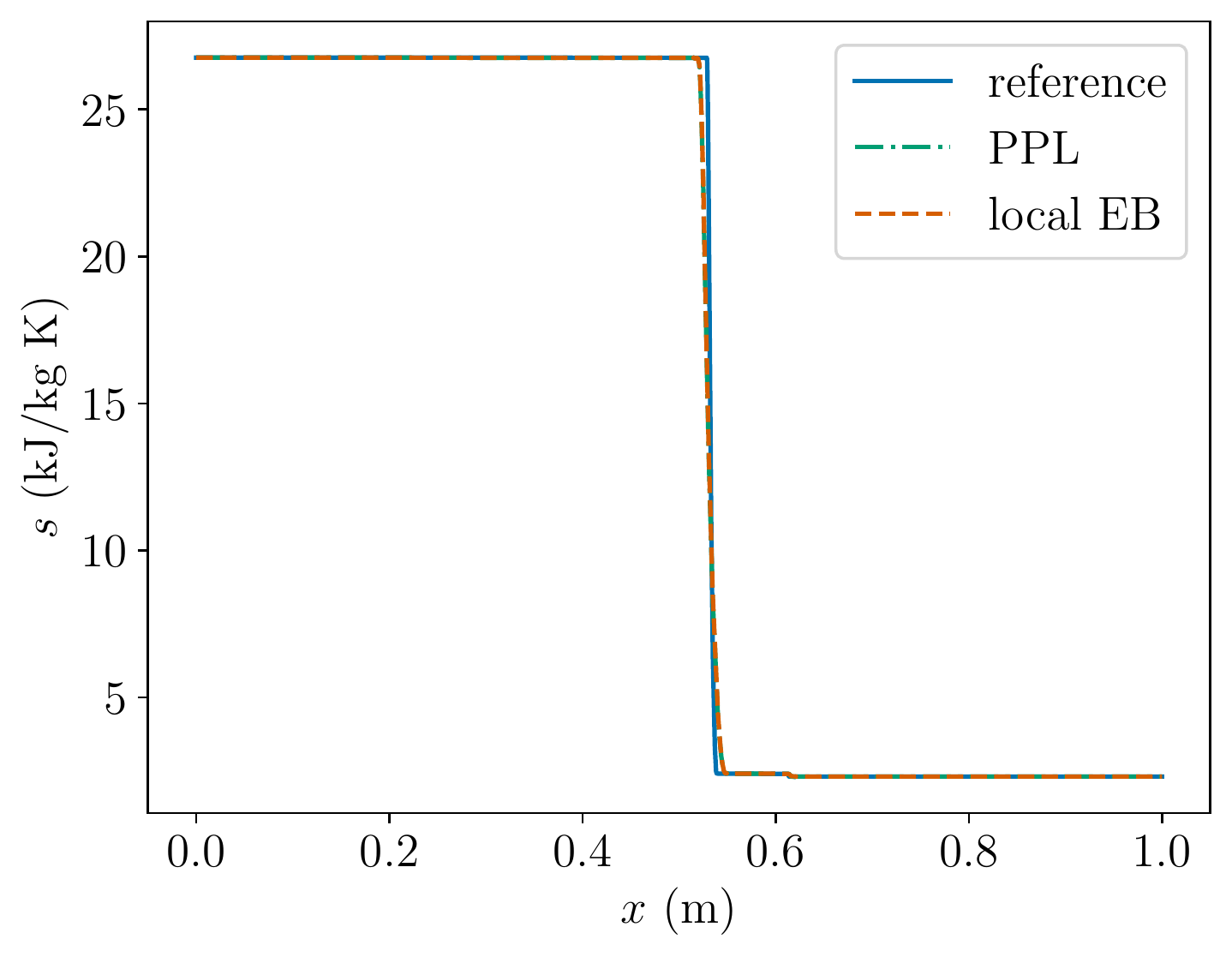}}

\caption{\label{fig:shock_tube_AV_Houim} Results for $p=3$ solutions on 200
elements with artificial viscosity for the one-dimensional, multicomponent
shock-tube problem with initialization in Equation~(\ref{eq:shock-tube-IC-Houim}).
``PPL'' corresponds to the positivity-preserving limiter by itself,
and ``local EB'' refers to both the positivity-preserving and entropy
limiters with the local entropy bound in Equation~(\ref{eq:local-entropy-bound}).
The reference solution~\citep{Joh20_2} is computed with $p=2$,
2000 elements, and artificial viscosity. The difference between these
results and those in Figure~\ref{fig:shock_tube_Houim} is the use
of artificial viscosity in the (non-reference) solutions here.}
\end{figure}

To highlight discrepancies between the local and global entropy bounds,
we consider different initial conditions, given by
\begin{equation}
\left(v_{1},T,P,Y_{N_{2}},Y_{He}\right)=\begin{cases}
\left(0\text{ m/s},300\text{ K},1\text{ atm},0,1\right), & x\geq0.4\\
\left(0\text{ m/s},300\text{ K},10\text{ atm},1,0\right), & x<0.4
\end{cases}.\label{eq:shock-tube-IC-other}
\end{equation}
The only difference with the previous initial conditions is in the
mass fractions. Displayed in Figure~\ref{fig:shock_tube_other} are
the results obtained with the global entropy bound (referred to as
``global EB'') and the local entropy bound (again referred to as
``local EB''). These $p=2$ solutions are computed on 200 elements
without any artificial viscosity. The positivity-preserving limiter
by itself yields very similar results (not shown for brevity) to the
entropy limiter with the global entropy bound. The reference solution
is again computed with $p=2$, 2000 elements, and artificial viscosity.
The mass fractions are well-captured in all cases. The differences
between the ``global EB'' and ``local EB'' solutions here are
not as large as those between the ``PPL'' and ``local EB'' solutions
in Figure~\ref{fig:shock_tube_Houim}. %
Nevertheless, the benefit of the local entropy bound is evident, specifically
at the shock. Spurious artifacts in the pressure profile and especially
the temperature profile near the shock are noticeably larger for the
global entropy bound. The discrepancies between the two solutions
are attributed to the considerable difference between the specific
thermodynamic entropy in the vicinity of the shock and the global
minimum, as illustrated in Figure~\ref{fig:shock_tube_s_other}.
As such, the local entropy bound is particularly beneficial in flow
problems with large variations in the specific thermodynamic entropy
throughout the domain.

\begin{figure}
\subfloat[\label{fig:shock_tube_Y_other}Mass fractions.]{\includegraphics[width=0.45\columnwidth]{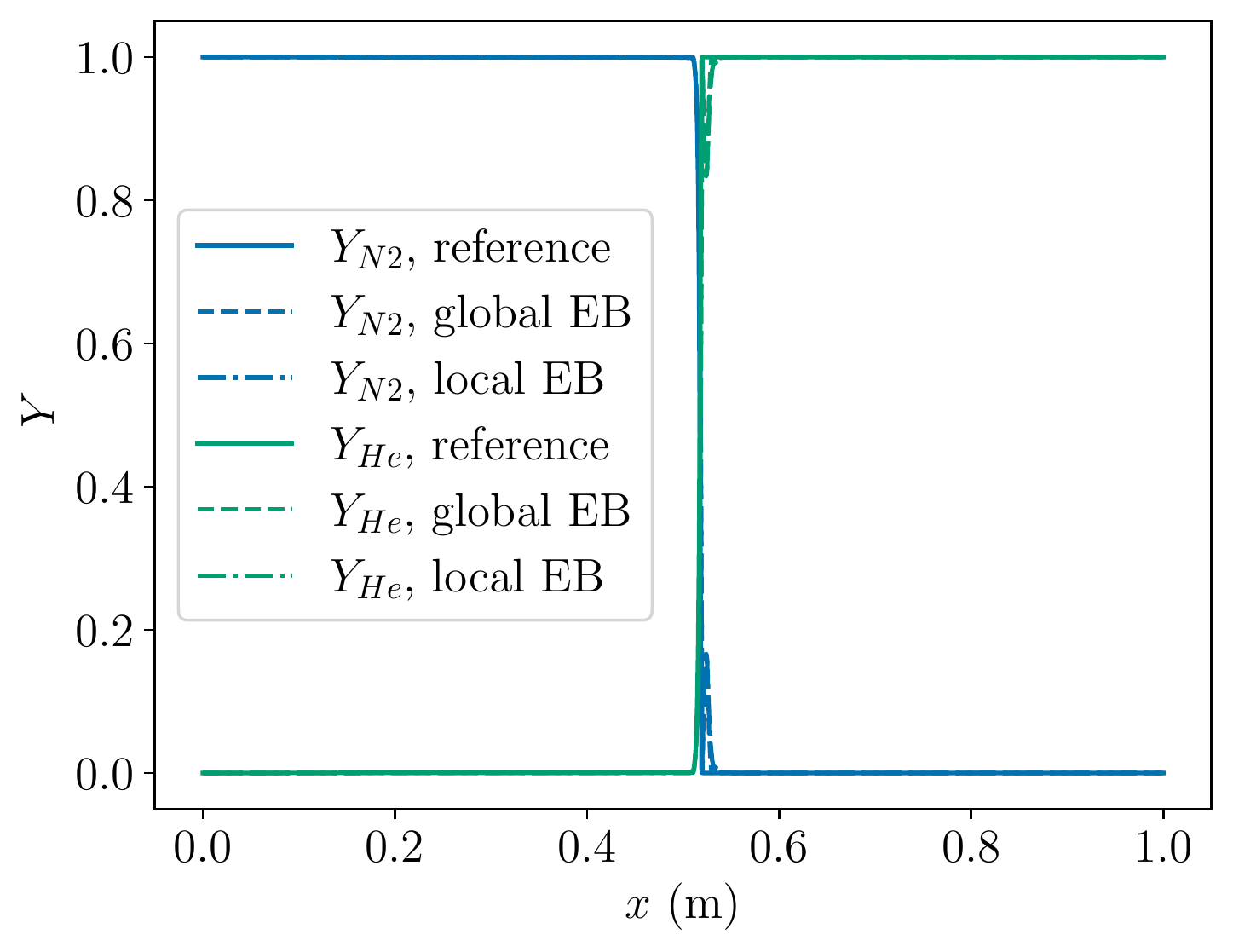}}\hfill{}\subfloat[\label{fig:shock_tube_P_other}Pressure.]{\includegraphics[width=0.45\columnwidth]{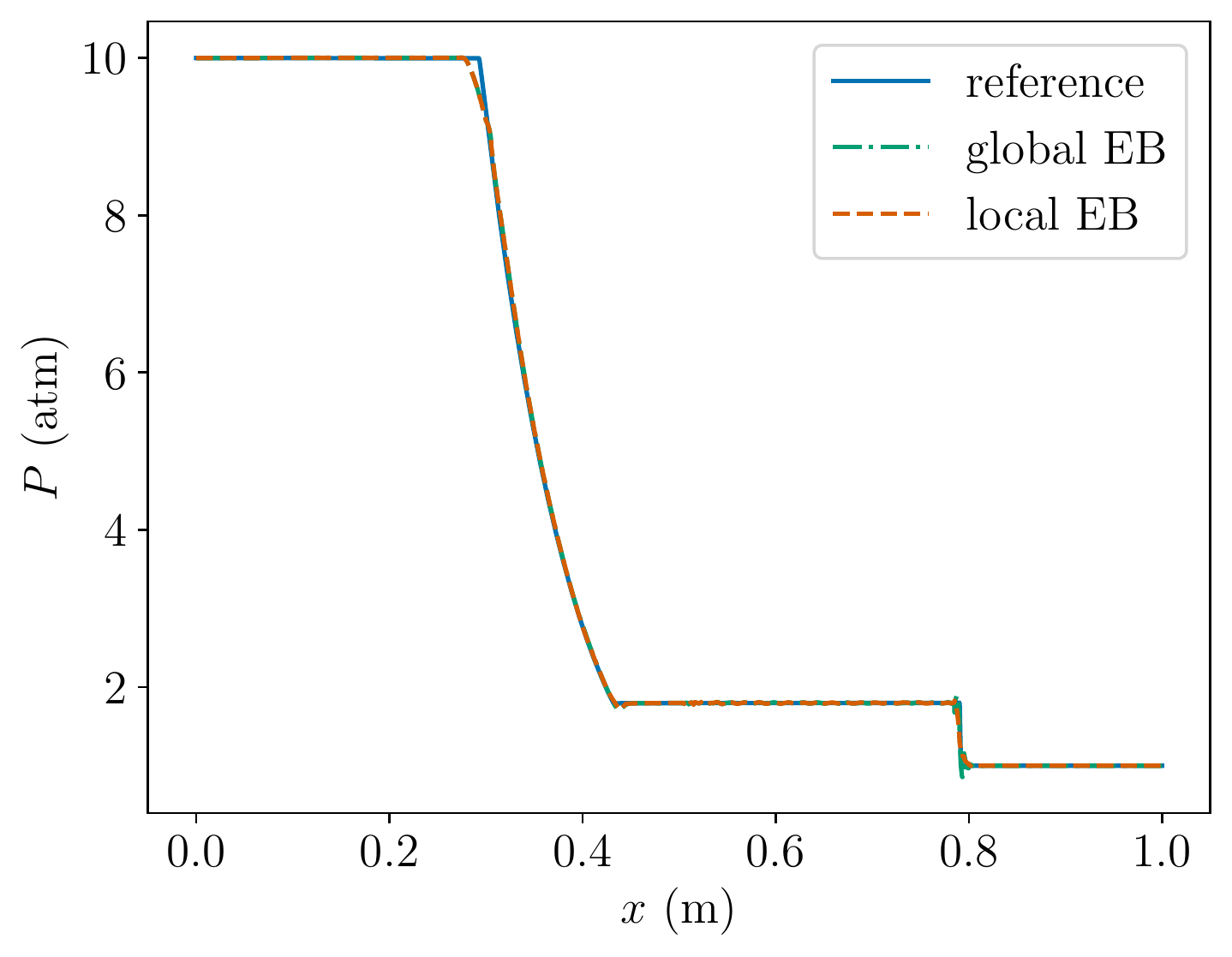}}\hfill{}\subfloat[\label{fig:shock_tube_T_other}Temperature.]{\includegraphics[width=0.46\columnwidth]{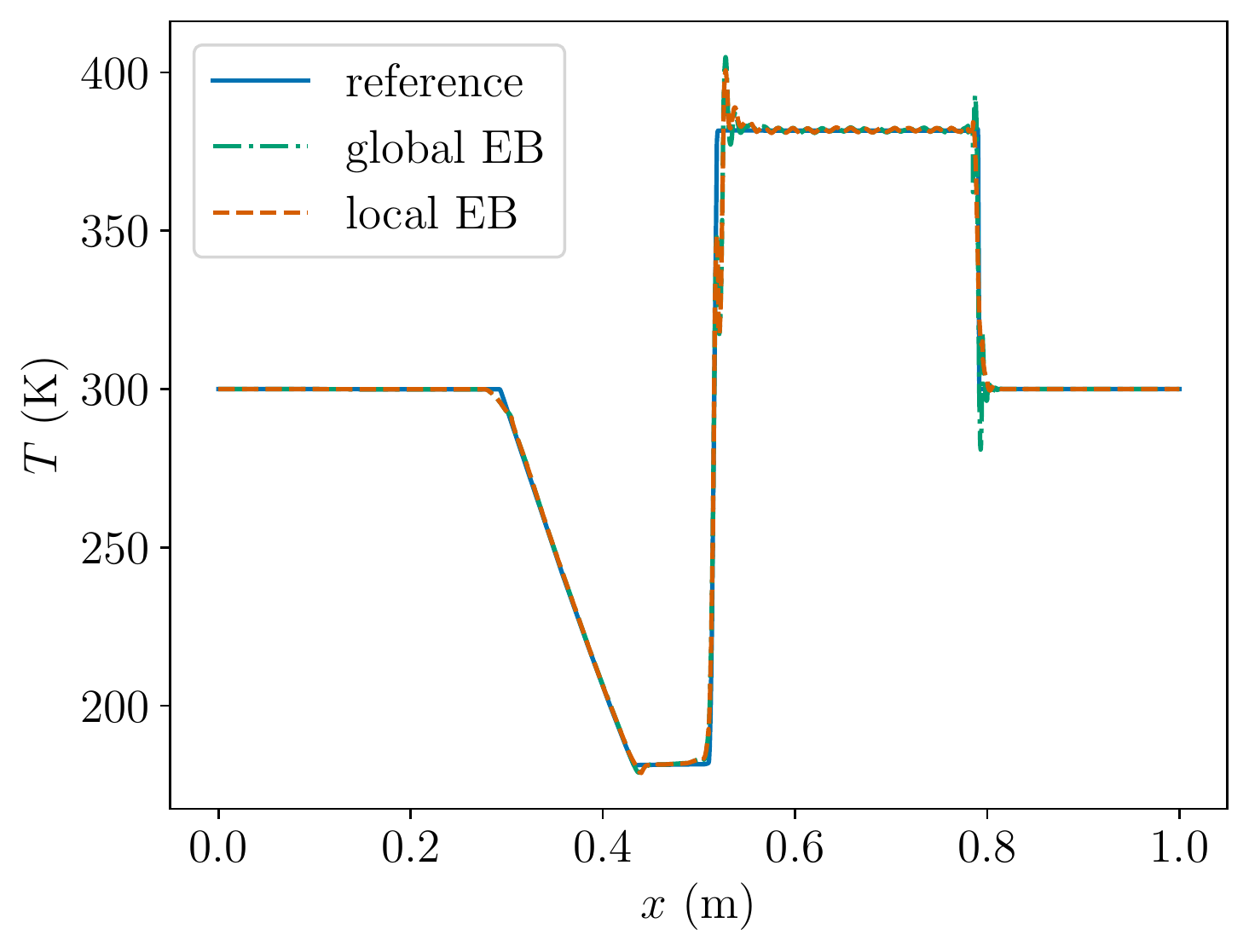}}\hfill{}\subfloat[\label{fig:shock_tube_s_other}Specfic thermodynamic entropy.]{\includegraphics[width=0.45\columnwidth]{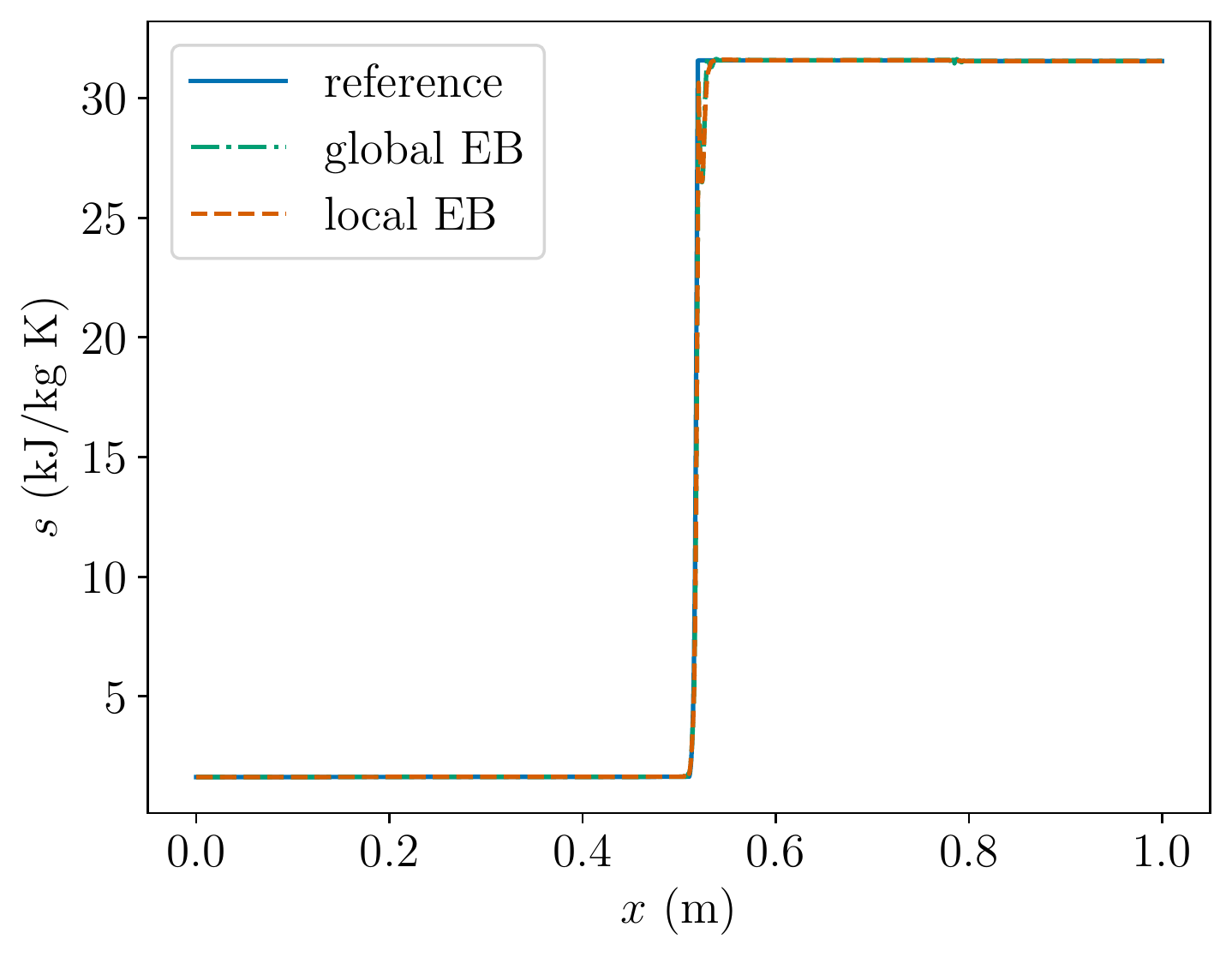}}

\caption{\label{fig:shock_tube_other} Results for $p=2$ solutions on 200
elements without artificial viscosity for the one-dimensional, multicomponent
shock-tube problem with initialization in Equation~(\ref{eq:shock-tube-IC-other}).
``Global EB'' refers to the entropy limiter with the global entropy
bound in Equation~(\ref{eq:global-entropy-bound}), and ``local
EB'' refers to the entropy limiter with the local entropy bound in
Equation~(\ref{eq:local-entropy-bound}). The reference solution
is computed with $p=2$, 2000 elements, and artificial viscosity.}
\end{figure}

\subsection{Detonation wave~\label{subsec:detonation-results-1D}}

The final one-dimensional test case is a hydrogen-oxygen detonation
wave diluted in Argon with initial conditions

\begin{equation}
\begin{array}{cccc}
\qquad\qquad\qquad\qquad v_{1} & = & 0\text{ m/s},\\
\quad\quad\quad\;X_{Ar}:X_{O_{2}}:X_{H_{2}} & = & 7:1:2 & \text{ }x>0.025\text{ m},\\
X_{Ar}:X_{O_{2}}:X_{H_{2}}:X_{OH} & = & 7:1:2:0.01 & \text{ }0.015\text{ m}<x<0.025\text{ m,}\\
\quad\quad\;X_{Ar}:X_{H_{2}O}:X_{OH} & = & 8:2:0.0001 & x<0.015\text{ m},\\
\qquad\qquad\qquad\qquad P & = & \begin{cases}
\expnumber{5.50}5 & \text{ Pa}\\
\expnumber{6.67}3 & \text{ Pa}
\end{cases} & \begin{array}{c}
x<0.015\text{ m}\\
x>0.015\text{ m}
\end{array},\\
\qquad\qquad\qquad\qquad T & = & \begin{cases}
300 & \text{ K}\\
350 & \text{ K}\\
3500\text{\hspace{1em}\hspace{1em}} & \text{ K}
\end{cases} & \begin{array}{c}
\text{ }x>0.025\text{ m}\\
\text{ }0.015\text{ m}<x<0.025\text{ m}\\
x<0.015\text{ m}
\end{array}.
\end{array}\label{eq:detonation-1d-initialization}
\end{equation}
The domain is $\Omega=(0,0.45)$ m, with walls at the left and right
boundaries. In previous work~\citep{Joh20_2}, this case was computed
with $p=1$ and mesh spacing $h=9\times10^{-5}$ m. At this resolution,
artificial viscosity was the only stabilization necessary to obtain
an accurate solution while maintaining conservation of mass and energy.
Good agreement with the Shock and Detonation Toolbox~\citep{sdtoolbox}
was observed. Additional details can be found in~\citep{Joh20_2}.
Here, our objective is to demonstrate that the proposed formulation,
specifically the entropy-bounded DG discretization of the convective
operator combined with entropy-stable DGODE for stiff chemical reactions,
can compute stable and accurate solutions with appreciably lower resolution.
In light of~\citep{Joh20_2}, we use a $p=1$, $h=9\times10^{-5}$
calculation as a reference solution. The only difference with~\citep{Joh20_2}
is that the reaction mechanism~\citep{Wes82} is slightly modified
to solely contain reversible reactions%
.

Figure~\ref{fig:1D_detonation} presents $p=1$, $p=2$, and $p=3$
results at $t=235$~$\mu\mathrm{s}$. The mesh spacing in these simulations
is $h=4.5\times10^{-4}$ m, which is fives time larger than for the
reference solution. Artificial viscosity and the local entropy limiter
are employed. At this mesh spacing, artificial viscosity alone is
not sufficient to stabilize the solution. Good agreement in temperature
and pressure with the reference solution is observed. As shown in
Figure~\ref{fig:1D_detonation_T_zoom}, which zooms in on the shock,
there are slight discrepancies in the predictions of the leading-shock-front
location and the induction region. However, these predictions improve
with increasing $p$. Also included in Figure~\ref{fig:1D_detonation_T_zoom}
is a $p=1$ solution with standard DGODE in order to ensure that entropy-stable
DGODE is not the primary cause of the disagreement. Figure~\ref{fig:1D_detonation_s}
shows small-scale entropy oscillations behind the leading shock front
in the $p=3$ calculation. The presence of these instabilities is
not too surprising given the linear nature of the limiting operator,
coupled with the increased susceptibility to oscillations of high
polynomial orders. Nevertheless, these instabilities do not significantly
pollute the rest of the solution. 

\begin{figure}
\subfloat[\label{fig:1D_detonation_T}Temperature.]{\includegraphics[width=0.45\columnwidth]{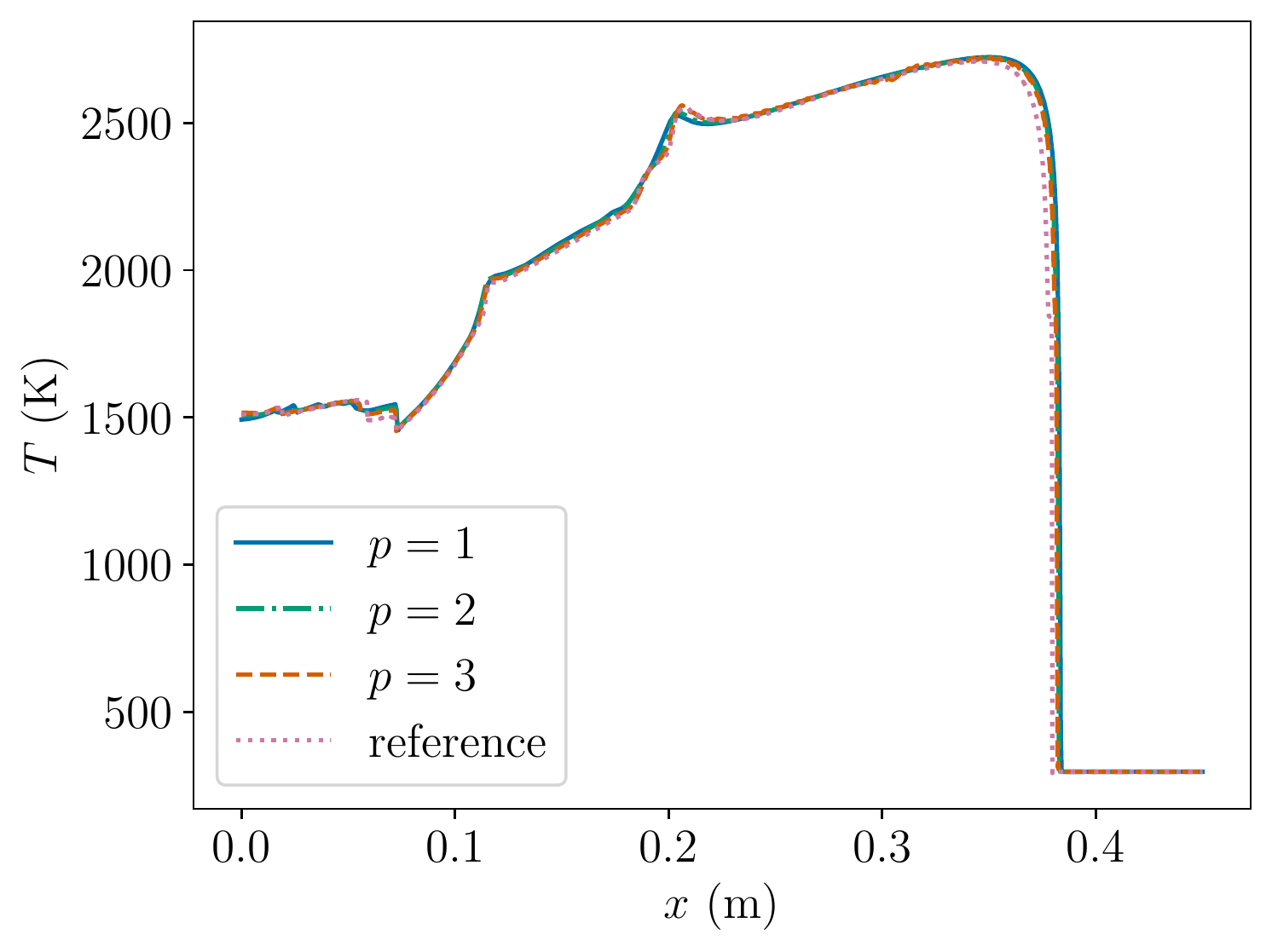}}\hfill{}\subfloat[\label{fig:1D_detonation_T_zoom}Temperature, zoomed in on shock.]{\includegraphics[width=0.45\columnwidth]{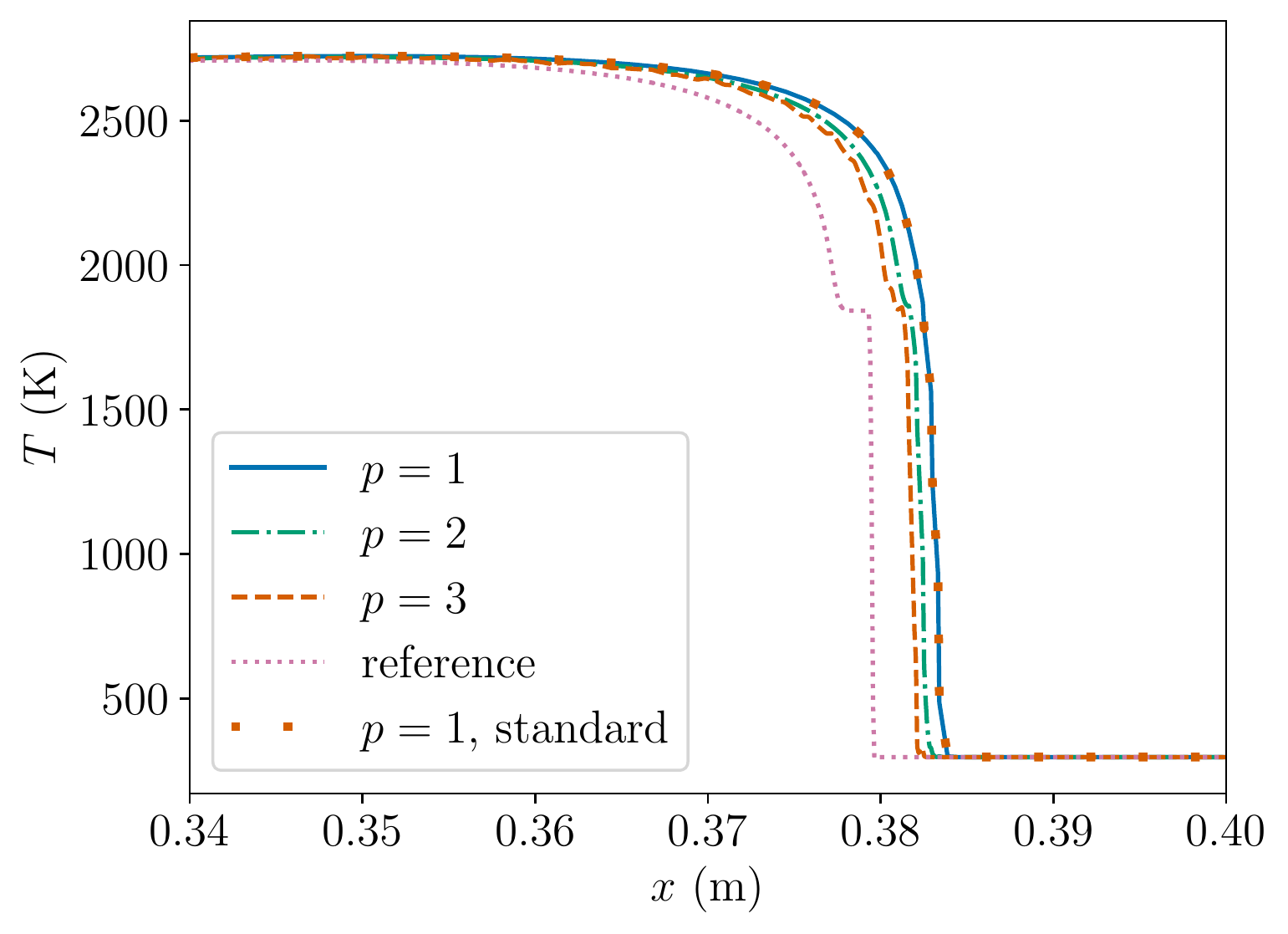}}\hfill{}\subfloat[\label{fig:1D_detonation_P}Pressure.]{\includegraphics[width=0.45\columnwidth]{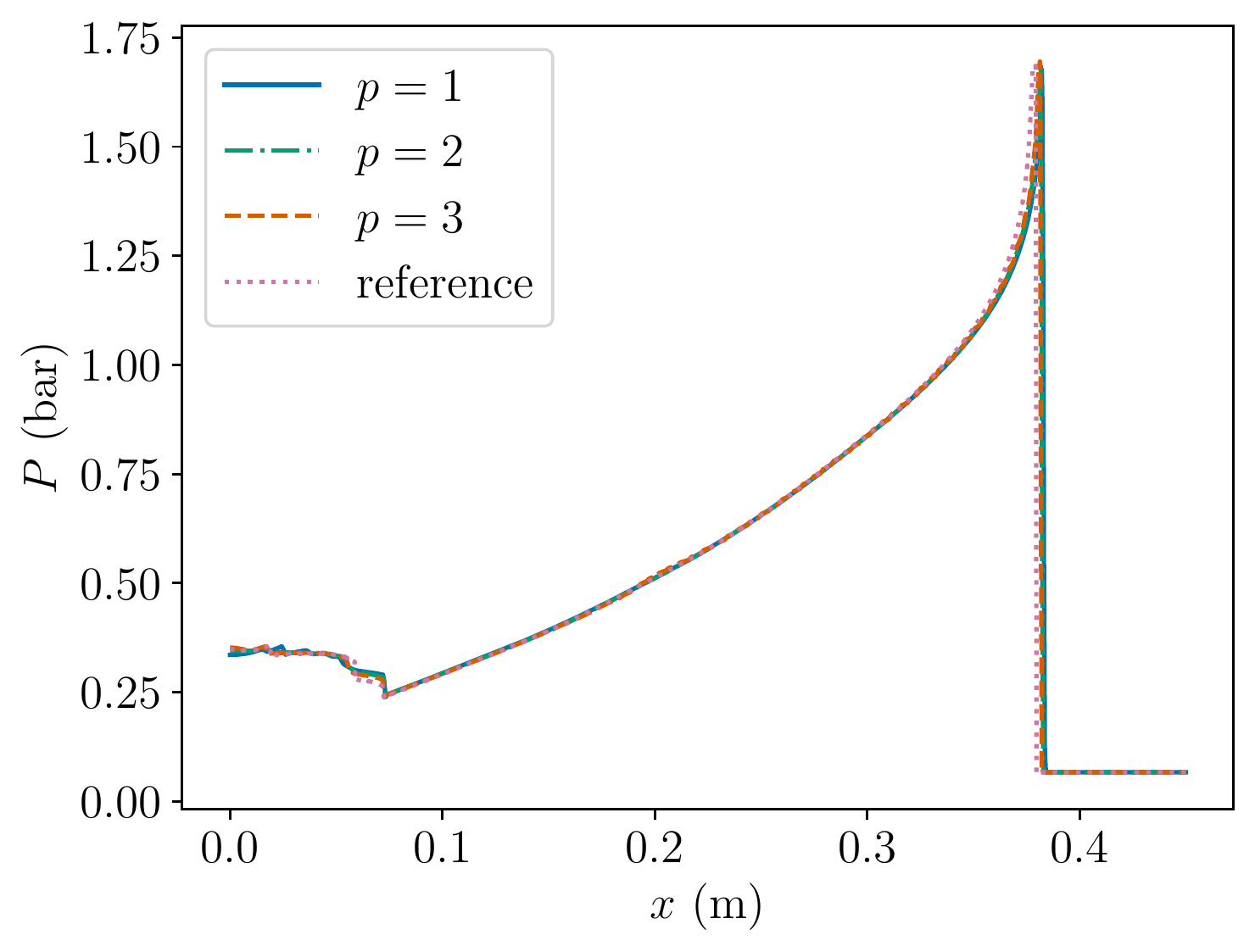}}\hfill{}\subfloat[\label{fig:1D_detonation_s}Specific thermodynamic entropy.]{\includegraphics[width=0.45\columnwidth]{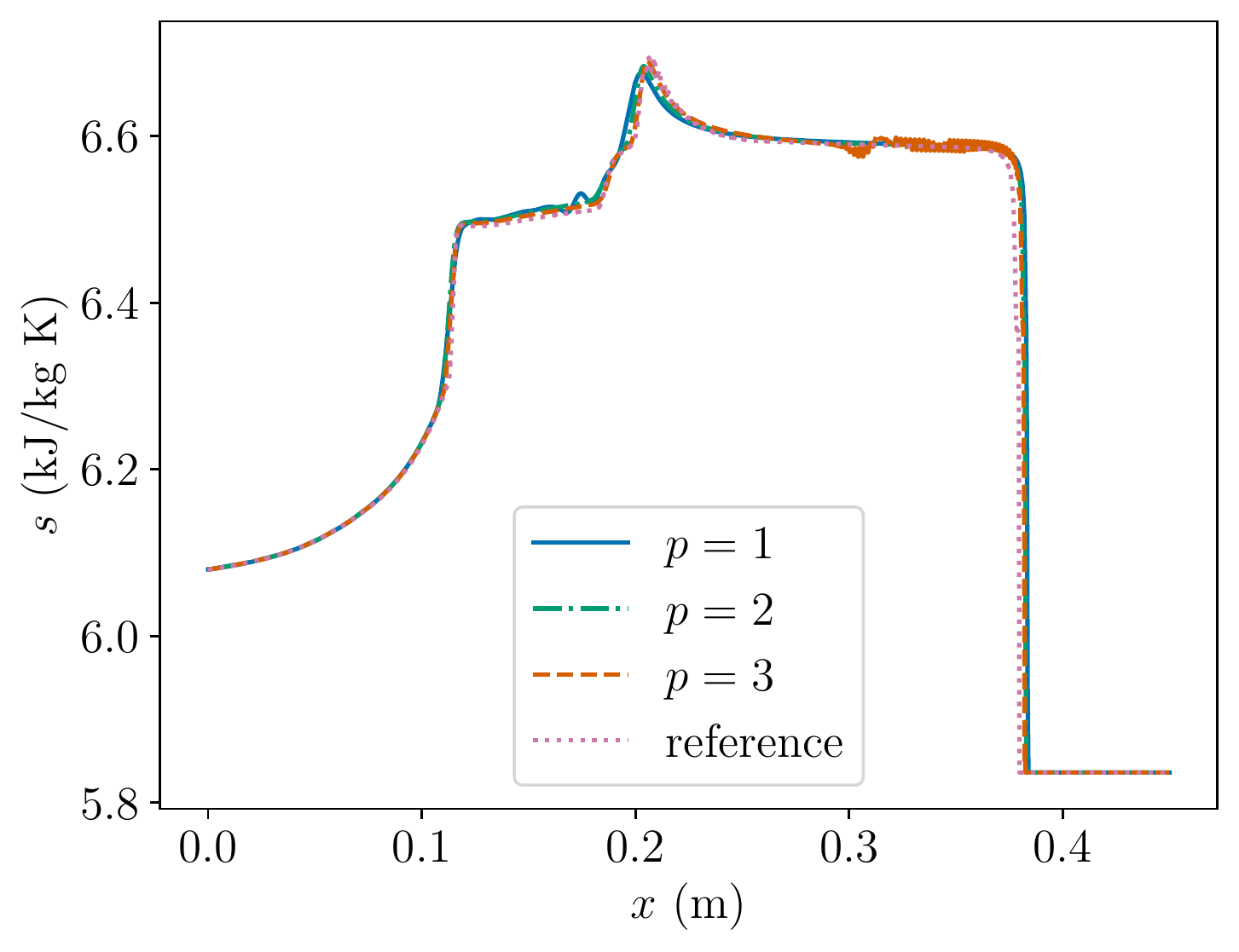}}

\caption{\label{fig:1D_detonation} $p=1$, $p=2$, and $p=3$ results at $t=235$~$\mu\mathrm{s}$
for the one-dimensional hydrogen detonation test case. The mesh spacing
in these simulations is $h=4.5\times10^{-4}$ m, which is fives time
larger than for the reference solution. Artificial viscosity and the
local entropy limiter are employed.}
\end{figure}

Figure~\ref{fig:1D_detonation_mass_fractions} compares mass-fraction
profiles of selected species obtained with $p=3$ to those of the
reference solution. Marginal oscillations can be observed behind the
shock, particularly in the mass-fraction profiles of $\mathrm{H_{2}O}$
and $\mathrm{HO_{2}}$. In Figure~\ref{fig:1D_detonation_mass_fractions_3},
the $p=3$ solution predicts slightly lower peaks in the mass fractions
of $\mathrm{HO_{2}}$ and $\mathrm{H_{2}O_{2}}$. Furthermore, the
aforementioned discrepancy in the leading-shock location is reflected
in these results. Nevertheless, there is very good agreement between
the $p=3$ solution and the reference solution, illustrating the ability
of the developed formulation to obtain stable and accurate detonation
results on coarser meshes.

\begin{figure}
\begin{centering}
\subfloat[\label{fig:1D_detonation_mass_fractions_1}]{\includegraphics[width=0.45\columnwidth]{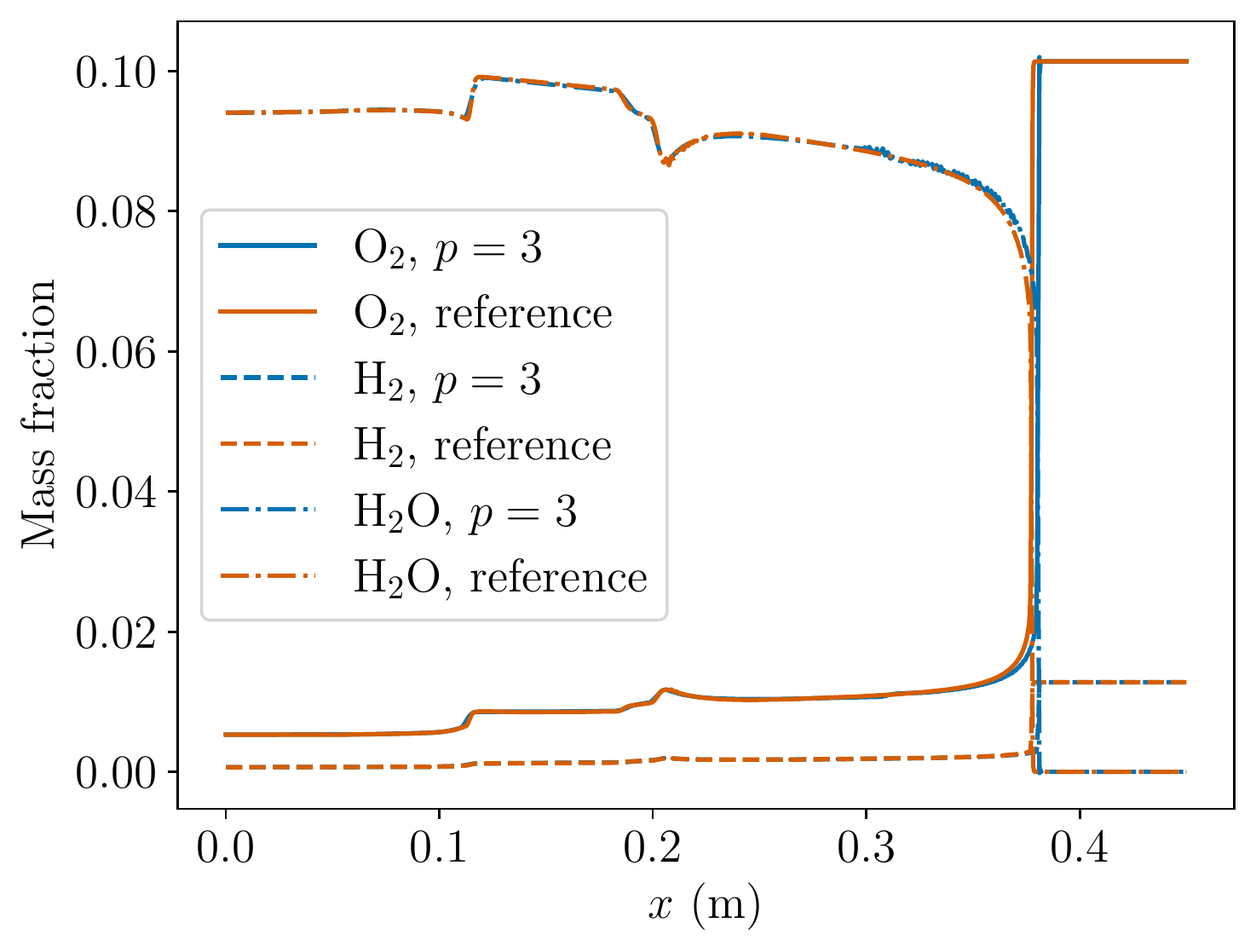}}\hfill{}\subfloat[\label{fig:1D_detonation_mass_fractions_2}]{\includegraphics[width=0.45\columnwidth]{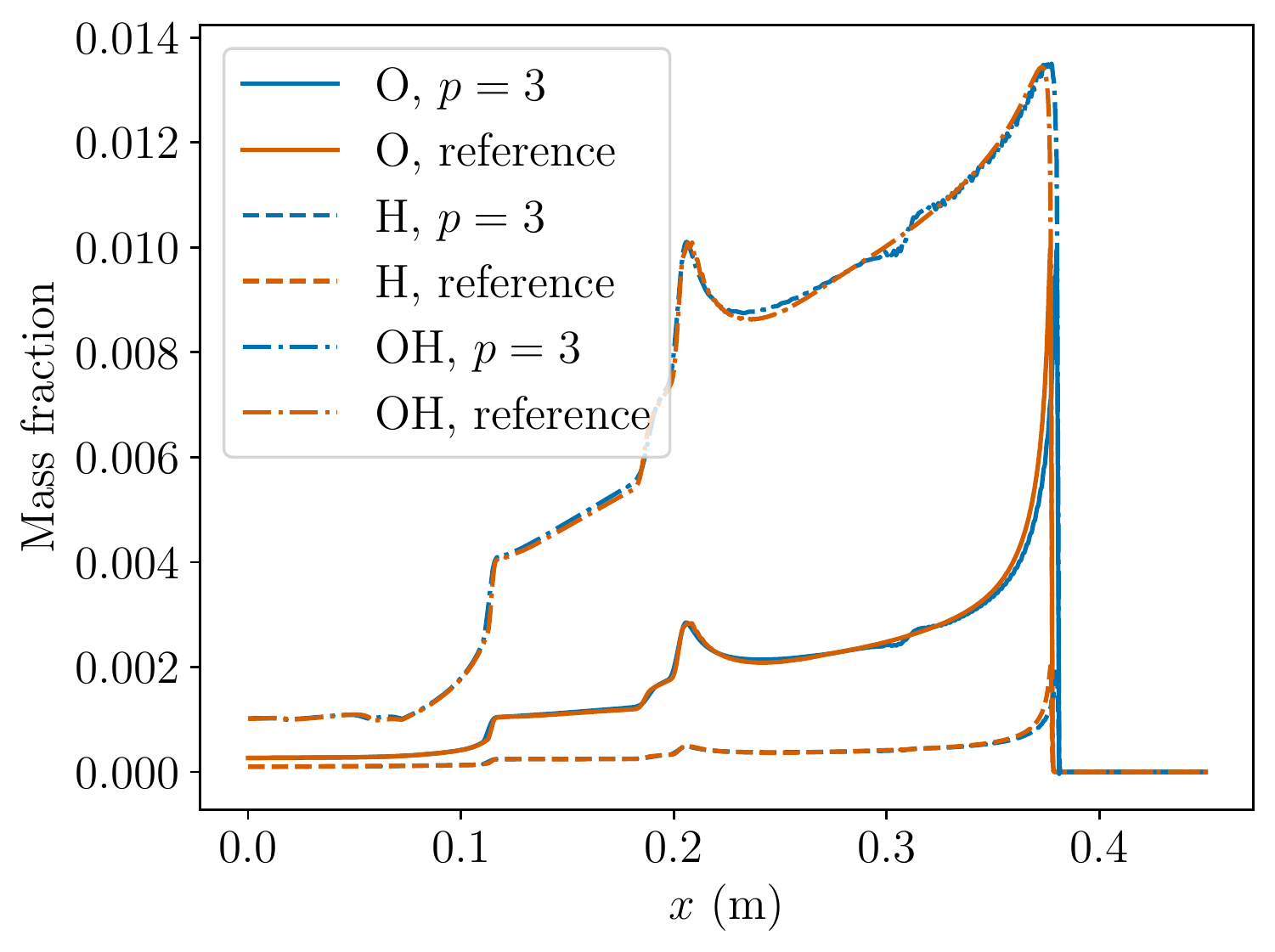}}\hfill{}\subfloat[\label{fig:1D_detonation_mass_fractions_3}]{\centering{}\includegraphics[width=0.45\columnwidth]{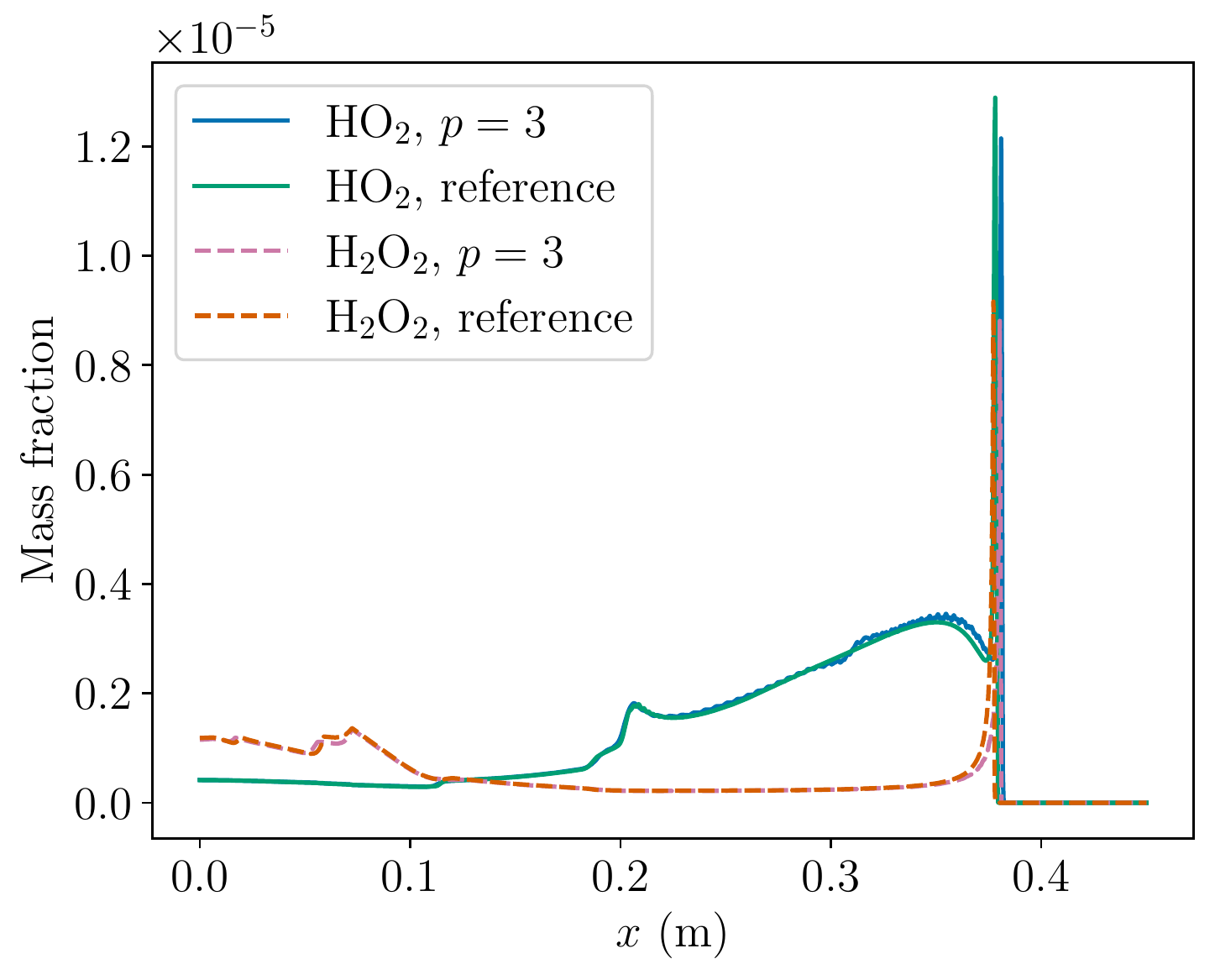}}
\par\end{centering}
\caption{\label{fig:1D_detonation_mass_fractions} Comparison of $p=3$ predictions
of species mass fractions with those of the reference solution for
the one-dimensional hydrogen detonation test case. The mesh spacing
is $h=4.5\times10^{-4}$ m, which is fives time larger than for the
reference solution.}
\end{figure}

Finally, Figure~\ref{fig:1D_detonation_conservation_error} gives
the percent error in conservation of mass, energy, and atomic elements
for $p=1$ as a representative example, calculated every $0.235\;\mu\mathrm{s}$
(for a total of 1000 samples). $\mathsf{N}_{O}$, $\mathsf{N}_{H}$,
and $\mathsf{N}_{Ar}$ denote the total numbers of oxygen, hydrogen,
and argon atoms in the mixture. The error remains negligible throughout
the simulation, confirming that the methodology is conservative.

\begin{figure}[H]
\begin{centering}
\includegraphics[width=0.6\columnwidth]{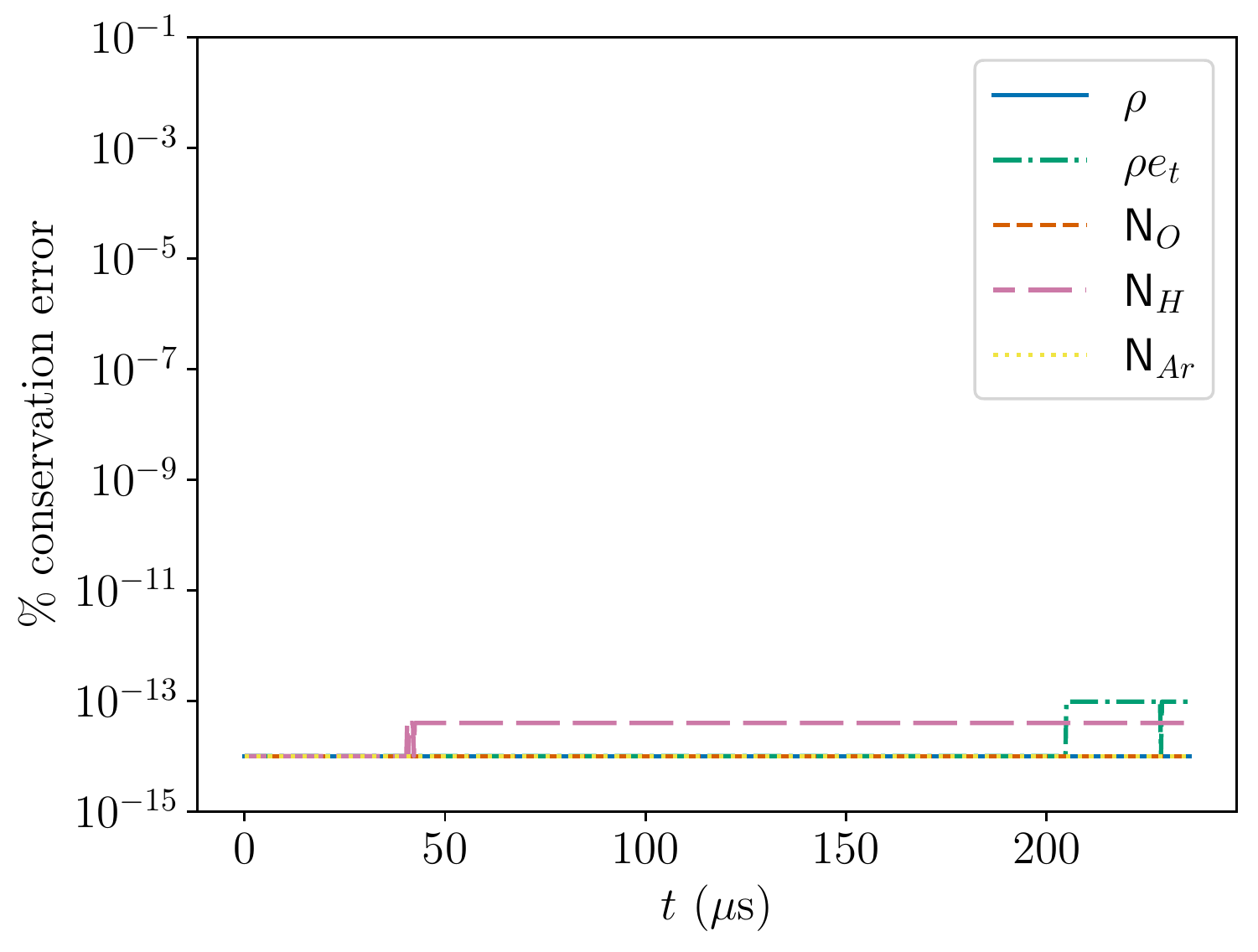}
\par\end{centering}
\caption{\label{fig:1D_detonation_conservation_error}Percent error in conservation
of mass, energy, and atomic elements for the $p=1$ calculation with
$h=4.5\times10^{-4}$ m. The initial conditions for this one-dimensional
hydrogen detonation problem are given in Equation~(\ref{eq:detonation-1d-initialization}).}
\end{figure}

\section{Concluding remarks}

In this paper, we introduced a positivity-preserving and entropy-bounded
DG methodology for the chemically reacting, compressible Euler equations.
The methodology builds on the fully conservative, high-order DG method
previously developed by two of the authors~\citep{Joh20_2}, which
does not generate spurious pressure oscillations in smooth flow regions
or across material interfaces when the temperature is continuous.
As a prerequisite for the proposed formulation, we proved a minimum
entropy principle for the compressible, multicomponent, chemically
reacting Euler equations, which follows from the proof by Gouasmi
et al.~\citep{Gou20} of a minimum entropy principle for the compressible,
multicomponent, nonreacting Euler equations. 

In this first part of our two-part paper, we focused on the one-dimensional
case. A simple linear-scaling limiter ensures that the solution at
a given time step is admissible (i.e., species concentrations are
nonnegative, density is positive, pressure is positive, and entropy
is greater than some lower bound). A requirement of the limiter is
that the element average of the state is itself admissible, which
we showed to be true under the following conditions: (a) a time-step-size
constraint is satisfied, (b) an invariant-region-preserving numerical
flux is employed, and (c) certain pointwise values of the solution
at the previous time step are admissible. Both a global entropy bound
and a local entropy bound were discussed. Since the linear scaling
does not completely eliminate small-scale oscillations, artificial
viscosity is employed in tandem. We also detailed how to maintain
compatibility between the proposed framework and the pressure-equilibrium-maintaining
discretization in~\citep{Joh20_2}. The temporal integration of the
convection operator is decoupled from that of the stiff chemical source
term via Strang splitting. To guarantee satisfaction of the minimum
entropy principle in the reaction step, we developed an entropy-stable
discontinuous Galerkin method based on diagonal-norm summation-by-parts
operators for temporal integration of the source term, which involved
the derivation of an entropy-conservative two-point numerical state
function.

The methodology was applied to canonical one-dimensional test cases.
The first two entailed nonreacting flows: advection of a smooth, hydrogen-oxygen
thermal bubble and nitrogen-helium shock-tube flow. In the former,
we demonstrated optimal convergence of the methodology and sufficient
preservation of pressure equilibrium. In the latter, we observed the
following:
\begin{itemize}
\item The positivity-preserving limiter (which does not consider an entropy
bound) prevents the solver from crashing, but, in the absence of additional
stabilization, gives rise to large-scale oscillations. Such instabilities
are substantially larger than those typically seen in the monocomponent,
calorically perfect case, illustrating the challenges of stabilizing
computations of multicomponent flows with realistic thermodynamics. 
\item The entropy limiter, on the other hand, considerably reduces the magnitude
of the aforementioned instabilities, suggesting that the relative
benefit of the entropy limiter is much greater in the multicomponent,
thermally perfect case. Small-scale oscillations are still present,
but these can be cured with artificial viscosity. Note that the inability
to completely eliminate oscillations is a well-known property of the
linear-scaling limiter~\citep{Zha10,Zha17,Lv15_2}. Furthermore,
unless a very fine resolution is employed, artificial viscosity alone
is insufficient for robustness.
\item Enforcing a local entropy bound can be more effective than enforcing
a global entropy bound when the entropy varies significantly throughout
the domain.
\end{itemize}
In our final test case, we computed a moving hydrogen-oxygen detonation
wave diluted in Argon, demonstrating that the developed methodology
can accurately and robustly calculate a chemically reacting flow with
detailed chemistry using high-order polynomial approximations on relatively
coarse meshes. Conservation of mass, total energy, and atomic elements
was confirmed.

In Part II~\citep{Chi22_2}, we will extend our formulation to multiple
dimensions. In our developed multidimensional framework, restrictions
on the physical modeling, geometry, numerical flux function, and quadrature
rules are milder than those currently in the literature. Complex two-
and three-dimensional detonations will be accurately computed in a
stable manner using high-order polynomial approximations.

\section*{Acknowledgments}

This work is sponsored by the Office of Naval Research through the
Naval Research Laboratory 6.1 Computational Physics Task Area. 

\bibliographystyle{elsarticle-num}
\bibliography{citations}

\begin{thebibliography}{100}
\expandafter\ifx\csname url\endcsname\relax
  \def\url#1{\texttt{#1}}\fi
\expandafter\ifx\csname urlprefix\endcsname\relax\def\urlprefix{URL }\fi
\expandafter\ifx\csname href\endcsname\relax
  \def\href#1#2{#2} \def\path#1{#1}\fi

\bibitem{Ree73}
W.~H. Reed, T.~Hill, Triangular mesh methods for the neutron transport
  equation, Tech. rep., Los Alamos Scientific Lab., N. Mex.(USA) (1973).

\bibitem{Bas97_2}
F.~Bassi, S.~Rebay, High-order accurate discontinuous finite element solution
  of the 2{D} {E}uler equations, Journal of Computational Physics 138~(2)
  (1997) 251--285.

\bibitem{Bas97}
F.~Bassi, S.~Rebay, A high-order accurate discontinuous finite element method
  for the numerical solution of the compressible {N}avier--{S}tokes equations,
  Journal of Computational Physics 131~(2) (1997) 267--279.

\bibitem{Coc98}
B.~Cockburn, C.-W. Shu, The {R}unge--{K}utta discontinuous {G}alerkin method
  for conservation laws {V}: multidimensional systems, Journal of Computational
  Physics 141~(2) (1998) 199--224.

\bibitem{Coc00}
B.~Cockburn, G.~Karniadakis, C.-W. Shu, The development of discontinuous
  {G}alerkin methods, in: Discontinuous {G}alerkin Methods, Springer, 2000, pp.
  3--50.

\bibitem{Wan13}
Z.~Wang, K.~Fidkowski, R.~Abgrall, F.~Bassi, D.~Caraeni, A.~Cary, H.~Deconinck,
  R.~Hartmann, K.~Hillewaert, H.~Huynh, N.~Kroll, G.~May, P.-O. Persson, B.~van
  Leer, M.~Visbal, High-order {CFD} methods: Current status and perspective,
  International Journal for Numerical Methods in Fluids (2013).
\newblock \href {https://doi.org/10.1002/fld.3767}
  {\path{doi:10.1002/fld.3767}}.

\bibitem{Abg88}
R.~Abgrall, Generalisation of the {R}oe scheme for the computation of mixture
  of perfect gases, La Recherche Aérospatiale 6 (1988) 31--43.

\bibitem{Kar94}
S.~Karni, Multicomponent flow calculations by a consistent primitive algorithm,
  Journal of Computational Physics 112~(1) (1994) 31 -- 43.
\newblock \href {https://doi.org/https://doi.org/10.1006/jcph.1994.1080}
  {\path{doi:https://doi.org/10.1006/jcph.1994.1080}}.

\bibitem{Abg96}
R.~Abgrall, How to prevent pressure oscillations in multicomponent flow
  calculations: A quasi conservative approach, Journal of Computational Physics
  125~(1) (1996) 150 -- 160.
\newblock \href {https://doi.org/https://doi.org/10.1006/jcph.1996.0085}
  {\path{doi:https://doi.org/10.1006/jcph.1996.0085}}.

\bibitem{Abg01}
R.~Abgrall, S.~Karni, Computations of compressible multifluids, Journal of
  Computational Physics 169~(2) (2001) 594 -- 623.
\newblock \href {https://doi.org/https://doi.org/10.1006/jcph.2000.6685}
  {\path{doi:https://doi.org/10.1006/jcph.2000.6685}}.

\bibitem{Bil11}
G.~Billet, J.~Ryan, A {R}unge–{K}utta discontinuous {G}alerkin approach to
  solve reactive flows: The hyperbolic operator, Journal of Computational
  Physics 230~(4) (2011) 1064 -- 1083.
\newblock \href {https://doi.org/https://doi.org/10.1016/j.jcp.2010.10.025}
  {\path{doi:https://doi.org/10.1016/j.jcp.2010.10.025}}.

\bibitem{Lv15}
Y.~Lv, M.~Ihme, Discontinuous {G}alerkin method for multicomponent chemically
  reacting flows and combustion, Journal of Computational Physics 270 (2014)
  105 -- 137.
\newblock \href {https://doi.org/https://doi.org/10.1016/j.jcp.2014.03.029}
  {\path{doi:https://doi.org/10.1016/j.jcp.2014.03.029}}.

\bibitem{Ban20}
K.~Bando, M.~Sekachev, M.~Ihme, Comparison of algorithms for simulating
  multi-component reacting flows using high-order discontinuous {G}alerkin
  methods (2020).
\newblock \href {https://doi.org/10.2514/6.2020-1751}
  {\path{doi:10.2514/6.2020-1751}}.

\bibitem{Joh20_2}
R.~F. Johnson, A.~D. Kercher, A conservative discontinuous {G}alerkin
  discretization for the chemically reacting {N}avier--{S}tokes equations,
  Journal of Computational Physics 423 (2020) 109826.
\newblock \href {https://doi.org/10.1016/j.jcp.2020.109826}
  {\path{doi:10.1016/j.jcp.2020.109826}}.

\bibitem{Dei03}
R.~Deiterding, Parallel adaptive simulation of multi-dimensional detonation
  structures, Dissertation. de, 2003.

\bibitem{Gou20}
A.~Gouasmi, K.~Duraisamy, S.~M. Murman, E.~Tadmor, A minimum entropy principle
  in the compressible multicomponent {E}uler equations, ESAIM: Mathematical
  Modelling and Numerical Analysis 54~(2) (2020) 373--389.

\bibitem{Lv15_2}
Y.~Lv, M.~Ihme, Entropy-bounded discontinuous {Galerkin scheme for Euler}
  equations, Journal of Computational Physics 295 (2015) 715--739.

\bibitem{Fri19}
L.~Friedrich, G.~Schn{\"u}cke, A.~R. Winters, D.~C. D.~R. Fern{\'a}ndez, G.~J.
  Gassner, M.~H. Carpenter, Entropy stable space--time discontinuous {G}alerkin
  schemes with summation-by-parts property for hyperbolic conservation laws,
  Journal of Scientific Computing 80~(1) (2019) 175--222.

\bibitem{Chi22_2}
E.~J. Ching, R.~F. Johnson, A.~D. Kercher, Positivity-preserving and
  entropy-bounded discontinuous {G}alerkin method for the chemically reacting,
  compressible {E}uler equations. {P}art {II}: {T}he multidimensional case,
  arXiv preprint arXiv:2211.16297~\url{https://arxiv.org/abs/2211.16297}
  (2022).

\bibitem{Zha10}
X.~Zhang, C.-W. Shu, On positivity-preserving high order discontinuous
  {Galerkin schemes for compressible Euler} equations on rectangular meshes,
  Journal of Computational Physics 229~(23) (2010) 8918--8934.

\bibitem{Zha12}
X.~Zhang, Y.~Xia, C.-W. Shu, Maximum-principle-satisfying and
  positivity-preserving high order discontinuous {G}alerkin schemes for
  conservation laws on triangular meshes, Journal of Scientific Computing
  50~(1) (2012) 29--62.

\bibitem{Jia18}
Y.~Jiang, H.~Liu, Invariant-region-preserving {DG methods for multi-dimensional
  hyperbolic conservation law systems, with an application to compressible
  E}uler equations, Journal of Computational Physics 373 (2018) 385--409.

\bibitem{Per06}
P.-O. Persson, J.~Peraire, Sub-cell shock capturing for discontinuous
  {G}alerkin methods, in: AIAA (Ed.), 44th AIAA Aerospace Sciences Meeting and
  Exhibit, 2006, {AIAA}-2006-112.
\newblock \href {https://doi.org/10.2514/6.2006-112}
  {\path{doi:10.2514/6.2006-112}}.

\bibitem{Bar10}
G.~Barter, D.~Darmofal, Shock capturing with {PDE}-based artificial viscosity
  for {DGFEM}: Part {I}. formulation, Journal of Computational Physics 229~(5)
  (2010) 1810--1827.

\bibitem{Chi19}
E.~Ching, Y.~Lv, P.~Gnoffo, M.~Barnhardt, M.~Ihme, Shock capturing for
  discontinuous {G}alerkin methods with application to predicting heat transfer
  in hypersonic flows, Journal of Computational Physics 376 (2019) 54--75.

\bibitem{Luo07}
H.~Luo, J.~Baum, R.~L{\"o}hner, A hermite weno-based limiter for discontinuous
  {G}alerkin method on unstructured grids, Journal of Computational Physics
  225~(1) (2007) 686--713.

\bibitem{Maz19}
A.~Mazaheri, C.-W. Shu, V.~Perrier, Bounded and compact weighted essentially
  nonoscillatory limiters for discontinuous {Galerkin schemes: T}riangular
  elements, Journal of Computational Physics 395 (2019) 461--488.

\bibitem{Coc89}
B.~Cockburn, C.-W. Shu, {TVB Runge-Kutta local projection discontinuous
  Galerkin finite element method for conservation laws. II. General} framework,
  Mathematics of Computation 52~(186) (1989) 411--435.

\bibitem{Coc89_2}
B.~Cockburn, S.-Y. Lin, C.-W. Shu, {TVB Runge-Kutta local projection
  discontinuous Galerkin finite element method for conservation laws III:
  One-dimensional} systems, Journal of Computational Physics 84~(1) (1989)
  90--113.

\bibitem{Kri07}
L.~Krivodonova, Limiters for high-order discontinuous {G}alerkin methods,
  Journal of Computational Physics 226~(1) (2007) 879--896.

\bibitem{Cor18}
A.~Corrigan, A.~Kercher, D.~Kessler, A moving discontinuous {G}alerkin finite
  element method for flows with interfaces, International Journal for Numerical
  Methods in Fluids 89~(9) (2019) 362--406.
\newblock \href {https://doi.org/10.1002/fld.4697}
  {\path{doi:10.1002/fld.4697}}.

\bibitem{Zah18}
M.~Zahr, P.-O. Persson, An optimization-based approach for high-order accurate
  discretization of conservation laws with discontinuous solutions, Journal of
  Computational Physics (2018).

\bibitem{Zah21}
M.~J. Zahr, J.~M. Powers, High-order resolution of multidimensional
  compressible reactive flow using implicit shock tracking, AIAA Journal 59~(1)
  (2021) 150--164.

\bibitem{Zha11}
X.~Zhang, C.-W. Shu, Positivity-preserving high order discontinuous {Galerkin
  schemes for compressible Eu}ler equations with source terms, Journal of
  Computational Physics 230~(4) (2011) 1238--1248.

\bibitem{Zha17}
X.~Zhang, On positivity-preserving high order discontinuous {Galerkin schemes
  for compressible Navier--Stokes} equations, Journal of Computational Physics
  328 (2017) 301--343.

\bibitem{Zha12_2}
X.~Zhang, C.-W. Shu, A minimum entropy principle of high order schemes for gas
  dynamics equations, Numerische Mathematik 121~(3) (2012) 545--563.

\bibitem{Lv17}
Y.~Lv, M.~Ihme, High-order discontinuous {G}alerkin method for applications to
  multicomponent and chemically reacting flows, Acta Mechanica Sinica 33~(3)
  (2017) 486--499.

\bibitem{Wu21}
K.~Wu, C.-W. Shu, Geometric quasilinearization framework for analysis and
  design of bound-preserving schemes, arXiv preprint
  arXiv:2111.04722~\url{https://arxiv.org/abs/2111.04722} (2021).

\bibitem{Gue19}
J.-L. Guermond, B.~Popov, I.~Tomas, Invariant domain preserving
  discretization-independent schemes and convex limiting for hyperbolic
  systems, Computer Methods in Applied Mechanics and Engineering 347 (2019)
  143--175.

\bibitem{Paz21}
W.~Pazner, Sparse invariant domain preserving discontinuous {G}alerkin methods
  with subcell convex limiting, Computer Methods in Applied Mechanics and
  Engineering 382 (2021) 113876.

\bibitem{Gut22}
J.~Guti{\'e}rrez-Jorquera, F.~Kummer, A fully coupled high-order {Discontinuous
  Galerkin method for diffusion flames in a low-Mach} number framework,
  International Journal for Numerical Methods in Fluids (2022).

\bibitem{May21}
G.~May, K.~Devesse, A.~Rangarajan, T.~Magin, A hybridized discontinuous
  {G}alerkin solver for high-speed compressible flow, Aerospace 8~(11) (2021)
  322.

\bibitem{Pap18}
A.~Papoutsakis, S.~S. Sazhin, S.~Begg, I.~Danaila, F.~Luddens, {An efficient
  {Adaptive Mesh Refinement (AMR) algorithm for the discontinuous Galerkin
  method: Applications} for the computation of compressible two-phase flows},
  Journal of Computational Physics 363 (2018) 399--427.

\bibitem{Wan12}
C.~Wang, X.~Zhang, C.-W. Shu, J.~Ning, Robust high order discontinuous
  {G}alerkin schemes for two-dimensional gaseous detonations, Journal of
  Computational Physics 231~(2) (2012) 653--665.

\bibitem{Du19}
J.~Du, Y.~Yang, Third-order conservative sign-preserving and
  steady-state-preserving time integrations and applications in stiff
  multispecies and multireaction detonations, Journal of Computational Physics
  395 (2019) 489--510.

\bibitem{Du19_2}
J.~Du, C.~Wang, C.~Qian, Y.~Yang, High-order bound-preserving discontinuous
  {G}alerkin methods for stiff multispecies detonation, SIAM Journal on
  Scientific Computing 41~(2) (2019) B250--B273.

\bibitem{Hua18}
J.~Huang, C.-W. Shu, Bound-preserving modified exponential {Runge--Kutta
  discontinuous Ga}lerkin methods for scalar hyperbolic equations with stiff
  source terms, Journal of Computational Physics 361 (2018) 111--135.

\bibitem{Hua19}
J.~Huang, C.-W. Shu, Positivity-preserving time discretizations for
  production--destruction equations with applications to non-equilibrium flows,
  Journal of Scientific Computing 78~(3) (2019) 1811--1839.

\bibitem{Hua19_2}
J.~Huang, W.~Zhao, C.-W. Shu, A third-order unconditionally
  positivity-preserving scheme for production--destruction equations with
  applications to non-equilibrium flows, Journal of Scientific Computing 79~(2)
  (2019) 1015--1056.

\bibitem{Pan21}
J.~Pan, Y.-Y. Chen, L.-S. Fan, Second-order unconditional positive preserving
  schemes for non-equilibrium reactive flows with mass and mole balance,
  Journal of Computational Physics (2021) 110477.

\bibitem{Gou20_2}
A.~Gouasmi, K.~Duraisamy, S.~M. Murman, Formulation of entropy-stable schemes
  for the multicomponent compressible {E}uler equations, Computer Methods in
  Applied Mechanics and Engineering 363 (2020) 112912.

\bibitem{Ren21}
F.~Renac, Entropy stable, robust and high-order {DGSEM} for the compressible
  multicomponent {E}uler equations, Journal of Computational Physics 445 (2021)
  110584.

\bibitem{Pey22}
A.~Peyvan, K.~Shukla, J.~Chan, G.~Karniadakis, High-order methods for
  hypersonic flows with strong shocks and real chemistry, arXiv preprint
  arXiv:2211.12635~\url{https://arxiv.org/abs/2211.12635} (2022).

\bibitem{Che17}
T.~Chen, C.-W. Shu, Entropy stable high order discontinuous {G}alerkin methods
  with suitable quadrature rules for hyperbolic conservation laws, Journal of
  Computational Physics 345 (2017) 427--461.

\bibitem{Cha19}
J.~Chan, L.~C. Wilcox, On discretely entropy stable weight-adjusted
  discontinuous {G}alerkin methods: curvilinear meshes, Journal of
  Computational Physics 378 (2019) 366--393.

\bibitem{Cha18_2}
J.~Chan, On discretely entropy conservative and entropy stable discontinuous
  {G}alerkin methods, Journal of Computational Physics 362 (2018) 346--374.

\bibitem{Cre18}
J.~Crean, J.~E. Hicken, D.~C. D.~R. Fern{\'a}ndez, D.~W. Zingg, M.~H.
  Carpenter, Entropy-stable summation-by-parts discretization of the {E}uler
  equations on general curved elements, Journal of Computational Physics 356
  (2018) 410--438.

\bibitem{Che20}
T.~Chen, C.-W. Shu, Review of entropy stable discontinuous {G}alerkin methods
  for systems of conservation laws on unstructured simplex meshes, CSIAM
  Transactions on Applied Mathematics 1~(1) (2020) 1--52.

\bibitem{Gas21}
G.~J. Gassner, A.~R. Winters, A novel robust strategy for discontinuous
  {Galerkin methods in computational fluid mechanics: Why? When? What? Where?},
  Frontiers in Physics (2021) 612.

\bibitem{Gio99}
V.~Giovangigli, Multicomponent flow modeling, Birkhauser, Boston, 1999.

\bibitem{Mcb93}
B.~J. McBride, S.~Gordon, M.~A. Reno, Coefficients for calculating
  thermodynamic and transport properties of individual species (1993).

\bibitem{Mcb02}
B.~J. McBride, M.~J. Zehe, S.~Gordon, {N}{A}{S}{A} {G}lenn coefficients for
  calculating thermodynamic properties of individual species (2002).

\bibitem{Kee96}
R.~J. Kee, F.~M. Rupley, E.~Meeks, J.~A. Miller, {CHEMKIN-III: A FORTRAN}
  chemical kinetics package for the analysis of gas-phase chemical and plasma
  kinetics, Tech. rep., Sandia National Labs., Livermore, CA (United States)
  (1996).

\bibitem{Lin22}
F.~A. Lindemann, S.~Arrhenius, I.~Langmuir, N.~Dhar, J.~Perrin, W.~M. Lewis,
  Discussion on “the radiation theory of chemical action”, Transactions of
  the Faraday Society 17 (1922) 598--606.

\bibitem{Gil83}
R.~Gilbert, K.~Luther, J.~Troe, Theory of thermal unimolecular reactions in the
  fall-off range. {II. W}eak collision rate constants, Berichte der
  Bunsengesellschaft f{\"u}r physikalische Chemie 87~(2) (1983) 169--177.

\bibitem{Har13}
R.~Hartmann, T.~Leicht, Higher order and adaptive {DG} methods for compressible
  flows, in: H.~Deconinck (Ed.), VKI LS 2014-03: 37\textsuperscript{th}
  Advanced VKI CFD Lecture Series: Recent developments in higher order methods
  and industrial application in aeronautics, Dec. 9-12, 2013, Von Karman
  Institute for Fluid Dynamics, Rhode Saint Gen{\`e}se, Belgium, 2014.

\bibitem{Str68}
G.~Strang, On the construction and comparison of difference schemes, SIAM
  Journal on Numerical Analysis 5~(3) (1968) 506--517.

\bibitem{Wu19}
H.~Wu, P.~C. Ma, M.~Ihme, Efficient time-stepping techniques for simulating
  turbulent reactive flows with stiff chemistry, Computer Physics
  Communications 243 (2019) 81--96.

\bibitem{Atk96}
H.~Atkins, C.~Shu, Quadrature-free implementation of discontinuous {G}alerkin
  methods for hyperbolic equations, {ICASE} {R}eport 96-51, 1996, Tech. rep.,
  NASA Langley Research Center, nASA-CR-201594 (August 1996).

\bibitem{Atk98}
H.~L. Atkins, C.-W. Shu, Quadrature-free implementation of discontinuous
  {G}alerkin method for hyperbolic equations, AIAA Journal 36~(5) (1998)
  775--782.

\bibitem{Har83_3}
A.~Harten, P.~D. Lax, B.~v. Leer, On upstream differencing and {G}odunov-type
  schemes for hyperbolic conservation laws, SIAM review 25~(1) (1983) 35--61.

\bibitem{Moc80}
M.~S. Mock, Systems of conservation laws of mixed type, Journal of Differential
  equations 37~(1) (1980) 70--88.

\bibitem{Tad86}
E.~Tadmor, A minimum entropy principle in the gas dynamics equations, Applied
  Numerical Mathematics 2~(3-5) (1986) 211--219.

\bibitem{Tad84}
E.~Tadmor, Skew-selfadjoint form for systems of conservation laws, Journal of
  Mathematical Analysis and Applications 103~(2) (1984) 428--442.

\bibitem{Cha10}
C.~Chalons, F.~Coquel, E.~Godlewski, P.-A. Raviart, N.~Seguin, Godunov-type
  schemes for hyperbolic systems with parameter-dependent source: the case of
  euler system with friction, Mathematical Models and Methods in Applied
  Sciences 20~(11) (2010) 2109--2166.

\bibitem{Bou04}
F.~Bouchut, Nonlinear Stability of Finite Volume Methods for Hyperbolic
  Conservation Laws: And Well-Balanced schemes for Sources, Springer Science \&
  Business Media, 2004.

\bibitem{Rug89}
T.~Ruggeri, Galilean invariance and entropy principle for systems of balance
  laws, Continuum Mechanics and Thermodynamics 1~(1) (1989) 3--20.

\bibitem{Wu21_2}
K.~Wu, Minimum principle on specific entropy and high-order accurate invariant
  region preserving numerical methods for relativistic hydrodynamics, arXiv
  preprint arXiv:2102.03801 (2021).

\bibitem{Gue16_2}
J.-L. Guermond, B.~Popov, Invariant domains and first-order continuous finite
  element approximation for hyperbolic systems, SIAM Journal on Numerical
  Analysis 54~(4) (2016) 2466--2489.

\bibitem{Per96}
B.~Perthame, C.-W. Shu, On positivity preserving finite volume schemes for
  {E}uler equations, Numerische Mathematik 73~(1) (1996) 119--130.

\bibitem{Lax71}
P.~Lax, Shock waves and entropy, in: Contributions to nonlinear functional
  analysis, Elsevier, 1971, pp. 603--634.

\bibitem{Tor13}
E.~Toro, Riemann solvers and numerical methods for fluid dynamics: {A}
  practical introduction, Springer Science \& Business Media, 2013.

\bibitem{Got01}
S.~Gottlieb, C.~Shu, E.~Tadmor, Strong stability-preserving high-order time
  discretization methods, SIAM review 43~(1) (2001) 89--112.

\bibitem{Spi02}
R.~Spiteri, S.~Ruuth, A new class of optimal high-order
  strong-stability-preserving time discretization methods, SIAM Journal on
  Numerical Analysis 40~(2) (2002) 469--491.

\bibitem{Gue16}
J.-L. Guermond, B.~Popov, Fast estimation from above of the maximum wave speed
  in the {Riemann problem for the E}uler equations, Journal of Computational
  Physics 321 (2016) 908--926.

\bibitem{Tor20}
E.~F. Toro, L.~O. M{\"u}ller, A.~Siviglia, Bounds for wave speeds in the
  {Riemann problem: D}irect theoretical estimates, Computers \& Fluids 209
  (2020) 104640.

\bibitem{Fro19}
R.~Frolov, An efficient algorithm for the multicomponent compressible
  {Navier--Stokes equations in low-and high-Mach} number regimes, Computers \&
  Fluids 178 (2019) 15--40.

\bibitem{Dza22}
T.~Dzanic, F.~D. Witherden, Positivity-preserving entropy-based adaptive
  filtering for discontinuous spectral element methods, arXiv preprint
  arXiv:2201.10502 (2022).

\bibitem{Tor02}
E.~Toro, V.~Titarev, Solution of the generalized {R}iemann problem for
  advection--reaction equations, Proceedings of the Royal Society of London.
  Series A: Mathematical, Physical and Engineering Sciences 458~(2018) (2002)
  271--281.

\bibitem{Mon12}
G.~Montecinos, C.~E. Castro, M.~Dumbser, E.~F. Toro, Comparison of solvers for
  the generalized riemann problem for hyperbolic systems with source terms,
  Journal of Computational Physics 231~(19) (2012) 6472--6494.

\bibitem{Bec01}
A.~Beccantini, Riemann solvers for multi-component gas mixtures with
  temperature dependent heat capacities, Ph.D. thesis, Université d'Évry
  Val-d'Essonne (2000).

\bibitem{Bec10}
A.~Beccantini, E.~Studer, The reactive {R}iemann problem for thermally perfect
  gases at all combustion regimes, International Journal for Numerical Methods
  in Fluids 64~(3) (2010) 269--313.

\bibitem{Hai96}
E.~Hairer, G.~Wanner, {S}olving {O}rdinary {D}ifferential {E}quations {II}.
  {S}tiff and {D}ifferential-{A}lgebraic {P}roblems, Vol.~14, 1996.
\newblock \href {https://doi.org/10.1007/978-3-662-09947-6}
  {\path{doi:10.1007/978-3-662-09947-6}}.

\bibitem{For11}
L.~Formaggia, A.~Scotti, Positivity and conservation properties of some
  integration schemes for mass action kinetics, SIAM journal on numerical
  analysis 49~(3) (2011) 1267--1288.

\bibitem{San01}
A.~Sandu, Positive numerical integration methods for chemical kinetic systems,
  Journal of Computational Physics 170~(2) (2001) 589--602.

\bibitem{Gas16}
G.~J. Gassner, A.~R. Winters, D.~A. Kopriva, Split form nodal discontinuous
  {G}alerkin schemes with summation-by-parts property for the compressible
  {E}uler equations, Journal of Computational Physics 327 (2016) 39--66.

\bibitem{Roe06}
P.~L. Roe, Affordable, entropy consistent flux functions, in: Eleventh
  International Conference on Hyperbolic Problems: Theory, Numerics and
  Applications, Lyon, 2006.

\bibitem{Ism09}
F.~Ismail, P.~L. Roe, Affordable, entropy-consistent euler flux functions {II}:
  {E}ntropy production at shocks, Journal of Computational Physics 228~(15)
  (2009) 5410--5436.

\bibitem{Sla11}
J.~C. Slattery, P.~G. Cizmas, A.~N. Karpetis, S.~B. Chambers, Role of
  differential entropy inequality in chemically reacting flows, Chemical
  engineering science 66~(21) (2011) 5236--5243.

\bibitem{Rea18}
A.~E. Ream, J.~C. Slattery, P.~G. Cizmas, A method for generating reduced-order
  combustion mechanisms that satisfy the differential entropy inequality,
  Physics of Fluids 30~(4) (2018) 043601.

\bibitem{Cor18_SCITECH}
A.~Corrigan, A.~Kercher, J.~Liu, K.~Kailasanath, Jet noise simulation using a
  higher-order discontinuous {G}alerkin method, in: 2018 AIAA SciTech Forum,
  2018, {A}IAA-2018-1247.

\bibitem{Hou11}
R.~Houim, K.~Kuo, A low-dissipation and time-accurate method for compressible
  multi-component flow with variable specific heat ratios, Journal of
  Computational Physics 230~(23) (2011) 8527 -- 8553.
\newblock \href {https://doi.org/https://doi.org/10.1016/j.jcp.2011.07.031}
  {\path{doi:https://doi.org/10.1016/j.jcp.2011.07.031}}.

\bibitem{sdtoolbox}
J.~E. Shepherd,
  \href{http://shepherd.caltech.edu/EDL/PublicResources/sdt/}{Explosion
  dynamics laboratory: Shock and detonation toolbox - 2018 version}, jES
  9-19-2018 (2018).
\newline\urlprefix\url{http://shepherd.caltech.edu/EDL/PublicResources/sdt/}

\bibitem{Wes82}
C.~K. Westbrook, Chemical kinetics of hydrocarbon oxidation in gaseous
  detonations, Combustion and Flame 46 (1982) 191 -- 210.
\newblock \href {https://doi.org/https://doi.org/10.1016/0010-2180(82)90015-3}
  {\path{doi:https://doi.org/10.1016/0010-2180(82)90015-3}}.

\bibitem{Bol15}
M.~Bolla, B.~Bullins, S.~Chaturapruek, S.~Chen, K.~Friedl, Spectral properties
  of modularity matrices, Linear Algebra and Its Applications 473 (2015)
  359--376.

\bibitem{Dub09}
P.~Dube, R.~Jain, Bertrand games between multi-class queues, in: Proceedings of
  the 48h IEEE Conference on Decision and Control (CDC) held jointly with 2009
  28th Chinese Control Conference, IEEE, 2009, pp. 8588--8593.

\end{thebibliography}

\appendix

\section{Supporting lemmas associated with specific entropy\label{sec:supporting-lemmas-specific-entropy}}

We restate two key results established by Zhang and Shu~\citep{Zha12_2}.
We assume throughout that $\rho>0$, $C_{i}>0$, and $T>0$.
\begin{lem}[\citep{Zha12_2}]
\label{lem:quasi-concavity-specific-entropy}$s(y)$ is quasi-concave,
such that
\begin{equation}
s\left(\beta y_{1}+(1-\beta)y_{2}\right)>\min\left\{ s(y_{1}),s(y_{2})\right\} ,
\end{equation}
where $y_{1}\neq y_{2}$ and $0<\beta<1$.
\end{lem}

\begin{proof}
Since $U=-\rho s$ is strictly convex~\citep{Gou20}, $U$ satisfies
Jensen's inequality:
\[
U\left(\beta y_{1}+(1-\beta)y_{2}\right)<\beta U(y_{1})+(1-\beta)U(y_{2}).
\]
Therefore,
\begin{align*}
-\rho\left(\beta y_{1}+(1-\beta)y_{2}\right)s\left(\beta y_{1}+(1-\beta)y_{2}\right) & <-\beta\rho(y_{1})s(y_{1})-(1-\beta)\rho(y_{2})s(y_{2})\\
 & \leq-\beta\rho(y_{1})\min\left\{ s(y_{1}),s(y_{2})\right\} -(1-\beta)\rho(y_{2})\min\left\{ s(y_{1}),s(y_{2})\right\} \\
 & =-\left[\beta\rho(y_{1})+(1-\beta)\rho(y_{2})\right]\min\left\{ s(y_{1}),s(y_{2})\right\} \\
 & =-\rho\left(\beta y_{1}+(1-\beta)y_{2}\right)\min\left\{ s(y_{1}),s(y_{2})\right\} ,
\end{align*}
where the last equality is due to linearity. We then have 
\[
s\left(\beta y_{1}+(1-\beta)y_{2}\right)>\min\left\{ s(y_{1}),s(y_{2})\right\} 
\]
\end{proof}
\begin{lem}[\citep{Zha12_2}]
\label{lem:specific-entropy-of-solution-average}For given $y_{\kappa}$
and $\overline{y}_{\kappa}$ as defined in Equations~(\ref{eq:solution-approximation})
and~(\ref{eq:solution-element-average}), respectively, we have
\begin{equation}
s\left(\overline{y}_{\kappa}\right)\geq\min_{x\in\kappa}s\left(y(x)\right).
\end{equation}
\end{lem}

\begin{proof}
Let $\overline{\rho}_{\kappa}=\frac{1}{|\kappa|}\int_{\kappa}\rho(y(x))dx$.
With $U=-\rho s$, we have
\begin{align*}
\overline{\rho}_{\kappa}s\left(\overline{y}_{\kappa}\right) & =-U\left(\overline{y}_{\kappa}\right)\\
 & \geq-\frac{1}{|\kappa|}\int_{\kappa}U(y(x))dx\\
 & =\frac{1}{|\kappa|}\int_{\kappa}\rho(y(x))s(y(x))dx\\
 & \geq\overline{\rho}_{\kappa}\min_{x\in\kappa}s\left(y(x)\right),
\end{align*}
where the second line is due to Jensen's inequality since $U$ is
convex. 
\end{proof}

\section{Concavity of shifted internal energy\label{sec:concavity-of-shifted-internal-energy}}
\begin{lem}
\label{lem:concavity-of-shifted-internal-energy}The ``shifted''
internal energy per unit volume, 
\begin{equation}
\rho u^{*}=\rho u-\sum_{i=1}^{n_{s}}\rho_{i}b_{i0},\label{eq:shifted-internal-energy-per-unit-volume}
\end{equation}
is a concave function of the state for $\rho>0$, $T>0$.
\end{lem}

\begin{proof}
Throughout this proof, we work with $d=3$ and, without loss of generality,
a re-ordered state vector where the species concentrations are replaced
by the partial densities:
\[
y=\left(\rho v_{1},\ldots,\rho v_{d},\rho_{1},\ldots,\rho_{n_{s}},\rho e_{t}\right)^{T}.
\]
It is well-known that a function is concave if and only if its Hessian
is negative semidefinite. Here, $\mathcal{H}=\frac{d^{2}\left(\rho u^{*}\right)}{dy^{2}}$
denotes the Hessian of $\rho u^{*}$, which is a symmetric matrix
of size $m$. We observe that since the second term on the RHS of
Equation~(\ref{eq:shifted-internal-energy-per-unit-volume}) is linear
with respect to the state, $\frac{d^{2}}{dy^{2}}\left(\sum_{i=1}^{n_{s}}\rho_{i}b_{i0}\right)$
gives a matrix of zeros; thus, $\mathcal{H}=\frac{d^{2}\left(\rho u\right)}{dy^{2}}$.
There exist various ways to show negative semidefiniteness. One approach
is to check the signs of the principal minors~\citep{Bol15,Dub09},
where an $l$th-order principal minor, $\mathcal{M}_{l}$, of $\mathcal{H}$
is the determinant of a submatrix obtained by eliminating $m-l$ rows
and the corresponding $m-l$ columns from $\mathcal{H}$. Specifically,
$\mathcal{H}$ is negative semidefinite if and only if all even-order
principal minors are nonnegative and all odd-order principal minors
are nonpositive. 

We start with some useful relations:
\begin{align*}
 & \frac{\partial^{2}\left(\rho u\right)}{\partial\left(\rho v_{k}\right)\partial\left(\rho v_{k}\right)}=-\frac{1}{\rho},\\
 & \frac{\partial^{2}\left(\rho u\right)}{\partial\left(\rho v_{k}\right)\left(\partial\rho v_{l}\right)}=0,\quad k\neq l,\\
 & \frac{\partial^{2}\left(\rho u\right)}{\partial\rho_{i}\partial\rho_{j}}=-\frac{\left|v\right|^{2}}{\rho},\\
 & \frac{\partial^{2}\left(\rho u\right)}{\partial\rho_{i}\partial\left(\rho v_{k}\right)}=\frac{v_{k}}{\rho},\\
 & \frac{\partial^{2}\left(\rho u\right)}{\partial\rho_{i}\partial\left(\rho e_{t}\right)}=\frac{\partial^{2}\left(\rho u\right)}{\partial\left(\rho v_{k}\right)\partial\left(\rho e_{t}\right)}=\frac{\partial^{2}\left(\rho u\right)}{\partial\left(\rho e_{t}\right)\partial\left(\rho e_{t}\right)}=0.
\end{align*}
$\mathcal{H}$ can then be written as
\begin{equation}
\mathfrak{\mathcal{H}}=\left(\begin{array}{ccccccc}
-\frac{1}{\rho} & 0 & 0 & \frac{v_{1}}{\rho} & \ldots & \frac{v_{1}}{\rho} & 0\\
0 & -\frac{1}{\rho} & 0 & \frac{v_{2}}{\rho} & \ldots & \frac{v_{2}}{\rho} & \vdots\\
0 & 0 & -\frac{1}{\rho} & \frac{v_{3}}{\rho} & \ldots & \frac{v_{3}}{\rho} & 0\\
\frac{v_{1}}{\rho} & \frac{v_{2}}{\rho} & \frac{v_{3}}{\rho} & -\frac{\left|v\right|^{2}}{\rho} & \ldots & -\frac{\left|v\right|^{2}}{\rho} & 0\\
\vdots & \vdots & \vdots & \vdots & \ddots & \vdots & \vdots\\
\frac{v_{1}}{\rho} & \frac{v_{2}}{\rho} & \frac{v_{3}}{\rho} & -\frac{\left|v\right|^{2}}{\rho} & \ldots & -\frac{\left|v\right|^{2}}{\rho} & 0\\
0 & \ldots & 0 & 0 & \ldots & 0 & 0
\end{array}\right).\label{eq:Hessian-internal-energy}
\end{equation}
The $m$ first-order principal minors, $\mathcal{M}_{1}$, are simply
the diagonal entries, which are all nonpositive. The principal minors
of order greater than one take the following forms:
\begin{align}
\mathcal{M}_{l}^{0} & =\det\left(\begin{array}{cccc}
a_{1,1} & \ldots & a_{1,l-1} & 0\\
\vdots & \ddots & \vdots & \vdots\\
a_{l-1,1} & \ldots & a_{l-1,l-1} & \vdots\\
0 & \ldots & \ldots & 0
\end{array}\right),\;l=2,\ldots,m,\label{eq:principal-minor-specific-form-0}\\
\mathcal{M}_{l}^{1} & =\det\left(\begin{array}{ccc}
-\frac{1}{\rho} & 0 & 0\\
0 & \ddots & 0\\
0 & 0 & -\frac{1}{\rho}
\end{array}\right),\;l=2,\ldots,d,\label{eq:principal-minor-specific-form-1}\\
\mathcal{M}_{l}^{2} & =\det\left(B_{l}\right),B_{l}=\left(\begin{array}{ccc}
-\frac{|v|^{2}}{\rho} & \ldots & -\frac{|v|^{2}}{\rho}\\
\vdots & \ddots & \vdots\\
-\frac{|v|^{2}}{\rho} & \ldots & -\frac{|v|^{2}}{\rho}
\end{array}\right),\;l=2,\ldots,n_{s},\label{eq:principal-minor-specific-form-2}\\
\mathcal{M}_{l}^{3} & =\det\left(C_{l}\right)=\det\left(\begin{array}{c|c}
C_{1,1} & C_{1,2}\\
\hline C_{2,1} & B_{q}
\end{array}\right)\nonumber \\
 & =\det\left(\begin{array}{c|c}
C_{1,1} & C_{1,2}\\
\hline C_{2,1} & \begin{array}{ccc}
-\frac{|v|^{2}}{\rho} & \ldots & -\frac{|v|^{2}}{\rho}\\
\vdots & \ddots & \vdots\\
-\frac{|v|^{2}}{\rho} & \ldots & -\frac{|v|^{2}}{\rho}
\end{array}
\end{array}\right),\;q=1,\ldots,n_{s};\;l=q+1,\ldots,q+d.\label{eq:principal-minor-specific-form-3}
\end{align}
$\mathcal{M}_{l}^{0}$ is zero due to the row of zeros (note that
the $m$th-order principal minor, $\mathcal{M}_{m}=\det\left(\mathcal{H}\right)$,
takes this form); $\mathcal{M}_{l}^{1}$ is negative for $l$ odd
and positive for $l$ even; and $\mathcal{M}_{l}^{2}$ is zero due
to linear dependence of the rows. For $\mathcal{M}_{l}^{3}$, the
corresponding submatrix, $C_{l}$, is written in block-matrix form,
where the lower-right block is a matrix of size $q$ of the form in
Equation~(\ref{eq:principal-minor-specific-form-2}). We consider
two cases: $q=1$ and $q\geq1$. In the latter case, $\mathcal{M}_{l}^{3}=0$
since the last $q$ rows of $C_{l}$ are repeated. In the former,
we first consider $l=2$, which yields 
\[
\mathcal{M}_{2}^{3}=\det\left(\begin{array}{cc}
-\frac{1}{\rho} & \frac{v_{k}}{\rho}\\
\frac{v_{k}}{\rho} & -\frac{\left|v\right|^{2}}{\rho}
\end{array}\right)=\frac{\left|v\right|^{2}}{\rho^{2}}-\frac{v_{k}^{2}}{\rho^{2}},\;k=1,\ldots,d,
\]
which is nonnegative since $\left|v\right|^{2}=\sum_{i}v_{i}^{2}\geq v_{k}^{2}\geq0$.
For $l=3$, we have
\[
\mathcal{M}_{4}^{3}=\det\left(\begin{array}{ccc}
-\frac{1}{\rho} & 0 & \frac{v_{j}}{\rho}\\
0 & -\frac{1}{\rho} & \frac{v_{k}}{\rho}\\
\frac{v_{j}}{\rho} & \frac{v_{k}}{\rho} & -\frac{\left|v\right|^{2}}{\rho}
\end{array}\right)=-\frac{\left|v\right|^{2}}{\rho^{3}}+\frac{v_{j}^{2}}{\rho^{3}}+\frac{v_{k}^{2}}{\rho^{3}},\;j,k=1,\ldots,d;\;j\neq k.
\]
which is nonpositive. Finally, $l=4$ gives
\[
\mathcal{M}_{4}^{3}=\det\left(\begin{array}{cccc}
-\frac{1}{\rho} & 0 & 0 & \frac{v_{1}}{\rho}\\
0 & -\frac{1}{\rho} & 0 & \frac{v_{2}}{\rho}\\
0 & 0 & -\frac{1}{\rho} & \frac{v_{3}}{\rho}\\
\frac{v_{1}}{\rho} & \frac{v_{2}}{\rho} & \frac{v_{3}}{\rho} & -\frac{\left|v\right|^{2}}{\rho}
\end{array}\right)=\left(\begin{array}{ccc}
\horzbar & \mathsf{R}_{1} & \horzbar\\
\horzbar & \mathsf{R}_{2} & \horzbar\\
\horzbar & \mathsf{R}_{3} & \horzbar\\
\horzbar & \mathsf{R}_{4} & \horzbar
\end{array}\right),
\]
where $\left\{ \mathsf{R}_{1},\mathsf{R}_{2},\mathsf{R}_{3},\mathsf{R}_{4}\right\} $
denotes the rows. We observe that $\mathsf{R}_{4}=-v_{1}\mathsf{R}_{1}-v_{2}\mathsf{R}_{2}-v_{3}\mathsf{R}_{3}$;
therefore, the rows are linearly dependent and $\mathcal{M}_{4}^{3}=0$.
As such, the principal minors satisfy the nonpositive/nonnegative
requirements for negative semidefiniteness of $\mathcal{H}$.
\end{proof}

\end{document}